\documentclass[12pt,oneside,reqno,bibliography=totoc]{scrartcl}

\usepackage{cihan}

\AtBeginDocument{%
  \addtolength\abovedisplayskip{-0.5\baselineskip}%
  \addtolength\belowdisplayskip{-0.5\baselineskip}%
}

\title{Improved stability ranges in the homology of Torelli and congruence subgroups}
\author{Cihan Bahran}
\affil{Department of Mathematics, Bilkent University\\
06800 Ankara, Turkey}\date{}
\date{}

\newcommand{\FI}{\mathbf{FI}}

\newcommand{\OI}{\mathbf{OI}}

\DeclareMathOperator{\sr}{\mathbf{sr}}
\DeclareMathOperator{\glr}{\mathbf{glr}}

\DeclareMathOperator{\rk}{rk}
\DeclareMathOperator{\rank}{rank}

\DeclareMathOperator{\VIC}{\mathbf{VIC}}
\DeclareMathOperator{\VI}{\mathbf{VI}}
\DeclareMathOperator{\VIS}{\mathbf{VIS}}
\DeclareMathOperator{\SI}{\mathbf{SI}}
\DeclareMathOperator{\SU}{SU}
\DeclareMathOperator{\PB}{PB}
\DeclareMathOperator{\PBC}{PBC}

\DeclareMathOperator{\UuU}{U}
\DeclareMathOperator{\HU}{HU}

\DeclareMathOperator{\Link}{Link}

\DeclareMathOperator{\Homeo}{Homeo}

\DeclareMathOperator{\Face}{Face}
\DeclareMathOperator{\Pos}{Pos}

\DeclareMathOperator{\Mod}{Mod}

\DeclareMathOperator{\FB}{\mathbf{FB}}

\newcommand{\ia}{\mathtt{IA}}
\newcommand{\auf}{\mathtt{AF}}

\newcommand{\cauf}{\mathcal{AF}}
\newcommand{\cGL}{\mathcal{GL}}
\newcommand{\cSL}{\mathcal{SL}}
\newcommand{\cMod}{\mathcal{MOD}}
\newcommand{\ctor}{\mathcal{I}}
\newcommand{\cSp}{\mathcal{SP}}
\newcommand{\cia}{\mathcal{IA}}

\newcommand{\torel}{\mathtt{I}}
\newcommand{\hth}[1]{\mathbf{H3}(#1)}

\newcounter{dummy}
\makeatletter
\newcommand\sitem[1][]{\item[(#1)]\refstepcounter{dummy}\def\@currentlabel{#1}}
\makeatother

\newcommand{\cofi}[1]{\co_{#1}^{\FI}}

\DeclareMathOperator{\ce}{\mathsf{C}}
\DeclareMathOperator{\de}{\mathsf{D}}
\DeclareMathOperator{\je}{\mathsf{J}}
\DeclareMathOperator{\me}{\mathsf{M}}

\DeclareMathOperator{\ab}{\mathcal{A}}

\DeclareMathOperator{\cat}{\mathsf{Cat}}

\DeclareMathOperator{\reg}{reg}

\DeclareMathOperator{\induce}{\Ind_{\FB}^{\FI}}

\DeclareMathOperator{\kk}{\mathbb{F}}

\newcommand{\atsi}{\mathbf{A}}

\newcommand{\fU}{\mathfrak{U}}

\newcommand{\shift}[2]{\mathbf{\Sigma}^{#2}   #1 }

\newcommand{\weak}{\delta}

\newcommand{\deriv}{\mathbf{\Delta}}
\newcommand{\kert}{\mathbf{K}}

\newcommand{\bul}{\bullet}

\DeclareMathOperator{\Ch}{Ch}

\newcommand{\locoh}[1]{\co_{\mathfrak{m}}^{#1}}

\DeclareMathOperator{\lkan}{Lan}

\newcommand{\desu}{\mathbf{\Omega}}
\newcommand{\aus}{\widetilde{\operatorname{ss}}}
\newcommand{\ses}{\operatorname{ss}}

\newcommand{\manif}{\mathcal{M}}
\newcommand{\surf}{\mathcal{S}}

\AtBeginDocument{%
   \def\MR#1{}
}

\newcommand{\scx}{\mathsf{Scx}}
\newcommand{\CW}{\mathsf{CW}}

\newcommand{\moncat}{\mathsf{MonCat}}
\newcommand{\bmcat}{\mathsf{BrMonCat}}

\newcommand{\mongpd}{\mathsf{MonGpd}}
\newcommand{\bmgpd}{\mathsf{BrMonGpd}}

\makeatletter
\def\blfootnote{\gdef\@thefnmark{}\@footnotetext}
\makeatother

\begin{document}
\maketitle

\blfootnote{\textup{2010} \textit{Mathematics Subject Classification}.
Primary 20J06; Secondary 18A25, 11F75, 55U10.} 
\blfootnote{\textit{Key words and phrases}. Torelli groups, congruence subgroups, central stability.}
\blfootnote{The author was supported in part by T\"{U}B\.{I}TAK 119F422.}
\vspace{-0.8 in}
\begin{onecolabstract} 
We improve the central stability ranges for 
\begin{itemize}
\item $\co_{2}\!\big(\text{Torelli subgroup of $\Aut(F_{n})$'s}\big)$ as $\GL_{n}(\zz)$-representations,
 \item \vspace{0.1cm}$\co_{2}\!\big(\text{Torelli subgroup of mapping class groups of $\surf_{g}^{1}$'s}\big)$ as $\Sp_{2g}(\zz)$-representations,
 \item \vspace{0.07cm}$\co_{k}\!\big(\text{$I$-congruence subgroups of $\GL_{n}(R)$'s}\big)$ as $\SL_{n}^{\fU}(R/I)$-representations for a commutative ring $R$ of finite Bass stable rank with a proper ideal $I$.
\end{itemize} 
\end{onecolabstract}
\vspace{0.4 in}

{\tableofcontents}

\section{Introduction}
The minimal amount of structure present in most \textbf{homological stability} phenomena can be captured by the category
\[
 \atsi \colon \quad 0 \to 1 \to 2 \to \cdots
\]
More precisely, $\Obj(\atsi) = \zz_{\geq 0}$ and $\atsi$ has a unique morphism $(m,n) \colon m \to n$ whenever $m \leq n$ and no morphism $m \to n$ if $m > n$. A functor 
\begin{itemize}
 \item $X \colon \atsi \to \ce$ is called an \textbf{$\atsi$-object} in $\ce$,
 \item $G \colon \atsi \to \Grp$ is called an \textbf{$\atsi$-group},
 \item $V \colon \atsi \to \lMod{\zz}$ an \textbf{$\atsi$-module}, etc.
\end{itemize}
Maps between $\atsi$-objects are defined as natural transformations in the appropriate functor categories. 
For an $\atsi$-object $X$, we write $X_{n}$ for the evaluation of $X$ at $n \in \zz_{\geq 0}$ and $X_{(m,n)} \colon X_{m} \to X_{n}$ for the induced morphism when $m \leq n$. We employ similar terminology for functors out of categories other than $\atsi$ as well. 




In \textbf{representation stability} phenomena (as conceived in \cite{church-farb-rep-stab}) one has more structure: an $\atsi$-group G together with an $\atsi$-module on which $G$ acts in a compatible way. The minimal amount of structure here is captured by the following category:

\begin{defn} \label{grothendieck-const}
 Given an $\atsi$-group $G$, we write $\atsi_{G}$ for the Grothendieck construction \cite[Definition 1.1]{thomason-hocolim} $\int_{\atsi} G$ of the composite functor
 \[
 \atsi \xrightarrow{G} \Grp \emb \Cat \, ,
 \]
via considering every group as a category with a single object. More concretely, the category $\atsi_{G}$ has the following description:
\begin{itemize}
 \item $\Obj(\atsi_{G}) = \zz_{\geq 0}$.
 \item $\Hom_{\atsi_{G}}(m,n) = 
\begin{cases}
 \{(\gamma,m,n) : \gamma \in G_{n}\} & \text{if $m \leq n$,}
 \\
 \empt & \text{if $m > n$.}
\end{cases}$
 \item For $(\gamma,m,n) \colon m \to n$ and $(\delta,n,p) \colon n \to p$ in $\atsi_{G}$, 
 \[
  (\delta,n,p) \circ (\gamma,m,n) = \left(
  \delta \cdot G_{(n,p)}(\gamma),m,p\right) \, .
 \]
\end{itemize}
\end{defn}

\begin{rem}
 Note that for an $\atsi$-group $G$, we have $\Aut_{\atsi_{G}}(n) \cong G_{n}$. Moreover, the datum of an $\atsi_{G}$-module $V \colon \atsi_{G} \to \lMod{\zz}$ is precisely that of a \textbf{consistent sequence} $\{V_{n}\}$ of $G_{n}$-representations in the sense of \cite[Section 2.3]{church-farb-rep-stab} minus the finiteness requirement.
\end{rem}

\begin{ex} \label{symmetric-A}
 The assignment $n \mapsto \sym{n}$ where $\sym{n}$ is the symmetric group on $\{1,\dots,n\}$, together with the standard inclusions, defines an $\atsi$-group $\mathfrak{S}$. This is perhaps the most thoroughly studied $\atsi$-group in representation stability. 
\end{ex}

Our main results improve the stable ranges in three instances of representation stability: \textbf{Theorem \ref{H2-IA}} in Section \ref{torelli-AF}, \textbf{Theorem \ref{H2-torelli}} in Section \ref{torelli-MCG}, \textbf{Theorem \ref{congruence-larger-action}} in Section \ref{section:cong-GL}. Each instance occurs in the group  homology of the kernel of an $\atsi$-group homomorphism (compare with \cite[page 256]{church-farb-rep-stab}) in the following way:

\begin{prop} \label{conj-action}
 Given a map $\pi \colon G \to Q$ of $\atsi$-groups, the $\atsi$-subgroup $K \coloneqq \ker \pi$  of $G$ extends to an $\atsi_{G}$-group via the assignments
\begin{align*}
 K_{(\gamma,m,n)} \colon K_{m} &\to K_{n}
 \\
 x &\mapsto \gamma \cdot G_{(m,n)}(x) \cdot \gamma^{-1}
\end{align*}
 for each $(\gamma,m,n) \colon m \to n$ in $\atsi_{G}$.
\end{prop}

\begin{prop} \label{factor-AQ}
 Let $\pi \colon G \to Q$ be a map of $\atsi$-groups for which $\pi_{n} \colon G_{n} \to Q_{n}$ is surjective for every $n \in \zz_{\geq 0}$. Regarding $K \coloneqq \ker \pi$ as an $\atsi_{G}$-group as in Proposition \ref{conj-action}, for every $k \in \zz_{\geq 0}$, writing\footnote{Throughout this paper, $\co_{k}(\Gamma) \coloneqq \co_{k}(\Gamma;\zz)$ denotes the $k$-th \textbf{integral} homology of a group $\Gamma$.} 
\[
 \co_{k} \colon \Grp \to \lMod{\zz}
\]
for the $k$-th group homology functor, there is a (unique) functor $\atsi_{Q} \to \lMod{\zz}$ that makes the diagram
\begin{align*}
 \xymatrix{
 \atsi_{G} \ar[r]^-{K} \ar[d]_-{\atsi_{\pi}} & \Grp \ar[r]^-{\co_{k}} & \lMod{\zz}
 \\
 \atsi_{Q} \ar@{-->}_-{\exists !}[urr]
 }
\end{align*}
commute; in other words, the $\atsi_{G}$-module $\co_{k}(K)$ (uniquely) descends to an $\atsi_{Q}$-module.
\end{prop}
\begin{proof}
 This is simply because for every $n \in \zz_{\geq 0}$ the conjugation action of $G_{n}$ on $K_{n}$ induces an action of $G_{n}/K_{n} \cong Q_{n}$ (via $\pi$) on $\co_{k}(K_{n})$ \cite[Corollary II.6.3]{brown-coh-book}.
\end{proof}

\subsection{Types of stabilizations in representation stability}
The vast majority of homological stability results are stated with two stable ranges: one after which the maps become surjective (a form of \textbf{finite generation}), and one after which the maps become isomorphisms (a form of \textbf{finite presentation}). 

Fix an $\atsi$-group $G$. A fairly natural form of finite generation for $\atsi_{G}$-modules is as follows: whenever $m \leq n$, let us write \begin{align*}
 \Xi \colon \Hom_{G_{m}}\!\left(
 -,\Res_{G_{m}}^{G_{n}}(-)
 \right)
 \rarr
 \Hom_{G_{n}}\!\left(
 \Ind_{G_{m}}^{G_{n}}(-),-
 \right)
\end{align*}
for the natural isomorphism given by the induction-restriction adjunction associated to the map $G_{(m,n)} \colon G_{m} \to G_{n}$ of groups. 
\begin{defn} \label{surj-stab}
 Let $G$ be an $\atsi$-group and $d \in \zz_{\geq 0}$. An $\atsi_{G}$-module $V$ has \textbf{surjective stability degree} $\leq d$ (cf. \cite[Definition 2.3, part II]{church-farb-rep-stab}) if the $G_{n}$-equivariant map 
\begin{align} \label{nat1}
 \Xi V_{\pars*{1_{G_{n}},\,n-1,\,n}} \colon \Ind_{G_{n-1}}^{G_{n}} V_{n-1} \rarr V_{n}
\end{align}
is surjective for $n > d$. More categorically, this is equivalent to the property that the only $\atsi_{G}$-submodule $W \leq V$ with $W_{i} = V_{i}$ for $0 \leq i \leq d$ is $W = V$ itself. \end{defn}

Although there are other possible (arguably more natural) choices, the form of finite presentation we shall treat for $\atsi_{G}$-modules require some further input. First note that for $n \geq 2$, replacing $n$ with $n-1$ in (\ref{nat1}), we have a $G_{n-1}$-equivariant map
\[
 \Xi V_{\pars*{1_{G_{n-1}},\,n-2,\,n-1}} \colon \Ind_{G_{n-2}}^{G_{n-1}} V_{n-2} \rarr V_{n-1} \, ,
\]
so applying the functor $\Ind_{G_{n-1}}^{G_{n}}$, we get a $G_{n}$-equivariant map 
\begin{align} \label{nat2}
 \Ind_{G_{n-1}}^{G_{n}}\!\pars*{\Xi V_{\pars*{1_{G_{n-1}},\,n-2,\,n-1}}} \colon \Ind_{G_{n-2}}^{G_{n}}V_{n-2} \rarr \Ind_{G_{n-1}}^{G_{n}} V_{n-1} \, .
\end{align}
 
\begin{defn} \label{tail-central}
 Let $G$ be an $\atsi$-group. A sequence $(c_{n} : n \geq 2)$ of elements with $c_{n} \in G_{n}$ for every $n$ is called \textbf{tail-central} in $G$ if $c_{n}$ commutes with the image of $G_{n-2}$ inside $G_{n}$ for every $n \geq 2$.
\end{defn}

\begin{ex} \label{symm-tc}
The sequence
\[
 n \,\mapsto\, (n-1,n) \in \sym{n}
\]
of transpositions defines a tail-central sequence the $\atsi$-group $\mathfrak{S}$ of Example \ref{symmetric-A}.
\end{ex}

\begin{conv} \label{ring-conv}
 Every ring $R$ in this paper is assumed to be \textbf{nonzero}, \textbf{associative}, \textbf{unital}, and (most non-standardly) \textbf{stably finite} \cite[Section 1B]{lam-lectures}. The stably finite condition means that for every $n \in \zz_{\geq 0}$ and every pair of $n \times n$ matrices $A,B$ over $R$, we have
\[
 AB = I_{n} \,\imp\, BA = I_{n} \, .
\]
For example, commutative rings \cite[Proposition 1.12]{lam-lectures} and one-sided Noetherian rings \cite[Proposition 1.13]{lam-lectures} are stably finite.
\end{conv}

\begin{ex} \label{A-grp-GL}
 For every ring $R$, the assignment $n \mapsto \GL_{n}(R)$, together with upper-left block inclusions, defines an $\atsi$-group $\GL(R)$. For every $n \geq 2$, let $c_{n} \in \GL_{n}(R)$ be the matrix which is obtained by swapping the last two rows of the $n \times n$ identity matrix. This way $(c_{n})$ is a tail-central sequence for $\GL(R)$.
\end{ex}

Given a tail-central sequence $(c_{n})$ in the $\atsi$-group $G$, for each $n \geq 2$ we can twist the map (\ref{nat2}) by $c_{n}$ to get another $G_{n}$-equivariant map
\begin{align} \label{nat3}
 \left[\Ind_{G_{n-1}}^{G_{n}}\! \pars*{\Xi V_{\pars*{1_{G_{n-1}},\,n-2,\,n-1}}}\right]^{c_{n}} \colon \Ind_{G_{n-2}}^{G_{n}}V_{n-2} \rarr \Ind_{G_{n-1}}^{G_{n}} V_{n-1} 
\end{align}
as follows:
\begin{itemize}
 \item Apply $\Xi^{-1}$ to (\ref{nat2}) to get a $G_{n-2}$-equivariant map $V_{n-2} \rarr \Ind_{G_{n-1}}^{G_{n}} V_{n-1}$.
 \item Noting that the ``multiplication by $c_{n}$'' map on a $G_{n}$-module is $G_{n-2}$-equivariant (by the tail-central condition), consider the composite $G_{n-2}$-equivariant map
\[
 V_{n-2} \to \Ind_{G_{n-1}}^{G_{n}} V_{n-1} \xrightarrow{c_{n}} \Ind_{G_{n-1}}^{G_{n}} V_{n-1} \, .
\]
 \item Apply $\Xi$ back to the above $G_{n-2}$-equivariant map.
\end{itemize}
This way, the coequalizer of (\ref{nat2}) and (\ref{nat3}) is the largest quotient of $\Ind_{G_{n-1}}^{G_{n}} V_{n-1}$ that $c_{n}$ acts trivially on the image of $V_{n-2}$ .
\begin{defn} \label{defn:central}
 Let $G$ be an $\atsi$-group equipped with a tail-central sequence $(c_{n})$ and $r \in \zz_{\geq 1}$. An $\atsi_{G}$-module $V$ has \textbf{central stability degree} $\leq r$ if the maps (\ref{nat1}), (\ref{nat2}), (\ref{nat3}) form a coequalizer diagram
\begin{align*}
\mathlarger{
 \Ind_{G_{n-2}}^{G_{n}} V_{n-2} \rightrightarrows   
 \Ind_{G_{n-1}}^{G_{n}} V_{n-1}
 \rarr
 V_{n}
 }
\end{align*}
for $n > r$. We often refer to (\ref{nat1}), (\ref{nat2}), (\ref{nat3}) as the \textbf{natural maps} where the tail-central sequence will have been fixed in advance.
\end{defn}

\begin{rem}
 Central stability is a stronger condition than surjective stability since the coequalizer of two parallel morphisms is always an epimorphism.
\end{rem}

\subsection{Torelli subgroups of automorphism groups of free groups} \label{torelli-AF}
 Writing $F_{n}$ for the free group on $\{x_{1},\dots, x_{n}\}$, the assignment $n \mapsto \Aut(F_{n})$, together with the standard inclusions, defines an $\atsi$-group $\auf$.

The abelianization functor induces a surjective map
\begin{align*}
 \Aut(F_{n}) \rarr \Aut(\zz^{n}) \cong \GL_{n}(\zz)\, ,
\end{align*}
of groups, which patch to a surjective map $\auf \to \GL(\zz)$ of $\atsi$-groups: we write $\ia$ for the kernel. For each $n \in \zz_{\geq 0}$, $\ia_{n}$ is often called the \textbf{Torelli subgroup} of $\Aut(F_{n})$.

\begin{thm}[{\cite[Theorem B]{day-putman-L-pres}}] The $\atsi_{\GL(\zz)}$-module $\co_{2}(\ia)$ has surjective stability degree $\leq 6$, that is, the natural map
\begin{align*}
\mathlarger{
 \Ind_{\GL_{n-1}(\zz)}^{\GL_{n}(\zz)} \co_{2}(\ia_{n-1})
 \rarr
 \co_{2}(\ia_{n})
 }
\end{align*} 
of $\GL_{n}(\zz)$-modules is surjective whenever $n > 6$.
\end{thm}

\begin{thmro} \label{H2-IA}
 The $\atsi_{\GL(\zz)}$-module $\co_{2}(\ia)$ has central stability degree $\leq 9$, that is, the natural maps
\begin{align*}
\mathlarger{
 \Ind_{\GL_{n-2}(\zz)}^{\GL_{n}(\zz)} \co_{2}(\ia_{n-2}) \rightrightarrows   
 \Ind_{\GL_{n-1}(\zz)}^{\GL_{n}(\zz)} \co_{2}(\ia_{n-1})
 \rarr
 \co_{2}(\ia_{n})
 }
\end{align*}
of $\GL_{n}(\zz)$-modules form a coequalizer diagram whenever $n > 9$.
\end{thmro}

\begin{rem}
 The $\atsi_{\GL(\zz)}$-module $\co_{2}(\ia)$ was shown to have central stability degree $\leq 38$ in \cite[Theorem A]{mpw-torelli-H2}. Theorem \ref{H2-IA} improves this.
\end{rem}

\subsection{Torelli subgroups of mapping class groups} \label{torelli-MCG}
\paragraph{Intersection form in general.} Let $m \in \zz_{\geq 1}$ and $\manif$ be a closed $(m-1)$-connected oriented $2m$-manifold with fundamental class $[\manif] \in \co_{2m}(\manif)$. The bilinear form 
\begin{align*}
 \co^{m}(\manif) \times \co^{m}(\manif)
 &\to \co_{0}(\manif) = \zz
 \\
 (\alpha, \beta) &\mapsto (\alpha \cup \beta) \cap [\manif]
\end{align*}
\begin{itemize}
 \item is $(-1)^{m}$-symmetric by the graded commutativity of cup product,
 \vspace{0.1 cm}
 \item is invariant under the diagonal right action $(\alpha,\beta) \cdot f \coloneqq (f^{*}\alpha,f^{*}\beta)$ of the topological group (with the compact-open topology \cite[Theorem 1.26]{kawakubo-transf})
\[
 \Homeo^{+}(\manif) \coloneqq 
 \left\{f \colon \manif \to \manif : \!\!\! \text{
\begin{tabular}{l}
 $f$ is an orientation-preserving
 \\
 homeomorphism
\end{tabular}
}
  \!\!\!\right\}
\] 
by properties of the cup and cap products,
\vspace{0.1cm}  
 \item is perfect by Poincar{\'{e}} duality and the universal coefficient theorem,
\vspace{0.1cm}  
 \item has a dual $(-1)^{m}$-symmetric perfect \textbf{intersection form}
 \[
  q_{\manif} \colon 
  \co_{m}(\manif) \times \co_{m}(\manif)
 \to \zz
 \]  
 which
\begin{itemize}
 \item is invariant under the diagonal left action $f \cdot (y,z) \coloneqq (f_{*}y, f_{*}z$) of $\Homeo^{+}(\manif)$,
 \item realizes the intersection numbers\footnote{In the surface case $m=1$, these are often called the \emph{algebraic} intersection numbers \cite[Section 1.2.3]{fm-mcg} of transverse oriented simple closed curves on the surface.} \cite[Definition 10.31]{davkir} of transverse closed oriented $m$-submanifolds when $\manif$ is smooth \cite[Theorem 10.32]{davkir}.
\end{itemize}
\end{itemize}

Thanks to the left action of $\Homeo^{+}(\manif)$, the (discrete) group
\[
 \Aut\!\big(\!
 \co_{m}(\manif), q_{\manif}
 \big) \coloneqq \left\{
 \vphi \colon \co_{m}(\manif) \to \co_{m}(\manif) :\!\!
\begin{tabular}{l}
 $\vphi$ is a group isomorphism with
 \\
 $\forall{y,\!z} \,\, q_{\manif}(\vphi(y),\vphi(z)) = q_{\manif}(y,z)$
\end{tabular}
 \!\!\right\}
\]
receives a homomorphism
\( \LL_{\manif} \colon \Homeo^{+}(\manif) \to
 \Aut\!\big(\!
 \co_{m}(\manif), q_{\manif}
 \big) \).

\paragraph{Surfaces and varying the genus.} 
In case $m=1$ and $\manif = \surf_{g}$ is a closed connected oriented surface of genus $g \in \zz_{\geq 0}$, we would like to to vary $g$ in the above setup. To that end, it is more convenient to work with 
\begin{itemize}
 \item the compact surface $\surf_{g}^{1}$ obtained by removing the interior of a closed disk in $\surf_{g}$ (so the boundary $\partial \surf_{g}^{1}$ is a circle), and
 \item the topological group (with the compact-open topology)
\[
 \Homeo_{\partial}^{+}\pars*{\surf_{g}^{1}} \coloneqq 
 \left\{f \colon \surf_{g}^{1} \to \surf_{g}^{1} : \!\!\! \text{
\begin{tabular}{l}
 $f$ is an orientation-preserving
 \\
 homeomorphism with $f|_{\partial \surf_{g}^{1}} = \id_{\partial\surf_{g}^{1}}$
\end{tabular}
}
  \!\!\!\right\} \, .
\]
\end{itemize}

Since $\Homeo_{\partial}^{+}(\surf_{g}^{1})$ embeds in $\Homeo^{+}(\surf_{g})$ via capping the missing disk, we have a map
\begin{align*}
 \LL_{g} \colon \Homeo^{+}_{\partial}\!\left(
 \surf_{g}^{1}\right) &\to
 \Aut\!\big(\!
 \co_{1}(\surf_{g}), q_{\surf_{g}} \big)
 \cong \Sp_{2g}(\zz)
\end{align*}
where the isomorphism with the symplectic group is via picking an appropriate basis for $\co_{1}(\surf_{g})$. That $\LL_{g}$ is surjective for every $g \geq 0$ is often attributed to \cite{burkhardt-surj}; alternatively, the proof for $\Homeo^{+}\!\left(\surf_{g}\right)$ in \cite[Theorem 2]{meeks-patrusky-circles} can be adapted to work for $\Homeo_{\partial}^{+}\!\left(
  \surf_{g}^{1}\right)$.
  
For every commutative ring $R$, the assignment $g \mapsto \Sp_{2g}(R)$, together with upper-left block inclusions, defines an $\atsi$-group $\Sp(R)$. For every $g \geq 2$, let $c_{g}$ be the $2g \times 2g$ (symplectic permutation) matrix obtained by first swapping the $(2g-3)$-rd row with the $(2g-1)$-st row, and then swapping $(2g-2)$-nd row with the $(2g)$-th row of the identity $2g \times 2g$ matrix. This way $(c_{g})$ is a tail-central sequence in $\Sp(R)$.
  
Via the \textbf{boundary connected sum} operation $\natural$ as described in \cite[Section 5.6]{rw-wahl-stab}, we can identify $\surf_{g+1}^{1} = \surf_{g}^{1} \natural \surf_{1}^{1}$ and get an embedding
\begin{align*}
  \Homeo_{\partial}^{+}\!\left(
  \surf_{g}^{1}\right) &\emb
  \Homeo_{\partial}^{+}\!\left(
  \surf_{g+1}^{1}\right)
  \\
  f &\mapsto f \natural \id
\end{align*}
of topological groups. Putting these together for all $g \geq 0$ encodes a topological $\atsi$-group 
\[
\Homeo_{\partial}^{+}\!\left(
  \surf^{1}\right) \colon \atsi \to \Top\Grp
\]
such that $\LL_{g}$'s patch into a (surjective) map $\LL \colon \Homeo_{\partial}^{+}\!\left(
  \surf^{1}\right) \to \Sp(\zz)$ of $\atsi$-groups. As $\LL$ is continuous with a discrete target, it descends to a (still surjective) map
\[
 \lambda \colon \Mod\pars*{\surf^{1}} \coloneqq 
 \pi_{0}\!\left(
 \Homeo^{+}_{\partial}\!\left(
 \surf^{1}\right)
 \right) \to
 \Sp(\zz)
\]
of $\atsi$-groups: we write $\torel^{1} \coloneqq \ker \lambda$. Here $\Mod(\surf_{g}^{1})$ is  the \textbf{mapping class group} of $\surf_{g}^{1}$, where $\torel_{g}^{1} \coloneqq \ker \lambda_{g}$ is its so-called \textbf{Torelli subgroup}.

\begin{thmro} \label{H2-torelli}  The $\atsi_{\Sp(\zz)}$-module $\co_{2}\!\left(\torel^{1}\right)$ has
\begin{itemize}
 \item surjective stability degree $\leq 7$, that is, the natural map
\begin{align*}
 \mathlarger{
 \Ind_{\Sp_{2g-2}(\zz)}^{\Sp_{2g}(\zz)} \co_{2}\!\left(
 \torel_{g-1}^{1}\right)
 \rarr
 \co_{2}\!\left(\torel_{g}^{1}\right)
 }
\end{align*}
of $\Sp_{2g}(\zz)$-modules is surjective whenever $g > 7$, and
 \vspace{0.1cm}
 \item central stability degree $\leq 9$, that is, the natural maps  
\begin{align*}
\mathlarger{
 \Ind_{\Sp_{2g-4}(\zz)}^{\Sp_{2g}(\zz)} \co_{2}\!\left(
 \torel_{g-2}^{1} \right) \rightrightarrows   
 \Ind_{\Sp_{2g-2}(\zz)}^{\Sp_{2g}(\zz)} \co_{2}\!\left(
 \torel_{g-1}^{1}
 \right)
 \rarr
 \co_{2}\!\left(\torel_{g}^{1}\right)
 }
\end{align*}
of $\Sp_{2g}(\zz)$-modules form a coequalizer diagram whenever $g > 9$.
\end{itemize}
\end{thmro}

\begin{rem}
The $\atsi_{\Sp(\zz)}$-module $\co_{2}(\torel^{1})$ was shown to have surjective stability degree $\leq 33$ in \cite[Theorem B']{mpw-torelli-H2}, and central stability degree $\leq 69$ in \cite[Theorem B]{mpw-torelli-H2}. Theorem \ref{H2-torelli} improves these.
\end{rem}

\begin{rem}
The rationalized $\atsi_{\Sp(\zz)}$-module $\co_{2}(\torel^{1}) \otimes \qq$ has surjective stability degree $\leq 6$ by \cite[Theorem 1.0.1]{boldsen-dollerup}. Its explicit evaluations for $g \geq 6$ have been recently computed in \cite[Theorem A]{putman-minahan-H2-torelli}.
\end{rem}

\subsection{Congruence subgroups of general linear groups} \label{section:cong-GL}
For every ring $R$ (with Convention \ref{ring-conv}) with a proper ideal $I$, the mod-$I$ reductions define a morphism  
\[
\rho \colon \GL(R) \to \GL(R/I)
\]
of $\atsi$-groups for which we write $\GL(R,I) \coloneqq \ker\rho$. The subgroup $\GL_{n}(R,I)$ of $\GL_{n}(R)$ is often called the $I$\textbf{-congruence subgroup} of $\GL_{n}(R)$.

One issue here is that $\rho$ is in general not surjective. Yet through permutation matrices, the image $\im(\rho)$ contains a copy of $\mathfrak{S}$ (Example \ref{symmetric-A}) in $\GL(R/I)$ so that the $\atsi$-module $\co_{k}(\GL(R,I))$ extends to an $\atsi_{\mathfrak{S}}$-module for every $k \geq 0$. The central stability of these $\atsi_{\mathfrak{S}}$-modules (with respect to the tail-central sequence of Example \ref{symm-tc}) under fairly general conditions was first established in \cite[Theorem B]{putman-congruence}. Later results such as \cite[Theorem 1.6]{cefn}, \cite[Theorem D']{ce-homology}, \cite[Application B]{cmnr-range}, \cite[Theorem 1]{gan-li-linear}, \cite[Theorem H]{bahran-reg} widened its scope and/or improved the central stability degrees involved. An important hypothesis in these is a bound on the (Bass) stable rank of the ring $R$,\footnote{However, note that there are finiteness results such as \cite[Th\'{e}or\`{e}me 2]{djament-congruence-stab} on these representations without any assumption on the stable rank of $R$.} which we briefly recall.
\paragraph{Stable rank.} A column vector $\mathbf{v} \in \mat_{m \times 1}(R)$ of size $m$ is \textbf{unimodular} if there is a row vector $\mathbf{u} \in \mat_{1 \times m}(R)$ such that $\mathbf{u}\mathbf{v} = 1$. Writing $\mathbf{I}_{r} \in \mat_{r \times r}(R)$ for the identity matrix of size $r$, we say
a column vector $\mathbf{v}$ of size $m$ is \textbf{reducible} if there exists  $\mathbf{B} \in \mat_{(m-1) \times m}(R)$ with block form  $\mathbf{B} = [\mathbf{I}_{m-1} \,|\, \mathbf{x}]$ such that the column vector $\mathbf{B} \mathbf{v}$ (of size $m-1$) is unimodular. We write $\sr(R) \leq s$ if every unimodular column vector of size $> s$ is reducible.

\begin{thm} \label{congr-sym-ranges}
 Let $I$ be a proper ideal in a ring $R$ (under Convention \ref{ring-conv}) such that $\sr(R) \leq s$. Then for every $k \geq 1$, as an $\atsi_{\mathfrak{S}}$-module $\co_{k}(\GL(R,I))$  has
\begin{itemize}
 \item surjective stability degree
 $
 \leq
 \begin{cases}
 s+1 & \text{if $k=1$,}
 \\
 2s+5 & \text{if $k=2$,}
 \\
 4k+2s-2 & \text{if $k\geq 3$,}
\end{cases}
$
\vspace{0,1cm}
 \item and central stability degree 
$ \leq
\begin{cases}
 s+3 & \text{if $k=1$,}
 \\
 2s+6 & \text{if $k=2$,}
 \\
 4k+2s-1 & \text{if $k \geq 3$.}
\end{cases}
$
\end{itemize}
\end{thm}
\begin{proof}
 This follows from \cite[Theorem 4.15]{bahran-reg} and \cite[Proposition 2.4]{cmnr-range}.
\end{proof}


\paragraph{Special linear group with respect to a subgroup of the unit group.} For a \textbf{commutative} ring $R$ and a subgroup $\fU \leq R^{\times}$, we write\footnote{We shall use the notation in \cite{putman-sam}, whereas in \cite{mpw-torelli-H2} and \cite{mpp-secondary} this group is denoted $\GL_{n}^{\fU}(R)$.}
\begin{align*}
 \SL_{n}^{\fU}(R) := \{f \in \GL_{n}(R) : \det(f) \in \fU\} \, ,
\end{align*}
so that as $\atsi$-groups, we interpolate between $\SL(R) \leq \SL^{\fU}(R) \leq \GL(R)$  as we vary $1 \leq \fU \leq R^{\times}$. Note that the tail-central sequence of permutation matrices in $\GL(R)$ is contained in $\SL^{\fU}(R)$ provided that $-1 \in \fU$.

\begin{rem} \label{mod-I-surj}
 Given a commutative ring $R$ with an ideal $I$, it is straightforward to check that if the mod-$I$ reduction $\SL(R) \rarr \SL(R/I)$ between the special linear $\atsi$-groups is surjective, then setting 
\begin{align*}
  \fU := \{x + I : x \in R^{\times}\} \, ,
\end{align*}
the image of $\GL(R) \to \GL(R/I)$ equals $\SL^{\fU}(R/I)$.
\end{rem}

\begin{rem}
 Assuming the mod-$I$ reduction $\SL(R) \rarr \SL(R/I)$ is surjective as in the situation of Remark \ref{mod-I-surj}, writing $L := \SL^{\fU}(R/I)$, the surjective stability degree of $\co_{k}(\GL(R,I))$ as an $\atsi_{\mathfrak{S}}$-module is an upper bound for the surjective stability degree of $\co_{k}(\GL(R,I))$ as an $\atsi_L$-module. This is because the natural map
\begin{align*}
 \Ind_{L_{n-1}}^{L_{n}} \co_{k}(\GL_{n-1}(R,I))
 \rarr
 \co_{k}(\GL_{n}(R,I))
\end{align*}  
factors through the natural map
\begin{align*}
 \Ind_{\sym{n-1}}^{\sym{n}} \co_{k}(\GL_{n-1}(R,I))
 \rarr
 \co_{k}(\GL_{n}(R,I)) \, .
\end{align*}
The central stability degrees on the other hand are not as directly related.  
\end{rem}

\begin{thmro} \label{congruence-larger-action}
 Let $I$ be a proper ideal in a commutative ring $R$ and $s \geq 1$ such that 
\begin{itemize}
 \item $
 \sr(R) \leq s  
$, and
 \item the mod-$I$ reduction  
$\SL_{n}(R) \rarr \SL_{n}(R/I)$ is surjective for every $n \geq 0$. 
\end{itemize}
Then writing 
\begin{align*}
 \fU := \{x + I : x \in R^{\times}\} \leq (R/I)^{\times}
 \quad \text{and} \quad
 L_{n} := \SL_{n}^{\fU}(R/I) \, ,
\end{align*}
for every $k \geq 1$ the $\atsi_{L}$-module $\co_{k}(\GL(R,I))$ has central stability degree
\begin{align*}
 \leq C(k,s) :=
\begin{cases}
 s+5 & \text{if $k=1$,}
 \\
 2s+6 & \text{if $k=2$,}
 \\
 4k+2s-1 & \text{if $k \geq 3$,}
\end{cases}
\end{align*}
that is, the natural maps  
\begin{align*}
\mathlarger{
 \Ind_{L_{n-2}}^{L_{n}} \co_{k}(\GL_{n-2}(R,I)) \rightrightarrows   
 \Ind_{L_{n-1}}^{L_{n}} \co_{k}(\GL_{n-1}(R,I))
 \rarr
 \co_{k}(\GL_{n}(R,I))
 }
\end{align*}
of $L_{n}$-modules form a coequalizer diagram whenever $n > C(k,s)$.
\end{thmro}

\begin{rem}
Under the assumptions and notation of Theorem \ref{congruence-larger-action}, the best bound for the central stability degree previously established in the literature was
\begin{align*}
\leq
\begin{cases}
 2s + 4 & \text{if $k =1$,}
 \\
 4k + 2s + 1 & \text{if $k \geq 2$,}
\end{cases}
\end{align*}  
from \cite[Theorem D]{bahran-polynomial}. Theorem \ref{congruence-larger-action} slightly improves this.
\end{rem}

\subsection{Some pointers}
The main technical result in this paper is Theorem \ref{outsource}. The title of Section \ref{section:mpw-method} is an apt summary for what Theorem \ref{outsource} \textbf{does}: under certain assumptions, polynomiality in degree $< k$ homology implies central stability at degree $k$ homology. 

On one hand, the proof of Theorem \ref{outsource} is ``just'' obtained by combining material already present \cite{mpw-torelli-H2} and \cite{mpp-secondary}: we run the spectral sequence argument of \cite{mpw-torelli-H2} where instead of the exponential vanishing ranges in \cite[Theorem 3.26]{mpw-torelli-H2}, we use the linear vanishing ranges in \cite[Theorem 3.11]{mpp-secondary}. 

On the other hand, the background to even rigorously state Theorem \ref{outsource} is substantial. While a previous version of this paper made several precise inline references to the above papers and \cite{patzt-central} in its proofs to avoid reintroducing some of this background, the current version has several sections and appendices for the involved concepts (hence the explosion in the number of pages). Although the applications are still confined to the ``stability groupoid'' framework originating from \cite[Section 3]{patzt-central}, some of the concepts are introduced in a more general setup than this, which I find clarifying and could be of separate interest. For example, the simple notion of a category $\ce$ equipped with a shifting context (Definition \ref{defn:shifting-context}) is sufficient to afford both polynomial behavior (Definition \ref{poly-defn}) and a ``stability'' $\ce$-homology theory (Section \ref{sec-omega} and Definition \ref{C-homology}) for $\ce$-modules.

Setting the ``braided lifts'' aside, the main hypotheses of Theorem \ref{outsource} are the $\mathbf{H3}$ acyclicity property (Definition \ref{H3-defn}) and polynomiality (Definition \ref{poly-defn}). To see how exactly these hypotheses are invoked for our main results, one can look at the bullet points in Section \ref{main-proofs}.

\section{Categorical notions} \label{categorical}
\begin{conv}
 Throughout this section, $\ce, \de$ are categories such that 
\begin{itemize}
 \item $\ce$ has a small skeleton so that the isomorphism classes of objects form a set $\Iso(\ce)$,
 \item $\de$ has a zero object (and hence zero morphisms).
\end{itemize}
We write $[\ce,\de]$ for the category of functors $\ce \to \de$.
\end{conv}

\paragraph{Degree.}  Given a function $\rk \colon \Iso(\ce) \to \zz_{\geq 0}$, for every functor $V \colon \ce \to \de$ we write
\begin{align*}
 \deg V &\coloneqq \min\!\big\{
 N \in \zz_{\geq - 1} \cup \{\infty\}\,:\, V_{x} = 0 \text{ for every $x \in \Obj(\ce)$ with} \rk(x) > N
 \big\} \, ,
\end{align*}
called the \textbf{degree} of $V$.  

Before moving on with similarly general definitions, let us note that degrees behave as one might expect in a spectral sequence.

\begin{lem} \label{sseq-deg}
 Suppose $\de$ is an abelian category and $\rk \colon \Iso(\ce) \to \zz_{\geq 0}$ is a function so that degrees of functors $\ce \to \de$ are defined as above. If 
 \[ E_{2}^{p,q} \imp V^{p+q} \]  is a cohomologically graded spectral sequence in $[\ce,\de]$ and $p_{0} \in \zz$ such that
 \[
  E^{2}_{p,q} = 0 \text{ when } p < p_{0} \text{ or } q<0 \, ,
 \]
then for every $k \in \zz_{\geq 0}$ we have
\begin{align*}
 \deg E^{2}_{p_{0},k}
 &\leq \max\!\left(\{\deg E^{\infty}_{p_{0},k}\} \cup \left\{
 \deg E^{2}_{p_{0} + k + 1 - q,q} : 0\leq q < k
 \right\} 
 \right)
\end{align*}
and
\begin{align*}
 \deg E^{2}_{p_{0}+1,k}
 &\leq \max\!\left(\{\deg E^{\infty}_{p_{0},k}\} \cup \left\{
 \deg E^{2}_{p_{0} + k + 2 - q,q} : 0\leq q < k
 \right\} 
 \right) \, .
\end{align*}
\end{lem}
\begin{proof}
 For every $V \colon \ce \to \de$, let us write
\[
 \supp(V) \coloneqq \kume*{x \in \Iso(\ce) : V_{x} \neq 0} \subseteq \Iso(\ce) \, .
\] 
To prove the first inequality, let  
\(d \coloneqq \deg E^{\infty}_{p_{0},k}\).
Fix $x \in \supp\pars*{E^{2}_{p_{0},k}}$ with $\rk(x) > d$.
By the definition of degree, we have
\(
 \pars*{E^{\infty}_{p_{0},k}}_{\!x} = 0 \, .
\)
Therefore there exists a page $r \in \zz_{\geq 2}$ such that
\[
 \pars*{E^{r}_{p_{0},k}}_{\!x} \neq 0 \quad \text{and} \quad 
 \pars*{E^{r+1}_{p_{0},k}}_{\!x} = 0 \, .
\]
Since $0 = \pars*{E^{r+1}_{p_{0},k}}_{\!x}$ is the homology of
\[
 0 = \pars*{E^{r}_{p_{0}-r,k-1+r}}_{\!x} \to
 \pars*{E^{r}_{p_{0},k}}_{\!x} \to
 \pars*{E^{r}_{p_{0} + r,k+1-r}}_{\!x} \, ,
\]
we get
\[
 \pars*{E^{r}_{p_{0} + r,k+1-r}}_{\!x} \neq 0 \quad \text{so that} 
 \quad x \in \supp\pars*{E^{r}_{p_{0} + r,k+1-r}} \, .
\]
Note that this $r$ has to satisfy $k+1-r \geq 0$, hence
\[
 x \in \bigcup_{r = 2}^{k+1} \supp\pars*{E^{r}_{p_{0} + r,k+1-r}}
 = \bigcup_{q=0}^{k-1} \supp\pars*{E^{r}_{p_{0} + k+1-q,q}} \, .
\]
As $x$ was arbitrary, we get
\[
 \kume*{x \in \supp\pars*{E^{2}_{p_{0},k}} : \rk(x) > d} \subseteq 
 \bigcup_{q=0}^{k-1} \supp\pars*{E^{r}_{p_{0} + k+1-q,q}} \, .
\]
Writing $d_{q} \coloneqq \deg E^{r}_{p_{0} + k+1-q,q}$ for every $q \in \{0,\dots, k-1\}$, if
\[
 x \in \Obj(\ce) \quad \text{with} \rk(x) > \max\{d_{0}, \dots ,d_{k-1}\} \, ,
\]
then $\pars*{E^{r}_{p_{0} + k+1-q,q}}_{\!x} = 0$ for every $q \in \{0,\dots, k-1\}$, that is,
\[
 x \in \Iso(\ce) - \bigcup_{q=0}^{k-1} \supp\pars*{E^{r}_{p_{0} + k+1-q,q}} \, .
\]
Hence 
\[
 x \in \Iso(\ce) - \kume*{x \in \supp\pars*{E^{2}_{p_{0},k}} : \rk(x) > d} \, ,
\]
that is,
\[
 \pars*{E^{2}_{p_{0},k}}_{\!x} = 0 \quad \text{or} \quad \rk(x) \leq d \, .
\]
Consequently, whenever $\rk(x) > \max\{d, d_{0}, \dots ,d_{k-1}\}$, we have $\pars*{E^{2}_{p_{0},k}}_{\!x} = 0$. By the definition of degree, we get
\[
 \deg E^{2}_{p_{0},k} \leq \max\{d, d_{0}, \dots ,d_{k-1}\}
\]
as desired. The proof of the second inequality is similar.
\end{proof}

\begin{defn} \label{defn:shifting-context}
 A \textbf{shifting context} on $\ce$ is a pair $(\sigma, \eta)$ where 
\begin{itemize}
 \item $\sigma \colon \ce \to \ce$ is a functor,
 \item $\eta \colon \id_{\ce} \to \sigma$ is a natural transformation.
\end{itemize}
\end{defn}


\paragraph{The shift functor.} If $\ce$ is a category equipped with a shifting context $(\sigma, \eta)$, we write 
\begin{align*}
   \shift{}{} \colon [\ce, \de] &\to [\ce, \de]
   \\
   V &\mapsto V \circ \sigma
\end{align*} 
for the functor defined by precomposing with $\sigma$ and call it the \textbf{shift} functor. For every $V \in [\ce, \de]$, the natural transformation $\eta$ induces a morphism
\[
 \eta V \colon V \to \shift{V}{}
\]
in $[\ce, \de]$.

\paragraph{The kernel and derivative functors.} For every $V \in [\ce,\de]$, we write $\kert V$ and $\deriv V$ for the equalizer and coequalizer (if they exist) of the parallel morphisms
\[
 \eta V,0 \colon V \to \shift{V}{}
\] 
respectively. In other words $\kert V = \ker \eta V$ and $\deriv V = \coker \eta V$.

\begin{defn} \label{poly-defn}
 Let $\ce$ be a category with a small skeleton equipped with a function $\rk \colon \Iso(\ce) \to \zz_{\geq 0}$ and a shifting context $(\sigma, \eta)$. Let $\de$ be a complete and cocomplete category with a zero object so that we have a well-defined notion of degree and functors
\[
 \shift{}{}, \kert, \deriv \colon [\ce,\de] \to [\ce,\de]
\] 
via the above. Given 
 \(V \colon \ce \to \de\)
and integers $r \geq -1$, $M \geq 0$, we say $V$ is \textbf{polynomial of degree $r$ starting at $M$} if it satisfies the following recursion:
\begin{itemize}
 \item If $r = -1$, then $\deg V < M$.
 \item If $r \geq 0$, the following hold:
\begin{itemize}
 \item $\kert V$ is polynomial of degree $-1$ starting at $M$.
 \item $\deriv V$ is polynomial of degree $r-1$ starting at $\max\{0, M-1\}$. 
\end{itemize}
\end{itemize}
\end{defn}

\begin{defn} \label{endo-to-A-grp}
 Given a functor $\sigma \colon \ce \to \ce$ (not necessarily part of a shifting context) and $c \in \Obj(\ce)$, we write $\ce(c,\sigma)$ for the $\atsi$-group defined as 
 \[
  \ce(c,\sigma)_{n} \coloneqq \Aut_{\ce}(\sigma^{n}c)
 \]
with the transition maps
\begin{align*}
  \ce(c,\sigma)_{(n,n+1)} \colon \Aut_{\ce}(\sigma^{n}c) &\to \Aut_{\ce}(\sigma^{n+1}c) 
  \\
  f &\mapsto \sigma f
\end{align*}
 for every $n \in \zz_{\geq 0}$.
\end{defn}

\subsection{Monoidal notions}
 We denote a monoidal category additively as $(\me,\oplus,0)$ where $\me$ is the underlying category, $\oplus \colon \me \times \me \to \me$ is the bifunctor and $0$ the unit object (whereas most sources such as \cite[Definition 2.1.1]{tensor-cat-book} use the multiplicative $\otimes$ notation).
 
\begin{conv} \label{tekman}
 Although most of the actual monoidal categories we shall deal with, such as $\OI$ (Section \ref{semi-simp}), $\FI$ (Appendix \ref{poly-FI}), $\VIC(R)$ (Definition \ref{su-pbc-defn}), $\SI(R)$ (Definition \ref{hu-defn}), are not \textbf{strict} as in \cite[Definition 2.8.1]{tensor-cat-book}, thanks to Mac Lane's strictness theorem \cite[Definition 2.8.5]{tensor-cat-book} we freely assume an arbitrary $(\me,\oplus,0)$ to be strict.
\end{conv}
 
\begin{defn}[{\cite[Definition 2.4.3]{tensor-cat-book}} simplified via Convention \ref{tekman}] \label{monoidal-F}
 Let $(\me,\oplus,0)$ and $(\me^{\wr}, \oplus^{\wr}, 0^{\wr})$ be monoidal categories. A \textbf{monoidal functor} from $\me$ to $\me^{\wr}$ is a triple $(F,J,\psi)$ where $F \colon \me \to \me^{\wr}$ is a functor and $(J,\psi)$ (the coherence data) is the datum of
\begin{itemize}
 \item a collection 
 \(
  \{J_{X,Y} \colon F(X) \oplus^{\wr} F(Y) \to F(X \oplus Y)\}
 \)
 of isomorphisms natural in $X,Y$,
 \item an isomorphism $\psi \colon F(0) \to 0^{\wr}$,
\end{itemize}
subject to the commutativity of the diagram
 \[
 \xymatrixcolsep{2.4cm}
 \xymatrixrowsep{1.4cm} 
 \xymatrix{
 F(X) \oplus^{\wr} F(Y) \oplus^{\wr} F(Z) 
 \ar[r]^-{J_{X,Y} \,\oplus^{\wr}\, \id_{F(Z)}} 
 \ar[d]_-{\id_{F(X)} \oplus^{\wr} J_{Y,Z}} 
 &
 F(X \oplus Y) \oplus^{\wr} F(Z) 
 \ar[d]^-{J_{X \oplus Y, Z}}  
 \\
 F(X) \oplus^{\wr} F(Y \oplus Z) 
 \ar[r]^-{J_{X,Y \oplus Z}}  
 & F(X \oplus Y \oplus Z)
 }
 \]
 and the equations
\begin{align*}
 \psi \oplus^{\wr} \id_{F(X)} &= J_{0,X} \, ,
 \\
 \id_{F(X)} \oplus^{\wr}\, \psi &= J_{X,0}
\end{align*}
for every $X,Y,Z \in \Obj(\me)$.
\end{defn}


\begin{rem} \label{moncat-comp}
 Given monoidal categories $(\me,\oplus,0)$, $(\me^{\wr}, \oplus^{\wr}, 0^{\wr})$, $(\me^{\square}, \oplus^{\square}, 0^{\square})$ and monoidal functors
\[
 (\me,\oplus,0) \xrightarrow{(F,J,\psi)} (\me^{\wr},\oplus^{\wr},0^{\wr}) \xrightarrow{(H,L,\eta)} (\me^{\square},\oplus^{\square},0^{\square}) \, ,
\]
we compose them as
\[
 (H,L,\eta) \circ (F,J,\psi) \coloneqq (H\circ F,W,\theta)
\]
where $W_{X,Y}$ is the composite
\[
 H(F(X)) \oplus^{\square} H(F(Y)) \xrightarrow{L_{F(X),F(Y)}} 
 H\pars*{F(X) \oplus^{\wr} F(Y)} \xrightarrow{H\pars*{J_{X,Y}}} H(F(X \oplus Y))
\]
for every $X,Y \in \Obj(\me)$, and $\theta$ is the composite $H(F(0)) \xrightarrow{H(\psi)} H(0^{\wr}) \xrightarrow{\eta} 0^{\square}$. Together with the evident identity morphisms, they make the monoidal categories and monoidal functors a category for which we write $\moncat$.
\end{rem}

\begin{defn} \label{defn:mon-opp}
 Given a monoidal category $\me = (\me, \oplus, 0)$, we write $\me^{\circ} \coloneqq (\me, \oplus^{\circ}, 0)$ for the monoidal category with the same underlying category but with the monoidal operation reversed, that is, $A \oplus^{\circ}\! B$ in $\me^{\circ}$ is $B \oplus A$ in $\me$.
\end{defn}

\begin{rem} \label{op-functor}
 Given  monoidal categories $(\me,\oplus,0)$, $(\me^{\wr}, \oplus^{\wr}, 0^{\wr})$ and a monoidal functor $\mathbf{F} = (F,J,\psi) \colon \me \to \me^{\wr}$, we can write
 \(
  \mathbf{F}^{\circ} \coloneqq (F,J^{\circ},\psi)
 \) 
where 
\[
 J^{\circ}_{X,Y} \coloneqq J_{Y,X} \colon    F(X)\, {\oplus^{\wr}}^{\circ} F(Y) = F(Y) \oplus^{\wr} F(X) \to F(Y \oplus X) = F(X \oplus^{\circ} Y)
\]
for every $X,Y \in \Obj(\me)$. This way, $\mathbf{F}^{\circ} \colon \me^{\circ} \to {\me^{\wr}}^{\circ}$ is a monoidal functor and taking monoidal opposites is itself a functor
\vspace{-0.2cm}
\[
 \moncat \xrightarrow{{^{\circ}}} \moncat  \, .
\]
\end{rem}

\begin{defn}[{\cite[Definition 8.1.1]{tensor-cat-book}} simplified via Convention \ref{tekman}] \label{bm-cat}
 Let $(\me,\oplus,0)$ be a monoidal category. A natural isomorphism $\tau \colon \oplus \to \oplus^{\circ}$ between the functors
\(
 \oplus, \oplus^{\circ} \colon \me \times \me \to \me
\), that is, a collection 
\[
 \{\tau_{X,Y} \colon X \oplus Y \to Y \oplus X : X,Y \in \Obj(\me)\}
\]
 of isomorphisms in $\me$ which are natural in $X,Y$, is called a \textbf{braiding} on $\me$ if for every $X,Y,Z \in \Obj(\me)$ the diagrams
\[
 \xymatrixcolsep{2cm} 
 \xymatrixrowsep{1.5cm}  
 \xymatrix{
 X \oplus Y \oplus Z \ar[rr]^-{\tau_{X, Y \oplus Z}} \ar[dr]_-{\tau_{X,Y} \oplus \id_{Z}} &  &Y \oplus Z \oplus X \\
 & Y \oplus X \oplus Z \ar[ur]_-{\,\,\,\,\,\,\id_{Y} \oplus \tau_{X,Z}} &
 }
 \]
 and
 \[
 \xymatrixcolsep{2cm} 
 \xymatrixrowsep{1.5cm}  
 \xymatrix{
 X \oplus Y \oplus Z \ar[rr]^-{\tau_{X \oplus Y, Z}} \ar[dr]_-{\id_{X} \oplus \tau_{Y,Z}} &  
 &Z \oplus X \oplus Y \\
 & X \oplus Z \oplus Y \ar[ur]_-{\,\,\,\,\,\,\tau_{X,Z} \oplus \id_{Y}} &
 }
 \] 
commute. A quadruple $(\me,\oplus,0,\tau)$ as above is called a \textbf{braided monoidal category}.

If a braiding  $\tau \colon \oplus \to \oplus^{\circ}$ on $\me$ satisfies $\tau^{\circ} = \tau^{-1}$, where $\tau^{\circ} \colon \oplus^{\circ} \to \oplus$ is given by $\tau^{\circ}_{X,Y} \coloneqq \tau_{Y,X}$, then $\tau$ is called \textbf{symmetric}. A braided monoidal category $(\me,\oplus,0,\tau)$ with a symmetric braiding $\tau$ is called a \textbf{symmetric monoidal category}.
\end{defn}

\begin{rem} \label{braided-opp}
 Given a braided monoidal category $(\me,\oplus,0,\tau)$, the collection
\[
 \{\tau^{\circ}_{X,Y} \colon X \oplus^{\circ} Y \to Y \oplus^{\circ} X : X,Y \in \Obj(\me)\}
\] 
defined as $\tau^{\circ}_{X,Y} \coloneqq \tau_{Y,X}$ defines a braiding $\tau^{\circ}$ on the monoidal category $\me^{\circ} = (\me,\oplus^{\circ},0)$, resulting in a braided monoidal category for which we again write $\me^{\circ} = (\me,\oplus^{\circ},0, \tau^{\circ})$.
\end{rem}

\begin{defn}[{\cite[Definition 8.1.1]{tensor-cat-book}}] \label{bm-functor}
Let $(\me,\oplus,0,\tau)$ and $(\me^{\wr}, \oplus^{\wr}, 0^{\wr}, \tau^{\wr})$ be braided monoidal categories. A monoidal functor 
\[
 (F,J,\psi) \colon (\me,\oplus,0) \to (\me^{\wr},\oplus^{\wr},0^{\wr})
\]
as in Definition \ref{monoidal-F} is called \textbf{braided} if the diagram
\[
 \xymatrixcolsep{2cm}
 \xymatrixrowsep{1.2cm} 
 \xymatrix{
 F(X) \oplus^{\wr} F(Y) \ar[r]^-{J_{X,Y}} \ar[d]_-{\tau^{\wr}_{F(X),F(Y)}}
 & F(X \oplus Y) \ar[d]^{F(\tau_{X,Y})}
 \\
 F(Y) \oplus^{\wr} F(X) \ar[r]^-{J_{Y,X}}
 & F(Y \oplus X)
 }
\]
commutes for every $X,Y \in \Obj(\me)$. Braided monoidal categories together with monoidal functors between them which are braided in this sense make up the category $\bmcat$.
\end{defn}

\begin{prop}[{\cite[Exercise 8.1.9]{tensor-cat-book}}] \label{exer}
 Let $(\me,\oplus,0,\tau)$ be a braided monoidal category. Then the triple $(\id_{\me},\tau^{-1}, \id_{0})$ defines a braided monoidal functor $\me \to \me^{\circ}$, with $(\id_{\me},\tau, \id_{0}) \colon \me^{\circ} \to \me$ its inverse braided monoidal functor.
\end{prop}

\begin{prop} \label{braid-op-iso}
 The endofunctor $\moncat \xrightarrow{{^{\circ}}} \moncat$ from Remark \ref{op-functor} lifts to an endofunctor
\[
 \bmcat \xrightarrow{{^{\circ}}} \bmcat \, ,
\] 
which is naturally isomorphic to $\id_{\bmcat}$ through the isomorphisms 
\[\pmb{\tau}_{\me} \coloneqq (\id_{\me}, \tau, \id_{0}) \colon \me^{\circ} \to \me\] from Proposition \ref{exer}.
\end{prop}
\begin{proof}
 Note that whether a monoidal functor between braided monoidal categories is braided or not is a property of the functor (not an extra structure). Hence it suffices to show for the first claim that
given $(\me,\oplus,0,\tau)$ and $(\me^{\wr}, \oplus^{\wr}, 0^{\wr}, \tau^{\wr})$ braided monoidal categories and a monoidal functor 
\[
 \mathbf{F} = (F,J,\psi) \colon (\me,\oplus,0) \to (\me^{\wr},\oplus^{\wr},0^{\wr})
\]
which is braided (with respect to $\tau$ and $\tau^{\wr}$), then $\mathbf{F}^{\circ} =  (F,J^{\circ},\psi)$ is also braided (with respect to $\tau^{\circ}$ and ${\tau^{\wr}}^{\circ}$ from Remark \ref{braided-opp}). This ammounts to the commutativity of the diagram
\[
 \xymatrixcolsep{2cm}
 \xymatrixrowsep{1.2cm} 
 \xymatrix{
 F(X) \,{\oplus^{\wr}}^{\circ} F(Y) \ar[r]^-{J^{\circ}_{X,Y}} \ar[d]_-{{\tau^{\wr}}^{\circ}_{F(X),F(Y)}}
 & F(X \oplus^{\circ} Y) \ar[d]^{F\pars*{\tau^{\circ}_{X,Y}}}
 \\
 F(Y) \,{\oplus^{\wr}}^{\circ} F(X) \ar[r]^-{J^{\circ}_{Y,X}}
 & F(Y \oplus^{\circ} X)
 }
\]
for every $X,Y \in \Obj(\me)$, which follows from the commutativity of
\[
 \xymatrixcolsep{2cm}
 \xymatrixrowsep{1.2cm} 
 \xymatrix{
 F(Y) \oplus^{\wr} F(X) \ar[r]^-{J_{Y,X}} \ar[d]_-{\tau^{\wr}_{F(Y),F(X)}}
 & F(Y \oplus X) \ar[d]^{F\pars*{\tau_{Y,X}}}
 \\
 F(X) \oplus^{\wr} F(Y) \ar[r]^-{J_{X,Y}}
 & F(X \oplus Y) \, .
 }
\]
For the naturality statement, we are to show that given a monoidal functor 
\[
 \mathbf{F} = (F,J,\psi) \colon (\me,\oplus,0) \to (\me^{\wr},\oplus^{\wr},0^{\wr}) \, ,
\]
the diagram
\[
 \xymatrixcolsep{1.5cm}
 \xymatrixrowsep{1cm} 
 \xymatrix{
 \me^{\circ} \ar[r]^-{\mathbf{F}^{\circ}} \ar[d]_-{\pmb{\tau}_{\me}}
 & {\me^{\wr}}^{\circ} \ar[d]^{\pmb{\tau}_{\me^{\wr}}}
 \\
 \me \ar[r]^-{\mathbf{F}}
 & \me^{\wr}
 }
\]
in $\bmcat$ commutes. According to Remark \ref{moncat-comp}, we have
\begin{align*}
 \pmb{\tau}_{\me^{\wr}} \circ \mathbf{F}^{\circ} 
 &= \pars*{\id_{\me^{\wr}}, \tau^{\wr}, \id_{0^{\wr}}} \circ (F,J^{\circ},\psi) 
 \\
 &= \pars*{\id_{\me^{\wr}} \circ F,W,\theta} = \pars*{F,W,\theta} 
\end{align*}
where 
\begin{itemize}
 \item for every $X,Y \in \Obj(\me)$, the morphism $W_{X,Y}$ is the composite
\[
 F(X) \oplus^{\wr} F(Y) \xrightarrow{\tau^{\wr}_{F(X),F(Y)}} 
 F(X) {\oplus^{\wr}}^{\circ} F(Y) \xrightarrow{J^{\circ}_{X,Y}} F(X \oplus^{\circ} Y) \, ,
\]
 equivalently the composite
\[
 F(X) \oplus^{\wr} F(Y) \xrightarrow{\tau^{\wr}_{F(X),F(Y)}} 
 F(Y) \oplus^{\wr} F(X) \xrightarrow{J_{Y,X}} F(Y \oplus X) \, ,
\]
\item  and $\theta$ is the composite $F(0) \xrightarrow{\psi} 0^{\wr} \xrightarrow{\id_{0^{\wr}}} 0^{\wr}$, that is $\theta = \psi$.
\end{itemize}
Again according to Remark \ref{moncat-comp}, we have
\begin{align*}
 \mathbf{F} \circ \pmb{\tau}_{\me} &= (F,J,\psi) \circ (\id_{\me}, \tau, \id_{0})
 \\
 &= \pars*{F \circ \id_{\me}, \ov{W}, \ov{\theta}} = \pars*{F, \ov{W}, \ov{\theta}}
\end{align*}
where
\begin{itemize}
 \item for every $X,Y \in \Obj(\me)$, the morphism $\ov{W}_{\!X,Y}$ is the composite
\[
 F(X) \oplus^{\wr} F(Y) \xrightarrow{J_{X,Y}} 
 F\pars*{X \oplus Y} \xrightarrow{F\pars*{\tau_{X,Y}}} F(X \oplus^{\circ} Y) = F(Y \oplus X) \,,
\] 
\item and $\ov{\theta}$ is the composite $F(0) \xrightarrow{F(\id_{0})} F(0) \xrightarrow{\psi} 0^{\wr}$, that is $\ov{\theta} = \psi$.
\end{itemize}
Here, as $\mathbf{F} = (F,J,\psi)$ is braided, we have
\[
 W_{X,Y} = J_{Y,X}\circ \tau^{\wr}_{F(X),F(Y)} = F(\tau_{X,Y}) \circ J_{X,Y} = \ov{W}_{\!X,Y}
\]
for every $X,Y$, and hence
\[
 \pmb{\tau}_{\me^{\wr}} \circ \mathbf{F}^{\circ} = \pars*{F,W,\theta} = \pars*{F,W,\psi}
 = \pars*{F,\ov{W},\ov{\theta}} = \mathbf{F} \circ \pmb{\tau}_{\me}
\]
as desired.
\end{proof}
\subsubsection{$\atsi$-groups out of monoidal categories}
\begin{defn} \label{monoidal-to-A-grp}
Let $(\me,\oplus,0)$ be a monoidal category and $c,x \in \Obj(\me)$. We shortly write $\me(c,x)$ for the $\atsi$-group $\me\pars*{c,\,-\oplus \id_{x}}$ defined in Definition \ref{endo-to-A-grp}: spelled out at the level of objects, we have
 \[
  \me(c,x)_{n} = \Aut_{\me}\pars*{c \oplus x^{\oplus n}} \, .
 \] 
\end{defn}

\begin{rem} \label{Gcx-maps}
 Let $\mathbf{F} = (F,J,\psi) \colon (\me,\oplus,0) \to (\me^{\wr},\oplus^{\wr},0^{\wr})$ be a monoidal functor between monoidal categories as in Definition \ref{monoidal-F}. Given
\begin{itemize}
 \item $c,x \in \Obj(\me)$,
 \item $c^{\wr}, x^{\wr} \in \Obj(\me)$,
 \item isomorphisms $\gamma \colon F(c) \to c^{\wr}$ and $\xi \colon F(x) \to x^{\wr}$ in $\me^{\wr}$,
\end{itemize}
we may define a map $\mathbf{F}[\gamma,\xi] \colon \me(c,x) \to \me^{\wr}(c^{\wr},x^{\wr})$ of $\atsi$-groups inductively as follows:
\begin{birki}
 \item For $n=0$, we define the group homomorphism
\begin{align*}
 \mathbf{F}[\gamma,\xi]_{0} \colon \Aut_{\me}(c) &\to \Aut_{\me^{\wr}}(c^{\wr})
 \\
 f &\mapsto \gamma \circ F(f) \circ \gamma^{-1} \, .
\end{align*}
 \item For $n \in \zz_{\geq 1}$, noting that $(\gamma \oplus^{\wr} \xi) \circ J_{c,x}^{-1} \colon F(c \oplus x) \to c^{\wr} \oplus^{\wr} x^{\wr}$ is an isomorphism in $\me^{\wr}$, we define
\begin{align*}
 \mathbf{F}[\gamma,\xi]_{n} \coloneqq \mathbf{F}\!\left[
 (\gamma \oplus^{\wr} \xi) \circ J_{c,x}^{-1},\,\xi
 \right]_{n-1} \colon \me(c \oplus x, x)_{n-1} \to 
 \me^{\wr}\pars*{c^{\wr} \oplus^{\wr} x^{\wr}, x^{\wr}}_{n-1}
\end{align*}
since $\me(c \oplus x, x)_{n-1} = \Aut_{\me}\pars*{c \oplus x \oplus x^{\oplus n-1}} = \Aut_{\me}\pars*{c \oplus x^{\oplus n}} = \me(c,x)_{n}$.
\end{birki}
Checking the naturality of $\mathbf{F}[\gamma,\xi]$ (for all $\xi$, $\gamma$) here amounts to the commmutativity of
\[
 \xymatrixcolsep{2.5cm}
 \xymatrixrowsep{1.5cm} 
 \xymatrix{
 \Aut_{\me}\pars*{c \oplus x^{\oplus n}} \ar[r]^-{- \oplus \id_{x}} 
 \ar[d]_-{\mathbf{F}[\gamma,\xi]_{n}}
 &
 \Aut_{\me}\pars*{c \oplus x^{\oplus n+1}}
 \ar[d]^-{\mathbf{F}[\gamma,\xi]_{n+1}} 
 \\
 \Aut_{\me^{\wr}}\pars*{c^{\wr} \oplus^{\wr} {x^{\wr}}^{\oplus^{\wr} n}}
 \ar[r]^-{- \oplus^{\wr} \id_{x^{\wr}}}
 &
 \Aut_{\me^{\wr}}\pars*{c^{\wr} \oplus^{\wr} {x^{\wr}}^{\oplus^{\wr} n+1}}
 }
\]
which can be shown by inducting on $n$: for $n=0$ it follows from the computation that given $f \in \Aut_{\me}\pars*{c}$, we have
\begin{align*}
 \mathbf{F}[\gamma,\xi]_{1}\pars*{f \oplus \id_{x}} 
 &=
 \mathbf{F}\!\left[
 (\gamma \oplus^{\wr} \xi) \circ J_{c,x}^{-1},\,\xi
 \right]_{0}\pars*{f \oplus \id_{x}}
 \\
 &= \pars*{(\gamma \oplus^{\wr} \xi) \circ J_{c,x}^{-1}}
 \circ F(f \oplus \id_{x}) \circ
 \pars*{(\gamma \oplus^{\wr} \xi) \circ J_{c,x}^{-1}}^{-1}
 \\
 &= (\gamma \oplus^{\wr} \xi) \circ J_{c,x}^{-1}
 \circ F(f \oplus \id_{x}) \circ
 J_{c,x} \circ \pars*{\gamma^{-1} \oplus^{\wr} \xi^{-1}}
 \\
 &= (\gamma \oplus^{\wr} \xi) \circ 
 \pars*{F(f) \oplus^{\wr} F(\id_{x})}
 \circ \pars*{\gamma^{-1} \oplus^{\wr} \xi^{-1}} 
 \\
 &= \pars*{\gamma \circ F(f) \circ \gamma^{-1}} \oplus^{\wr}
 \pars*{\xi \circ F(\id_{x}) \circ \xi^{-1}}
 \\
 &= \mathbf{F}[\gamma,\xi]_{0}(f) \oplus^{\wr} \id_{x^{\wr}} \, ,
\end{align*}
and for $n \geq 1$, it follows from the induction hypothesis that
\[
 \xymatrixcolsep{2.5cm}
 \xymatrixrowsep{1.5cm} 
 \xymatrix{
 \Aut_{\me}\pars*{c \oplus x \oplus x^{\oplus n-1}} \ar[r]^-{- \oplus \id_{x}} 
 \ar[d]_-{\mathbf{F}\!\left[
 (\gamma \oplus^{\wr} \xi) \circ J_{c,x}^{-1},\,\xi
 \right]_{n-1} =\, \mathbf{F}[\gamma,\xi]_{n}}
 &
 \Aut_{\me}\pars*{c \oplus x \oplus x^{\oplus n}}
 \ar[d]^-{\mathbf{F}[\gamma,\xi]_{n+1} =\, \mathbf{F}\!\left[
 (\gamma \oplus^{\wr} \xi) \circ J_{c,x}^{-1},\,\xi
 \right]_{n}} 
 \\
 \Aut_{\me^{\wr}}\pars*{c^{\wr} \oplus^{\wr} x^{\wr} \oplus^{\wr} {x^{\wr}}^{\oplus^{\wr} n-1}}
 \ar[r]^-{- \oplus^{\wr} \id_{x^{\wr}}}
 &
 \Aut_{\me^{\wr}}\pars*{c^{\wr} \oplus^{\wr} x^{\wr} \oplus^{\wr} {x^{\wr}}^{\oplus^{\wr} n}}
 }
\]
commutes.
\end{rem}

\begin{prop} \label{tc-from-mon}
 Let $(\me, \oplus, 0 ,\tau)$ be a braided monoidal groupoid and $c,x \in \Obj(\me)$ so that we have an associated $\atsi$-group $\me(c,x)$ from Definition \ref{monoidal-to-A-grp}. Then the sequence $(\gamma_{n}: n \geq 2)$ defined as
\begin{align*}
 \gamma_{n} \coloneqq 
 \id_{c} \oplus \id_{x^{\oplus n-2}} \oplus \tau_{x,x}
 \in
 \me(c,x)_{n} = \Aut_{\me}\pars*{c \oplus x^{\oplus n}}
\end{align*} 
is tail-central in $\me(c,x)$ in the sense of Definition \ref{tail-central}.
\end{prop}
\begin{proof}
 We simply compute that for every $n \geq 2$ and $g \in \me(c,x)_{n-2}$, we have
\begin{align*}
 \gamma_{n} \cdot \me(c,x)_{(n-2,n)}(g)
 &= \pars*{\id_{c} \oplus \id_{x^{\oplus n-2}} \oplus \tau_{x,x}} \circ 
 (g \oplus \id_{x \oplus x}) 
 \\
 &= \pars*{\id_{c \oplus x^{\oplus n-2}} \oplus \tau_{x,x}} \circ 
 (g \oplus \id_{x \oplus x}) 
 \\
 &= g \oplus \tau_{x,x} 
 \\
 &= (g \oplus \id_{x \oplus x}) \circ
 \pars*{\id_{c \oplus x^{\oplus n-2}} \oplus \tau_{x,x}} 
 \\
 &= \me(c,x)_{(n-2,n)}(g) \cdot \gamma_{n} \, .
\end{align*}
\end{proof}

\subsection{Semi-simplicial notions} \label{semi-simp}
Let $\OI$ be the category of totally ordered finite sets and order-preserving injections. The ordinal sum $\oplus$ \cite[Section 4.1]{nk-posets} defines a monoidal structure on $\OI$ with the empty set $\empt$ as the unit object and it has a universal property as such:

\begin{thm}[{\cite[Theorem 2.3]{patzt-central}}] \label{initial}
 Let $(\me,\oplus,0)$ be a monoidal category where the unit object $0$ is initial. Then for every object $x$ in $\me$, there is a unique monoidal functor $F_{x} \colon \OI \to \me$ with identity coherence data that sends every singleton to $x$.
\end{thm}

We write $\OI_{+}$ for the full subcategory of $\OI$ consisting of nonempty totally ordered finite sets. A functor $\OI_{+}^{\opp} \to \ce$ is called a \textbf{semi-simplicial} object in $\ce$, and we write
 \[
   \ses \ce \coloneqq \left[\OI_{+}^{\opp}, \ce\right] \, .
 \]
 A functor $\OI^{\opp} \to \ce$ is called a \textbf{augmented semi-simplicial} object in $\ce$, and we write
 \[
   \aus \ce \coloneqq [\OI^{\opp}, \ce] \, .
 \]
 We adopt the topologist's convention and write $[k]$ for the totally ordered set
 \[
 0 < \cdots < k
 \]
so that $[-1] = \empt$ and $[k]$ has $k+1$ elements for each $k \geq -1$. Given a semi-simplicial (or augmented semi-simplicial) object $X_{\star}$ in $\ce$, we write $X_{k} \coloneqq X_{[k]}$ for each $k \geq 0$ (or $k \geq -1$). 

\begin{rem}
The notation above is in clash with what we use for, say, an $\FI$-object $X$: there $X_{n}$ means evaluating $X$ at a finite set of size $n$, not $n+1$. In what follows, it will always be clear from context as to whether the indexing is semi-simplicial or not. One cue is that our semi-simplicial objects will almost always be defined with a five-pointed star such as $X_{\star}$ or $\PBC_{\star}(V)$ before their individual simplices are referred to.
\end{rem}

Given a non-empty semi-simplicial set $X_{\star}$, the quantity
\[
 \sup\{k \in \zz_{\geq 0} : X_{k} \neq \empt\} \in \zz_{\geq 0} \cup\{\infty\}
\]
is called the \textbf{dimension} of $X_{\star}$.

Given a semi-simplicial set $X_{\star}$, we write
\[
  \Face(X_{\star}) \coloneqq \bigsqcup_{k=0}^{\infty} X_{k}
\]
and order this set via declaring $\sigma \leq \tau$ for $\sigma \in X_{a}$ and $\tau \in X_{b}$ if there exists a morphism $f \colon [b] \to [a]$ in $\OI_{+}$ such that
\[
 X_{f} \colon X_{a} \to X_{b}
\]
sends $\tau$ to $\sigma$. It is straightforward to verify that $\leq$ is a partial order (the antisymmetry follows from the only endomorphisms in $\OI_{+}$ being the identities) and we call $\Face(X_{\star})$ the \textbf{face poset} of $X_{\star}$. The construction defines a functor
\[
\Face \colon \ses \Set \to \Pos \, .
\]
The following type of semi-simplicial sets turn out to be most closely related with their face posets as we shall see in Theorem \ref{regular-ok}.
\begin{defn} \label{regular-defn}
 A semi-simplicial set $X_{\star}$ is called \textbf{regular} if every simplex in $X$ has distinct vertices, more precisely, for every $k \in \zz_{\geq 0}$ and $a \in X_{k}$, the map
\begin{align*}
 \Hom_{\OI_{+}}([0],[k]) &\to X_{0}
 \\
 f &\mapsto X_{f}(a)
\end{align*}
is injective.
\end{defn}



\subsubsection{A general construction} \label{sec-omega}
Suppose $\ce$ is equipped with a shifting context $(\sigma, \eta)$ as in Definition \ref{defn:shifting-context}. The endofunctor category $\End\!\big( [\ce,\de] \big)$ is monoidal with functor composition, and the undercategory
\[
 \id_{[\ce,\de]} \darr \End\!\big( [\ce,\de] \big)
\]
is monoidal whose unit object is initial. The shift functor $\shift{}{} \colon [\ce,\de] \to [\ce,\de]$ receives a natural transformation 
\[
\eta_{*} \colon \id_{[\ce,\de]} \to \shift{}{} 
\]
via $\eta$. Hence by Theorem \ref{initial} there is a unique monoidal functor
\[
 \OI \to \id_{[\ce,\de]} \darr \End\!\big( [\ce,\de] \big)
\]
with identity coherence that assigns every singleton to $\eta_{*}$. If furthermore $\de$ is cocomplete, $\shift{}{} = - \circ \sigma$ has a left adjoint
\[
 \desu \colon [\ce,\de] \to [\ce,\de]
\]
given by the left Kan extension along $\sigma\colon \ce \to \ce$ \cite[Corollary 6.2.6]{riehl-context}. Through $\eta_{*}$, this adjunction also yields a natural transformation
\[
\xi \colon \desu \to \id_{[\ce,\de]}
\]
and for instance by \cite[Theorem 4.1]{kan-adjoint} a functor
\[
 \OI^{\opp} \to \End\!\big( [\ce,\de] \big) \darr \id_{[\ce,\de]}
\]
which sends each finite ordered set $[n] : 0 < \cdots < n$ to the natural transformation
\[
 \xi \circ \cdots \circ \Omega^{n}\xi \colon \desu^{n+1} \to \id_{[\ce,\de]} \, .
\]
We can equivalently specify this data as a functor
\[
 \wt{\desu}^{\ce}_{\star} \colon [\ce,\de] \to \big[\OI^{\opp},[\ce,\de]\big] = \aus[\ce,\de]
\]
which satisfies $\wt{\desu}^{\ce}_{-1}V = V$, $\wt{\desu}^{\ce}_{0}V = \desu V$, $\wt{\desu}^{\ce}_{1}V = \desu^{2} V$ etc. for every $V \colon \ce \to \de$. For the restriction of $\wt{\desu}^{\ce}_{\star}$ to  $\OI_{+}^{\opp}$ we write
\[
 \desu_{\star}^{\ce} \colon [\ce,\de] \to \big[\OI_{+}^{\opp},[\ce,\de]\big] = \ses[\ce,\de]\, .
\]
Writing \(
 \bul \,\colon \ce \to \Set
 \)
for the constant $\ce$-set that sends every object to the singleton $\{*\}$ and every morphism to the identity, the semi-simplicial $\ce$-set $\desu_{\star}^{\ce}(\bul)$ occupies an important place in the theory. 

Proposition \ref{des} below provides an explicit description for $\desu_{\star}^{\ce}(\bul)$. We shall write $\Pi_{0}$ for the functor that assigns a category with a small skeleton to the set of its connected components.

\begin{lem} \label{constant-colim}
 Let $\je, \de$ be categories where $\je$ has a small skeleton, $\de$ is cocomplete, and $d \in \Obj(\de)$.  Then the colimit of the constant functor
\begin{align*}
 \ul{d} \colon \je &\to \de \\
 x &\mapsto d
\end{align*}
is the $\Pi_{0}(\je)$-fold coproduct of $d$.
\end{lem}
\begin{proof}
 Considering $\Pi_{0}(\je)$ as a discrete category, $\ul{d}$ factors through $\Pi_{0}(\je)$ as
\begin{align*}
 \xymatrix{
 \je \ar[r]^-{\ul{d}} \ar[d]_-{\Pi_{0}} & \de \\
 \Pi_{0}(\je) \ar[ur]_-{\ul{d}}
 } \, .
\end{align*}
 Here every object in $\Pi_{0}(\je)$ is a connected component $C$ of objects in $\je$. For such $C$, the undercategory $C \darr \Pi_{0}$ is precisely the full subcategory of $\je$ on $C$, in particular $C \darr \Pi_{0}$ is connected. Thus by \cite[Lemma 8.3.4]{riehl-book} the functor $\Pi_{0}$ is final \cite[Definition 8.3.2]{riehl-book} and hence
 \[
  \colim(\ul{d} \colon \je \to \de) =
  \colim(\ul{d} \colon \Pi_{0}(\je) \to \de)  
 \]
which yields the claim as $\Pi_{0}(\je)$ is discrete and discrete colimits are coproducts.
\end{proof}
\begin{lem} \label{constant-Kan}
 Let $\ce, \de$ be categories where $\ce$ has a small skeleton, $\de$ is cocomplete, and $d \in \Obj(\de)$.  Then for every functor $\kappa \colon \ce \to \ce$,
the left Kan extension $\lkan_{\kappa}{\ul{d}}$ of the constant functor
\begin{align*}
 \ul{d} \colon \ce &\to \de \\
 x &\mapsto d
\end{align*}
along $\kappa$ can be described as
\begin{align*}
 \lkan_{\kappa}{\ul{d}} \colon \ce &\to \de
 \\
 y &\mapsto \text{the $\Pi_{0}(\kappa \darr y)$-fold coproduct of $d$} \, .
\end{align*}
\end{lem}
\begin{proof}
 By \cite[Theorem 6.2.1]{riehl-context} the evaluation of  $\lkan_{\kappa}{\ul{d}}$ at an object $y$ of $\ce$ is the colimit of the constant functor
 \[
  \ul{d} \colon \kappa \darr y \rarr \de \, ,
 \]
 where $\kappa \darr y$ inherits the property of having a small skeleton from $\ce$. We are done by Lemma \ref{constant-colim}.
\end{proof}

\begin{prop} \label{des}
 Let $\ce$ be a category with a small skeleton equipped with a shifting context $(\sigma, \eta)$. Then the semi-simplicial $\ce$-set $\desu_{\star}^{\ce}(\bul)$ can be described as
\begin{align*}
  \desu_{\star}^{\ce}(\bul) \colon \OI_{+}^{\opp} \times \ce &\to \Set
  \\
  ([k],y) &\mapsto \Pi_{0}(\sigma^{k+1} \darr y) \, .
\end{align*}
If furthermore $e$ is initial in $\ce$ such that $\Hom_{\ce}(x,e) = \empt$ whenever $x \ncong e$, we have the description
\begin{align*}
  \desu_{\star}^{\ce}(\bul) \colon \OI_{+}^{\opp} \times \ce &\to \Set
  \\
  ([k],y) &\mapsto \Hom_{\ce}(\sigma^{k+1}e, y) \, .
\end{align*}
\end{prop}
\begin{proof}
 As mentioned earlier, for every integer $k \geq -1$, the $\ce$-set $\desu_{k}^{\ce}(\bul)$ is equal to $\desu^{k+1}(\bul)$, that is, the $(k+1)$-fold composite of
\[
 \desu \colon [\ce,\Set] \to [\ce,\Set]
\]
applied to the constant $\ce$-set $\bul$, where $\desu = \lkan_{\sigma}$. We shall prove the first claim by induction on $k \geq -1$. The base step $k=-1$ follows because
\[
  \desu^{0}(\bul) = \bul = \Pi_{0}(\id_{\ce} \darr -)
\]
as $\id_{\ce} \darr y$ is connected for every object $y$ in $\ce$. Assuming the first claim holds for $k \geq 0$, the $\ce$-set $\Omega^{k+1}(\bul)$ is given by
\begin{align*}
  \desu^{k+1}(\bul) \colon \ce &\to \Set
  \\
  y &\mapsto \Pi_{0}(\sigma^{k+1} \darr y) \, .
\end{align*}
Since the set $\Pi_{0}(\sigma^{k+1} \darr y)$ is the $\Pi_{0}(\sigma^{k+1} \darr y)$-fold coproduct of the singleton $\{*\}$ in $\Set$, by Lemma \ref{constant-Kan} we have
\[
 \desu^{k+1}(\bul) = \lkan_{\sigma^{k+1}}(\bul)
\]
and hence
\begin{align*}
  \desu^{k+2}(\bul) &= \desu(\desu^{k+1}(\bul))
  \\
  &= \lkan_{\sigma}(\lkan_{\sigma^{k+1}}(\bul))
  \\
  &= \lkan_{\sigma^{k+2}}(\bul) \, ,
\end{align*}
which by Lemma \ref{constant-Kan} satisfies
\begin{align*}
  \desu^{k+2}(\bul) \colon \ce &\to \Set
  \\
  y &\mapsto \Pi_{0}(\sigma^{k+2} \darr y) \, ,
\end{align*}
completing the inductive step.

For the second claim, write $\kappa = \sigma^{k+1}$ and note that for every object $y$ in $\ce$ there is a natural map
\begin{align*}
  \Xi \colon \Hom_{\ce}(\kappa e, y) 
  &\to \Pi_{0}(\kappa \darr y) \\
  \alpha &\mapsto [(e,\alpha)] 
\end{align*}
where $[(e,\alpha)]$ is the set of objects in the connected component of 
\[
 (e,\alpha) \in \Obj(\kappa \darr y)
\] 
To see $\Xi$ is surjective, say $(x,\beta) \in \Obj(\kappa \darr y)$, that is, $x \in \Obj(\ce)$ and $\beta \colon \kappa x \to y$ is a morphism in $\ce$. The unique morphism $u \colon e \to x$, together with the commutative triangle
\begin{align*}
 \xymatrix{
 \kappa e \ar[r]^-{\kappa u} \ar[dr]_-{\alpha \,\coloneqq\, \beta \cdot \kappa u} 
 & \kappa x \ar[d]^-{\beta} \\
 & y
 } \, .
\end{align*}
defines a morphism $(e, \alpha) \to (x,\beta)$ in $\kappa \darr y$, so that
\[
 [(x,\beta)] = [(e,\alpha)] = \Xi(\alpha)
\]
Finally we show that the $\Xi$ is injective, this time using the hypothesis on $e$. Assume $\Xi(\alpha) = \Xi(\alpha')$, that is, $(e,\alpha)$ and $(e,\alpha)$ lie in the same connected component in the category $\kappa \darr y$. Suppose the shortest length of a zig-zag connecting $(e,\alpha)$ with $(e,\alpha')$ is $n \geq 1$. Such a zig-zag necessarily has the form
\begin{align*}
 \xymatrix{
 \kappa e \ar[r]^-{\kappa u_{1}} \ar[drr]_-{\alpha} 
 & \kappa x_{1} \ar^{\alpha_{1}}[dr] 
 & \kappa x_{2} \ar[l]_-{\kappa u_{2}} \ar[d]^-{\alpha_{2}}
 & \cdots & \kappa x_{n-1} \ar[dll]_-{\alpha_{n-1}}
 & \kappa e \ar[l]_-{\kappa u_{n}} \ar[dlll]^-{\alpha'} \\
 & & y
 } \, .
\end{align*}
But then the unique morphism $f \colon e \to x_{2}$ satisfies $u_{1} = u_{2}f$ due to $e$ being initial, so that
\[
 \alpha_{2} \cdot \kappa f = \alpha_{1} \cdot \kappa u_{2} \cdot \kappa f
 = \alpha_{1} \cdot \kappa u_{1} = \alpha \, ,
\]
hence we can skip $x_{1}$ write down a shorter zig-zag. This contradicts the minimality of $n$ and forces $n=0$, that is, $\alpha = \alpha'$ .
\end{proof}

We are now in a position to recover some of the semi-simplicial sets in Appendix \ref{specific}.

\begin{defn} \label{x-shift}
 Let $(\me,\oplus,0)$ be a monoidal category where the unit object $0$ is initial, and $x$ be an object in $\me$. The functor 
 \[
  - \oplus x \colon \me \to \me
 \]
 receives a natural transformation from $- \oplus 0 = \id_{\me}$ via the unique morphism $0 \to x$. We call the resulting shifting context (on $\me$, recall Definition \ref{defn:shifting-context})  \textbf{$x$-shifting}.
\end{defn}

\begin{rem} \label{patzt-Kx}
 If $\pars*{\me,\oplus,0}$ and $x$ satisfy the setting of Definition \ref{x-shift} so that $\me$ is equipped with $x$-shifting, the associated functor
\[
 \wt{\desu}^{\ce}_{\star} \colon [\ce,\de] \to \big[\OI^{\opp},[\ce,\de]\big] = \aus[\ce,\de]
\]
is denoted $K^{x}$ in \cite[Definition 2.4]{patzt-central}.
\end{rem}

Note that in the situation of Definition \ref{x-shift} when $\me$ has a small skeleton, if we also have $\Hom_{\me}(c,0) = \empt$ whenever $c \ncong 0$, by Proposition \ref{des} the resulting semi-simplicial $\me$-set $\desu_{\star}^{\me}(\bul)$ has the description
\begin{align*}
  \desu_{\star}^{\me}(\bul) \colon \OI_{+}^{\opp} \times \me &\to \Set
  \\
  ([k],y) &\mapsto \Hom_{\me}\!\left(x^{\oplus [k]}, y\right) \, .
\end{align*}
In particular, for every object $y$ in $\me$, the evaluation $  \desu_{\star}^{\me}(\bul)_{y}$ is precisely the semi-simplicial set (\ref{sfx}) in the beginning of Appendix \ref{specific} that all the specific semi-simplicial sets related to rings we specialized from. Let us note the ones we shall later refer to explicitly.

\begin{rem} \label{omega-PBC}
 Let $R$ be a ring. The monoidal category $\VIC(R)$ from Definition \ref{su-pbc-defn} has a countable skeleton, has $0$ as an initial object where $0$ does not receive any morphism from nonzero objects. Therefore equipping $\VIC(R)$ with $R$-shifting, the resulting semi-simplicial $\VIC(R)$-set $\desu_{\star}^{\VIC(R)}(\bul)$ has the description
 \begin{align*}
  \desu_{\star}^{\VIC(R)}(\bul) \colon \VIC(R) &\to \ses\Set
  \\
  V &\mapsto \PBC_{\star}(V)
\end{align*}
where the semi-simplicial set $\PBC_{\star}(V)$ is also defined in Definition \ref{su-pbc-defn}.
\end{rem}

\begin{rem} \label{omega-PBC-fU}
 Let $R$ be a commutative ring with a subgroup $\fU \leq R^{\times}$. The monoidal category $\VIC(R,\fU)$ from Definition \ref{pbc-U} has a countable set of objects, has $0$ as an initial object where $0$ does not receive any morphism from nonzero objects. Therefore equipping $\VIC(R,\fU)$ with $R$-shifting, the resulting semi-simplicial $\VIC(R,\fU)$-set has the description
 \begin{align*}
  \desu_{\star}^{\VIC(R,\fU)}(\bul) \colon \VIC(R,\fU) &\to \ses\Set
  \\
  R^{n} &\mapsto \PBC^{\fU}_{\star}(R^{n})
\end{align*}
where the semi-simplicial set $\PBC^{\fU}_{\star}(R^{n})$ is also defined in Definition \ref{pbc-U}.
\end{rem}

\begin{rem} \label{omega-HU}
 Let $R$ be a commutative ring. The monoidal category $\SI(R)$ from Definition \ref{hu-defn} has a small skeleton, has $0$ as an initial object where $0$ does not receive any morphism from nonzero objects. Therefore equipping $\SI(R)$ with $H_{R}\,$-shifting, the resulting semi-simplicial $\SI(R)$-set $\desu_{\star}^{\SI(R)}(\bul)$ has the description
 \begin{align*}
  \desu_{\star}^{\SI(R)}(\bul) \colon \SI(R) &\to \ses\Set
  \\
  V &\mapsto \HU_{\star}(V)
\end{align*}
where the semi-simplicial set $ \HU_{\star}(V)$ is also defined in Definition \ref{hu-defn}.
\end{rem}

\begin{rem}
 Let $\me$ be a monoidal category where the unit object $0$ is initial and $0$ does not receive any morphism from nonzero objects. Let $x$ be an object in $\me$ and equip $\me$ with $x$-shifting. The resulting semi-simplicial $\me$-set $\desu_{\star}^{\me}(\bul)$ has the property
 \begin{align*}
  \desu_{\star}^{\me}(\bul) \colon \me &\to \ses\Set \\
  x^{\oplus n} &\mapsto W_{n}(0,x)_{\bul}
\end{align*}
in the sense of the $W_{n}$-construction of \cite[Definition 2.1]{rw-wahl-stab}.
\end{rem}

%

\subsection{Topological notions}

Let $\Pos$ be the category of posets and monotone maps, $\scx$ be the category of abstract simplicial complexes and simplicial maps, and $\CW$ be the category of CW-complexes and cellular maps. There are functors
\[
 \Pos \xrightarrow{\Delta} \scx \xrightarrow{|\cdot|} \CW
\]
where $\Delta$ is the \textbf{order complex} \cite[Section 1.1]{wachs-poset} and $|\cdot|$ is the \textbf{geometric realization} \cite[Section 3.3]{fri-pic-cellular} constructions. There is also a realization functor 
\[
 |\cdot| \colon \ses \Set \to \CW
\]
for semi-simplicial sets \cite[Section 2]{rs-delta-1}. Using these functors, one can attach topological properties or invariants such as connectivity or homology groups to a poset, a simplicial complex, or a semi-simplicial set. 

\begin{conv}
 We often employ the common practice of omitting writing the realization functors for the topology of posets, simplicial complexes, or semi-simplicial sets. For instance when we say that a monotone map $f \colon P \to Q$ of posets is a homotopy equivalence, we mean that $|\Delta(f)| \colon |\Delta(P)| \to |\Delta(Q)|$ is a homotopy equivalence.
\end{conv}

\begin{thm} \label{regular-ok}
Every regular semi-simplicial set $X_{\star}$ is homeomorphic to its face poset $\Face(X_{\star})$.
\end{thm}
\begin{proof}
 The semi-simplicial set $X_{\star}$ being regular (Definition \ref{regular-defn}) implies that the CW-complex $|X_{\star}|$ is regular in the sense of \cite[Definition 2.55]{kozlov-book}. The poset $\F(|X_{\star}|)$ of \cite[Definition 10.11]{kozlov-book} is precisely $\Face(X_{\star})$ and the desired homeomorphism is then \cite[(10.5)]{kozlov-book}.
\end{proof}

\begin{defn} \label{space-conn}
Let $X$ be a nonempty (topological) space and $k \in \zz_{\geq 0}$. We say $X$ is $k$-connected if $\pi_{j}(X,x)$ is  a singleton for every $x \in X$ and $0\leq j \leq k$. We say $X$ is $k$-acyclic if $\wt{\co}_{j}(X) = 0$ for $0 \leq j \leq k$.
\end{defn}

\begin{conv}
 For an integer $k \leq -2$, every space is $k$-connected and $k$-acyclic. A space is $(-1)$-connected if and only if it is  $(-1)$-acyclic if and only if it is nonempty. 
\end{conv}

\begin{defn}
 Let $X,Y$ be nonempty spaces and $k \in \zz$. We say a (continuous) map $f \colon X \to Y$ is $k$-connected if for every $x \in X$, the homotopy fiber of the pointed map
 \[
  (X,x) \to (Y,f(x))
 \]
is $(k-1)$-connected.
\end{defn}

\begin{rem} \label{map-conn}
 By the long exact sequence of homotopy groups, a map $f \colon X \to Y$ between nonempty spaces is $k$-connected if and only if for every $x \in X$, and $j \in \zz_{\geq 0}$, the induced map
 \[
  \pi_{j}(f) \colon \pi_{j}(X,x) \to \pi_{j}(Y,f(x))
 \]
is an isomorphism if $j < k$ and surjective if $j = k$.
\end{rem}

%

\begin{prop} \label{inclusion-conn}
 Suppose $Y$ is a nonempty CW-complex and $\empt \neq X \subseteq Y$ is a subcomplex which contains every cell of $Y$ of dimension $\leq k$. Then the inclusion $X \emb Y$ is $k$-connected.
\end{prop}
\begin{proof}
Use Remark \ref{map-conn}.
\end{proof}

\begin{prop} \label{conn-transfer}
 Let $f \colon X \to Y$ be a map between nonempty spaces and $a,b \in \zz$ such that
\begin{itemize}
 \item $f$ is $a$-connected,
 \item $Y$ is $b$-connected,
 \item $0 \leq j \leq b$ implies $0 \leq j < a$.
\end{itemize}
Then $X$ is $b$-connected.
\end{prop}
\begin{proof}
 As $X \neq \empt$, the conclusion holds if $b \leq -1$. Next, assume $b \geq 0$ and fix $0 \leq j \leq b$. Then we also have $0 \leq j < a$ so that for every $x \in X$, the induced map
 \[
  \pi_{j}(f) \colon \pi_{j}(X,x) \to \pi_{j}(Y,f(x))
 \]
is an isomorphism due to $f$ being $a$-connected via Remark \ref{map-conn}. But $\pi_{j}(Y,f(x)) = 0$ because $Y$ is $b$-connected, so $\pi_{j}(X,x) = 0$. This vanishing happens for every $0 \leq j \leq b$ and $x \in X$, hence $X$ is $b$-connected by Definition \ref{space-conn}.
\end{proof}

\subsection{Homological notions}
\begin{defn} \label{H3-defn}
 Let $\ce$ be a category with a small skeleton equipped with a function $\rk \colon \Iso(\ce) \to \zz_{\geq 0}$ and a shifting context $(\sigma, \eta)$ as in Definition \ref{defn:shifting-context}. For $a \in \zz_{\geq 1}$, $b \in \zz_{\geq 0}$, we say that 
\[\text{$\ce$ satisfies $\hth{a,b}$}\] 
if one (hence both) of the following equivalent conditions hold:
\begin{birki}
 \item  For every integer $k \geq -1$, the $\ce$-module $\wt{\co}_{k}\!\left(
  \desu^{\ce}_{\star}(\bul)
  \right)$ satisfies
 \[
  \deg \wt{\co}_{k}\!\left(
  \desu^{\ce}_{\star}(\bul)
  \right) \leq \max\{-1, ak+b\} \, .
 \]
 \item For every object $x$ in $\ce$, the semi-simplicial set $\desu_{\star}^{\ce}(\bul)_{x}$ is
\[
  \text{$\floor*{\frac{\rk(x)-b-1}{a}}$-acyclic.}
\] 
\end{birki}
\end{defn}

\begin{thm} \label{VIC-H3}
 Let $R$ be a ring. Then the monoidal category $\VIC(R)$ from Definition \ref{su-pbc-defn} equipped with $R$-shifting and the function
\begin{align*}
 \rk \colon \Iso(\VIC(R)) &\to \zz_{\geq 0}
 \\
 V &\mapsto \rank_{R}(V) 
\end{align*}
 satisfies \(
\begin{cases}
 \hth{2,s+1} & \text{if $\sr(R) \leq s$,}
 \\
 \hth{2,2} & \text{if $R$ is a Euclidean domain.}
\end{cases}
 \)
\end{thm}
\begin{proof}
 Note that every object $V$ in $\VIC(R)$ is isomorphic to $R^{n}$ for some $n \in \zz_{\geq 0}$, so the semi-simplicial set
 \[
  \Omega_{\star}^{\VIC(R)}(\bul)_{V} = \PBC_{\star}(V)
  \cong \PBC_{\star}(R^{n}) \quad \text{(Remark \ref{omega-PBC})}
 \]
is 
\[
\begin{dcases*}
 \floor*{\frac{\rk(V) - (s+1) - 1}{2}}
 \text{-connected} & \text{if  $\sr(R) \leq s$,}
 \vspace{0.1cm}
 \\
 \floor*{\frac{\rk(V) - 2 - 1}{2}} 
 \text{-connected} & \text{if $R$ is a Euclidean domain,}
\end{dcases*}
\]
 by Corollary \ref{pbc-conn} and Corollary \ref{PBC-Euc-conn}.
\end{proof}

\begin{thm} \label{VIC-fU-H3}
 Let $R$ be a commutative ring with $\sr(R) \leq s$ and let $\fU \leq R^{\times}$ be a subgroup. Then the monoidal category $\VIC(R,\fU)$ from Definition \ref{pbc-U} equipped with $R$-shifting and the function $\rk(R^{n}) \coloneqq n$
 satisfies $\hth{2,s+1}$.
\end{thm}
\begin{proof}
  For every object $R^{n}$ in $\VIC(R)$, the semi-simplicial set
 \[
  \Omega_{\star}^{\VIC(R,\fU)}(\bul)_{R^{n}} = \PBC^{\fU}_{\star}(R^{n}) \quad \text{(Remark \ref{omega-PBC-fU})}
 \]
is \(\floor*{\frac{\rk(V) - (s+1) - 1}{2}} = \floor*{\frac{n-s-2}{2}}\)-connected by Corollary \ref{PBC-fU-conn}.
\end{proof}

\begin{thm} \label{SI-H3}
 Let $R$ be a Euclidean domain. Then the monoidal category $\SI(R)$ from Definition \ref{hu-defn} equipped with $H_{R}$-shifting and the function
 \begin{align*}
 \rk \colon \Iso(\SI(R)) &\to \zz_{\geq 0}
 \\
 V &\mapsto \frac{\rank_{R}(V)}{2}
\end{align*}
 satisfies $\hth{2,2}$.
\end{thm}
\begin{proof}
 Note that every object $V$ in $\SI(R)$ is isomorphic to a finite direct sum of $Y \coloneqq H_{R}$'s by \cite[Corollary 4.1.2]{knus-forms}, so $\rk$ is well-defined as $\rank_{R}(Y) = 2$. In other words, for every object $V$ in $\SI(R)$, there exists $g \in \zz_{\geq 0}$ such that $V \cong Y^{g}$; hence the semi-simplicial set
\[
 \Omega_{\star}^{\SI(R)}(\bul)_{V} = \HU_{\star}(V) \cong \HU_{\star}(Y^{g}) \quad \text{(Remark \ref{omega-HU})}
\]
is
\(
 \floor*{\frac{\rk(V) - 2 - 1}{2}} = \floor*{\frac{g-3}{2}}
\)-connected by Theorem \ref{hu-conn}.
\end{proof}

So far we have only applied the $\desu_{\star}^{\ce}$ construction to the constant $\ce$-set $\bul$ at a singleton. Note that for every abelian category $\ab$, there is a \textbf{Moore complex} construction
\[
 \mathbf{M} \colon \wt{\ses}\ab \to \Ch_{\geq -1}(\ab)
\]
that turns an augmented semi-simplicial object in $\ab$ into a chain complex in $\ab$ supported in degrees $\geq -1$ via setting $\mathbf{M}(Y_{\star})_{i} \coloneqq Y_{i}$ and the differential given by an alternating sum of the face maps.

\begin{rem} \label{augment-reduced}
 Given $X_{\star} \colon \OI_{+}^{\opp} \to \Set$, in other words a semi-simplicial set $X_{\star}$, form the augmented semi-simplicial set $\wt{X}_{\star}$ via
\begin{align*}
 \wt{X}_{\star} \colon \OI^{\opp} &\to \Set
 \\
 A &\mapsto \wt{X}_{A} \coloneqq
\begin{cases}
 X_{A} & \text{if $A \neq \empt$,}
 \\
 \{*\} & \text{if $A = \empt$,}
\end{cases}
\end{align*}
so the induced map for the ordered injection $\empt \to A$ for any linearly ordered set $A$ is the unique map $\wt{X}_{A} \to \{*\}$. Then the Moore complex \[\mathbf{M}\!\big(\zz [\wt{X}_{\star}]\big) \, ,\]
where $\zz[-]$ is the free abelian group functor, is precisely the augmented cellular chain complex of the CW-complex $|X_{\star}|$. In particular, for every integer $i \geq -1$ we have
\[\co_{i}\!\left(
\mathbf{M}\!\big(\zz [\wt{X}_{\star}]\big)\right) = \wt{\co}_{i}\big(X_{\star}\big)
\]
where the $\wt{\co}_{i}$ on the right denotes the $i$-th reduced homology of  a space.
\end{rem}

\begin{defn} \label{C-homology}
 Let $\ce$ be a category with a small skeleton equipped with a shifting context $(\sigma, \eta)$, and $\ab$ be an abelian category. For every integer $i \geq -1$, we write $\wt{\co}_{i}^{\ce}$ for the composite
 \[
  [\ce,\ab] \xrightarrow{\wt{\desu}^{\ce}_{\star}} [\ce,\wt{\ses}\ab]
  \xrightarrow{\mathbf{M} \,\circ\, -} \left[\ce, \Ch_{\geq -1}(\ab)\right]
  \xrightarrow{\co_{i}\circ\,-} [\ce,\ab]
 \]
and call it the \textbf{$i$-th $\ce$-homology}.
\end{defn}

\begin{rem} \label{patzt-Cx}
 Suppose $\pars*{\me,\oplus,0}$ and $x$ satisfy the setting of Definition \ref{x-shift} so that $\me$ is equipped with $x$-shifting. For every functor $V \colon \me \to \ab$ where $\ab$ is an abelian category, recalling Remark \ref{patzt-Kx}, the functor
 \[
  \mathbf{M}\pars*{\wt{\desu}^{\me}_{\star}(V)} \colon \me \to \Ch_{\geq -1}(\ab)
 \]
 is denoted $\wt{C}^{x}_{\ast}(V)$ in \cite[Definition 2.5]{patzt-central}. 
\end{rem}

\begin{rem}
 With $\ce$ as in Definition \ref{C-homology}, for every $\ce$-set $\Lambda$ the $\ce$-module $\zz[\Lambda]$ satisfies
\begin{align*}
   \wt{\co}^{\ce}_{i}(\zz[\Lambda]) 
   &= \co_{i}\!\left(
   \mathbf{M}\big(\wt{\desu}^{\ce}_{\star}(\zz[\Lambda])\big)
   \right)
   \\
   &= \co_{i}\!\left(
   \mathbf{M}\big(\zz\big[
   \wt{\desu}^{\ce}_{\star}(\Lambda)\big]\big)
   \right)
\end{align*}
because the $\wt{\desu}^{\ce}_{\star}$-construction in Section \ref{sec-omega} is constructed by composites of the left adjoint functor $\desu$, which will commute with $\zz[-]$, another left adjoint. 
As $\wt{\desu}^{\ce}_{-1}(\bul) = \bul$ for the constant $\ce$-set $\bul$ at a singleton, we have $\wt{\desu}^{\ce}_{\star}(\bul) = \wt{\desu^{\ce}_{\star}(\bul)}$ with respect to the augmenting operation in Remark \ref{augment-reduced}, so
\begin{align*}
 \wt{\co}^{\ce}_{i}(\zz[\bul]) = \wt{\co}_{i}\!\left(
 \desu^{\ce}_{\star}(\bul)
 \right)
\end{align*}
as a $\ce$-module by Remark \ref{augment-reduced}.
\end{rem}

\section{Patzt's stability framework}

A \textbf{stability groupoid} (originally defined in \cite[Definition 3.1]{patzt-central}) is a skeletal monoidal groupoid $(\GG,\oplus,0)$ with $\Obj(\GG) = \zz_{\geq 0}$ and the bifunctor $\oplus$ on objects being precisely addition in $\zz_{\geq 0}$, such that writing $\GG_{n} \coloneqq \Aut_{\GG}(n)$ the following hold:
\begin{birki}
 \item For every $m,n \in \zz_{\geq 0}$, the monoidal structure map
 \[
  \oplus_{m,n} \colon \GG_{m} \times \GG_{n} \to \GG_{m+n}
 \]
 is injective.
 \item $\GG_{0}$ is the trivial group.
 \item For every $l,m,n \in \zz_{\geq 0}$, the diagram
 \[
  \xymatrixcolsep{2cm}
  \xymatrixrowsep{1.1cm}  
  \xymatrix{
  1 \times \GG_{m} \times 1 \ar[r]^-{\oplus_{l,m}\, \times \, 1}
  \ar[d]_-{1 \, \times \,\oplus_{m,n}} & \GG_{l+m} \times 1
  \ar[d]^-{\oplus_{l+m,n}}
  \\
  1 \times \GG_{m+n} \ar[r]^-{\oplus_{l,m+n}} & \GG_{l + m + n}
  }
 \]
 of groups is a pullback.
\end{birki}

\begin{defn} \label{braided-lift}
 Let $\pi \colon G \to G^{\wr}$ be a map of $\atsi$-groups. A braided monoidal functor
\[
 \mathbf{F} = (F,J,\psi) \colon \pars*{\GG, \oplus, 0, \tau} \to 
 \pars*{\GG^{\wr}, \oplus^{\wr}, 0^{\wr}, \tau^{\wr}}
\] 
between braided monoidal groupoids as in Definition \ref{bm-functor} is called a \textbf{braided lift} of $\pi$ if the following hold:
\begin{birki}
\item $\pars*{\GG,\oplus,0}$ and $\pars*{\GG^{\wr},\oplus^{\wr},0}$ are stability groupoids with $F(1) = 1$.
\item In the sense of Definition \ref{monoidal-to-A-grp}, we have $G = \GG(0,1)$ and $G^{\wr} =\GG^{\wr}(0,1)$.
\item In the sense of Remark \ref{Gcx-maps}, $\pi = \mathbf{F}[\id_{0},\id_{1}]$.
\end{birki}
The braided monoidal groupoid $(\GG,\oplus,0, \tau)$ is called a \textbf{braided lift} of the $\atsi$-group $G$ if $\id_{\GG}$ is a braided lift of the identity map $\id_{G} \colon G \to G$ of $\atsi$-groups.
\end{defn}

\begin{rem}
 To every monoidal groupoid $\GG$ one can associate a category $U\GG$ so that when $\GG$ is a stability groupoid, the associated $U\GG$ is called the ``stability category'' in \cite[Section 3]{patzt-central} and occupies an important place in the theory. The $\GG \mapsto U\GG$ construction is treated in some depth and generality in Appendix \ref{UG-ulan}.  
\end{rem}

\begin{ex} \label{lift-Sym}
 The sequence of finite symmetric groups $(\sym{n} : n\in \zz_{\geq 0})$ can be realized as the automorphism groups of a stability groupoid $\sym{\!}$ where the monoidal operation is the disjoint union $\sqcup$ together with the identifications
\begin{align*}
  m \sqcup n &= \{1,\dots,m\} \sqcup \{1,\dots,n\}
  \\
  &\cong \{1,\dots,m,m+1,\dots,m+n\} = m+n\, .
\end{align*}
Swaps of the form $\sigma \sqcup \sigma' \mapsto \sigma' \sqcup \sigma$ define a symmetric braiding on $\sym{\!}$. We note the following: 
\begin{birki}
 \item In the sense of Definition \ref{braided-lift}, this stability groupoid $\sym{\!}$ is a braided lift of the $\atsi$-group $\sym{\!}$ of Example \ref{symmetric-A} (admittedly with an abuse of notation).
 \item The associated category $U \sym{\!}$ from Definition \ref{U-const}, with the monoidal $\oplus$ of Theorem \ref{enhance} part (1), is monoidally equivalent to the category $\FI$ of finite sets and injections with the disjoint union via the assignment $1 \mapsto \{*\}$. We investigate $\FI$-modules further in Section \ref{poly-FI}. 
\end{birki}
\end{ex}

\begin{ex} \label{lift-AF}
 The sequence of free groups $(F_{n} : n\in \zz_{\geq 0})$ of finite rank, as defined in Section \ref{torelli-AF} can be realized as the automorphism groups of a stability groupoid $\cauf$ where the monoidal operation is the free product $\ast$ together with the identifications
\begin{align*}
  F_{m} \ast F_{n} &= F_{\{x_{1}, \dots, x_{m}\}} \ast F_{\kume*{x_{1}, \dots ,x_{n}}}
  \\
  &\cong
  F_{\kume*{x_{1}, \dots, x_{m}}} \ast F_{\kume*{x_{m+1}, \dots, x_{m+n}}}
  \\ 
  &\cong F_{\kume*{x_{1}, \dots, x_{m}} \sqcup \kume*{x_{m+1}, \dots, x_{m+n}}} = F_{m+n} \, .
\end{align*}
Swaps of the form $a \ast b \mapsto b \ast a$ define a symmetric braiding on $\cauf$. We note the following:
\begin{birki}
 \item The inclusion of $\cauf$ in $f\GG_{\text{free}}$ of \cite[Section 5.2]{rw-wahl-stab} is a braided monoidal equivalence.
 \item The braided (in fact symmetric) stability groupoid $\cauf$ is denoted $\operatorname{AutF}$ in \cite[Example 1.1, part (d)]{patzt-central}.
 \item In the sense of Definition \ref{braided-lift}, $\cauf$ is a braided lift of the $\atsi$-group $\auf$ from Section \ref{torelli-AF}.
\end{birki}
\end{ex}

\begin{ex} \label{lift-GL}
 Fix a ring $R$.  The sequence of general linear groups $\pars*{\GL_{n}(R) : n\in \zz_{\geq 0}}$ can be realized as the automorphism groups of a stability groupoid $\cGL(R)$ where the monoidal operation is the block-sum
\[
  A \oplus B \coloneqq \matr{A & 0 \\ 0 & B}
\]
of matrices, which is (symmetrically) braided via swaps of the form
\[
 \matr{A & 0 \\ 0 & B} \mapsto \matr{B & 0 \\ 0 & A} \, .
\]
We note the following:
\begin{birki}
 \item The inclusion of $\cGL(R)$ in $fR\mhyphen\operatorname{Mod}$ of \cite[Section 5.3]{rw-wahl-stab} is a braided monoidal equivalence onto the full subgroupoid of $fR\mhyphen\!\operatorname{Mod}$ consisting of free $R$-modules.
 \item In the sense of Definition \ref{braided-lift}, $\cGL(R)$ is a braided lift of the $\atsi$-group $\GL(R)$ from Example \ref{A-grp-GL}.
 \item The associated category $U \cGL(R)$ from Definition \ref{U-const}, with the monoidal $\oplus$ of Theorem \ref{enhance} part (1), is monoidally equivalent to $\VIC(R)$ of Definition \ref{su-pbc-defn} via the assignment $1 \mapsto R$.
\end{birki}
\end{ex}

\begin{ex} \label{lift-Mod}
 The sequence of mapping class groups $\pars*{\Mod\pars*{\surf^{1}_{g}} : g \in \zz_{\geq 0}}$ of orientable surfaces with one boundary component, as defined in Section \ref{torelli-MCG}, can be realized as the automorphism groups of a stability groupoid $\cMod\pars*{\surf^{1}}$ by declaring $\cMod\pars*{\surf^{1}}$ to be the braided monoidal groupoid denoted $\mathcal{M}_{2}^{+}$ in \cite[Section 1.1.2.2]{palmer-soulie-poly}. We note the following: 
\begin{birki}
 \item As indicated in \cite[page 5777]{palmer-soulie-poly}, the inclusion of $\cMod\pars*{\surf^{1}}$ in $\mathcal{M}_{2}$ of \cite[Section 5.6]{rw-wahl-stab} is a braided monoidal equivalence onto the full subgroupoid of $\mathcal{M}_{2}$  consisting of orientable surfaces.
 \item  In the sense of Definition \ref{braided-lift}, $\cMod\pars*{\surf^{1}}$ is a braided lift of the $\atsi$-group $\Mod\pars*{\surf^{1}}$ from Section \ref{torelli-MCG}.
\end{birki}
\end{ex}

\begin{ex} \label{lift-Sp}
 Fix a commutative ring $R$.  The sequence of symplectic groups \[\pars*{\Sp_{2g}(R) : g \in \zz_{\geq 0}}\,,\] as defined in Section \ref{torelli-MCG} can be realized as the automorphism groups of a stability groupoid $\cSp(R)$ where the monoidal operation is the block-sum
\[
  A \oplus B \coloneqq \matr{A & 0 \\ 0 & B}
\]
of matrices, which is (symmetrically) braided via swaps of the form
\[
 \matr{A & 0 \\ 0 & B} \mapsto \matr{B & 0 \\ 0 & A}
\]
 define a symmetric braiding on $\cSp(R)$. We note the following: 
\begin{birki}
 \item The inclusion of $\cSp(R)$ in \(f(R,-1,R)\mhyphen\!\operatorname{Quad}\) 
of \cite[Section 5.4]{rw-wahl-stab} is a braided monoidal equivalence onto the full subgroupoid of \(f(R,-1,R)\mhyphen\!\operatorname{Quad}\)  consisting of direct sums of the symplectic $R$-module $H_{R}$ from Definition \ref{hu-defn}.
 \item In the sense of Definition \ref{braided-lift}, $\cSp(R)$ is a braided lift of the $\atsi$-group $\Sp(R)$ from Section \ref{torelli-MCG}.
 \item The associated category $U \cSp(R)$ from Definition \ref{U-const}, with the monoidal $\oplus$ of Theorem \ref{enhance} part (1), is monoidally equivalent to $\SI(R)$ of Definition \ref{hu-defn} via the assignment $1 \mapsto H_{R}$.
\end{birki}
\end{ex}

\begin{ex} \label{lift-SL-fU}
 Fix a commutative ring $R$ and a subgroup $\fU \leq R^{\times}$.  The sequence of groups $\pars*{\SL^{\fU}_{n}(R) : n\in \zz_{\geq 0}}$, as defined in Section \ref{section:cong-GL} can be realized as the automorphism groups of a stability groupoid $\cSL^{\fU}(R)$ where the monoidal operation is the block-sum
\[
  A \oplus B \coloneqq \matr{A & 0 \\ 0 & B}
\]
of matrices, which, in case $-1 \in \fU$, is (symmetrically) braided via swaps of the form
\[
 \matr{A & 0 \\ 0 & B} \mapsto \matr{B & 0 \\ 0 & A} \, .
\]
We note the following:
\begin{birki}
 \item The stability groupoid $\cSL^{\fU}(R)$ is denoted $\GL^{\fU}(R)$ in \cite[Example 3.6]{mpw-torelli-H2}.
 \item Suppose $-1 \in \fU$. Then in the sense of Definition \ref{braided-lift}, $\cSL^{\fU}(R)$ is a braided lift of the $\atsi$-group $\SL^{\fU}(R)$ from Section \ref{section:cong-GL}.
 \item The associated category $U \cSL^{\fU}(R)$ from Definition \ref{U-const} is equivalent to $\VIC(R,\fU)$ of Definition \ref{pbc-U} via the assignment $1 \mapsto R$. When $-1 \in \fU$, this becomes a monoidal equivalence under the monoidal $\oplus$ of Theorem \ref{enhance} part (1).
\end{birki}
\end{ex}

Having lifted the $\atsi$-groups from the introduction, we note that the relevant maps between them have braided lifts as well, with the underlying functors being \textbf{full}:

\begin{prop} \label{lift-AF-to-GL}
 In the sense of Definition \ref{braided-lift}, the surjective map $\auf \to \GL(\zz)$ of $\atsi$-groups from Section \ref{torelli-AF} has a full braided lift of the form
 \[ 
   \cauf \to \cGL(\zz)
 \]
with respect to the braided stability groupoids defined in Example \ref{lift-AF} and Example \ref{lift-GL}.
\end{prop}

\begin{prop} \label{lift-Mod-to-Sp}
 In the sense of Definition \ref{braided-lift}, the surjective map $\Mod\pars*{\surf^{1}} \to \Sp(\zz)$ of $\atsi$-groups from Section \ref{torelli-MCG} has a full braided lift of the form
 \[ 
   \cMod\pars*{\surf^{1}} \to \cSp(\zz)
 \]
with respect to the braided stability groupoids defined in Example \ref{lift-Mod} and Example \ref{lift-Sp}.
\end{prop}

\begin{prop} \label{lift-GL-to-SLfU}
  Let $I$ be a proper ideal in a commutative ring $R$ such that the mod-$I$ reduction  
$\SL_{n}(R) \rarr \SL_{n}(R/I)$ is surjective for every $n \geq 0$. Then setting 
\begin{align*}
  \fU := \{x + I : x \in R^{\times}\} \, ,
\end{align*}
the surjective map $\GL(R) \to \SL^{\fU}(R/I)$ of $\atsi$-groups from Remark \ref{mod-I-surj} has a full braided lift of the form
\[
 \cGL(R) \to \cSL^{\fU}(R/I)
\]
with respect to the braided stability groupoids defined in Example \ref{lift-GL} and Example \ref{lift-SL-fU}.
\end{prop}

\begin{defn} \label{braided-lift-module}
 Let $G$ be an $\atsi$-group and $V$ be an $\atsi_{G}$-module. A pair $\pars*{\GG,\wh{V}}$ is called a \textbf{braided lift of} $V$ under the following conditions: 
\begin{birki}
 \item $\GG = \pars*{\GG,\oplus,0, \tau}$ is a braided lift of $G$ as in Definition \ref{braided-lift}.
 \item The functor $T_{\GG,0,1} \colon \atsi_{G} \to U\GG^{\circ}$ from Proposition \ref{A-to-U} and the functor 
 \[
  U_{\pmb{\tau}} \colon U\GG^{\circ} \to U\GG
 \] 
 induced by the (braided) monoidal functor
 \[
  \pmb{\tau} = \pars*{\id_{\GG},\tau,\id_{0}} \colon \GG^{\circ} \to \GG
 \]
 from Proposition \ref{exer} via Proposition \ref{UF-oldu} make the diagram
\[
 \xymatrixcolsep{1.5cm}
 \xymatrixrowsep{1.1cm} 
 \xymatrix{
 \atsi_{G} \ar[r]^-{T_{\GG,0,1}} \ar[drr]_-{V} & 
 U\GG^{\circ} 
 \ar[r]^-{U_{\pmb{\tau}}} & U\GG \ar[d]^-{\wh{V}}
 \\
 & & \lMod{\zz} 
 }
\]  
commute up to natural isomorphism. 
\end{birki}
\end{defn}

\begin{prop}\label{Hk-lift}
  Let $\pi \colon G \to G^{\wr}$ be a surjective map of $\atsi$-groups with a full braided lift
\[
 \mathbf{F} = (F,J,\psi) \colon \pars*{\GG, \oplus, 0, \tau} \to 
 \pars*{\GG^{\wr}, \oplus^{\wr}, 0, \tau^{\wr}} \, .
\] 
Then writing $K \coloneqq \ker \pi$, for every $k \in \zz_{\geq 0}$ the $\atsi_{G^{\wr}}$-module $\co_{k}(K)$ from Proposition \ref{factor-AQ} has a unique braided lift of the form
\[
 \pars*{\GG^{\wr}, \co_{k}(\K)}
\]
up to natural isomorphism.
\end{prop}
\begin{proof}
 Since $\psi \colon F(0) \to 0$ is an isomorphism by Definition \ref{monoidal-F} and $\GG^{\wr}$ is skeletal, we have $F(0) = 0$. We have $F(1) = 1$ by Definition \ref{braided-lift}. Moreover for every $n \in \zz_{\geq 2}$, the morphism 
 \[
  J_{n-1,1} \colon F(n-1) \oplus^{\wr} F(1) \to F\pars*{\pars*{n-1} \oplus 1} = F(n)
 \]
 in $\GG^{\wr}$ is an isomorphism by Definition \ref{monoidal-F}, so again $\GG^{\wr}$ being skeletal forces
 \[
  F(n-1) \oplus^{\wr} F(1) = F(n) \in \Obj(\GG^{\wr}) \, .
 \]
Since $\oplus^{\wr}$ is also addition on $\Obj(\GG^{\wr}) = \zz_{\geq 0}$, we can induct on $n$ to conclude $F(n) = n$ for every $n \in \zz_{\geq 0}$. In other words, $F$ is the identity on objects.

As $F$ is essentially surjective and full, we can now invoke Theorem \ref{essek}, part (2). The uniqueness claim follows from the proof there.
\end{proof}

\begin{cor} \label{torelli-AF-lift}
 For every $k \in \zz_{\geq 0}$, the $\atsi_{\GL(\zz)}$-module $\co_{k}(\ia)$ from Section \ref{torelli-AF} has a unique braided lift (Definition \ref{braided-lift-module}) of the form
\[
 \pars*{\,\cGL(\zz), \co_{k}(\cia)\,}
\]
up to natural isomorphism, where $\cGL(\zz)$ is from Example \ref{lift-GL}.
\end{cor}
\begin{proof}
 This follows from Proposition \ref{lift-AF-to-GL} and Proposition \ref{Hk-lift}.
\end{proof}

\begin{cor} \label{torelli-MCG-lift}
 For every $k \in \zz_{\geq 0}$, the $\atsi_{\Sp(\zz)}$-module $\co_{k}\pars*{\torel^{1}}$ from Section \ref{torelli-MCG} has a unique braided lift (Definition \ref{braided-lift-module}) of the form
\[
 \pars*{\,\cSp(\zz), \co_{k}\pars*{\ctor^{1}}\,}
\]
up to natural isomorphism, where $\cSp(\zz)$ is from Example \ref{lift-Sp}.
\end{cor}
\begin{proof}
 This follows from Proposition \ref{lift-Mod-to-Sp} and Proposition \ref{Hk-lift}.
\end{proof}

\begin{cor} \label{cong-lift}
 Let $I$ be a proper ideal in a commutative ring $R$ such that the mod-$I$ reduction  
$\SL_{n}(R) \rarr \SL_{n}(R/I)$ is surjective for every $n \geq 0$. Then setting 
\begin{align*}
 \fU := \{x + I : x \in R^{\times}\} \leq (R/I)^{\times} \, ,
\end{align*}
 for every $k \in \zz_{\geq 0}$, the $\atsi_{\SL^{\fU}(R/I)}$-module $\co_{k}\pars*{\GL(R,I)}$ from Section \ref{section:cong-GL} has a has a unique braided lift (Definition \ref{braided-lift-module}) of the form
\[
 \pars*{\,\cSL^{\fU}(R/I),\, \co_{k}\pars*{\cGL(R,I)}\,}
\]
up to natural isomorphism, where $\cSL^{\fU}(R/I)$ is from Example \ref{lift-SL-fU}.
\end{cor}
\begin{proof}
 This follows from Proposition \ref{lift-GL-to-SLfU} and Proposition \ref{Hk-lift}.
\end{proof}

\subsection{Polynomiality and $U\GG$-homology}
\begin{conv} \label{conv-shift}
 Given a braided stability groupoid $\GG$, noting that $0$ is initial in $U\GG$ by \cite[Proposition 1.8.(i)]{rw-wahl-stab}, we may, and do, equip $U\GG$ with $1$-shifting of Definition \ref{x-shift} via the monoidal $\oplus$ in Theorem \ref{enhance} part (1). Using this shifting context and the identity rank function 
 \[ 
  \id_{\zz_{\geq 0}} \colon \Iso(U\GG) = \zz_{\geq 0} \to \zz_{\geq 0} \, ,
 \]
we have polynomial conditions via Definition \ref{poly-defn}. Moreover we have associated $U\GG$-homology functors
 \[
  \wt{\co}_{i}^{U\GG} \colon [U\GG, \lMod{\zz}] \to [U\GG, \lMod{\zz}]
 \]
for every $i \geq -1$ coming from Definition \ref{C-homology}, through which we can ask whether $U\GG$ satisfies $\hth{a,b}$ from Definition \ref{H3-defn}.
\end{conv}

\begin{rem} \label{H-match}
 Let $\GG$ be a braided stability groupoid under Convention \ref{conv-shift}. According to Remark \ref{patzt-Cx}, for every $i \in \zz_{\geq -1}$ and $U\GG$-module $V$, the $U\GG$-module $\wt{\co}_{i}^{U\GG}(V)$ is the $i$-th homology of the chain complex denoted $\wt{C}^{1}_{\ast}(V)$ in \cite[Definition 2.5]{patzt-central}, which is abbreviated as $\wt{C}_{\ast}(V)$ from \cite[page 892]{patzt-central} onwards. The augmented semi-simplicial $U\GG$-module behind $\wt{C}_{\ast}(V)$ (via the Moore complex construction) is denoted 
 \[
   \wt{\desu}^{U\GG}_{\star}(V)
 \]
in this paper, whereas it is denoted $K_{\bul}^{1}V$ in \cite[Definition 2.5]{patzt-central}, which is abbreviated again as $K_{\bul}V$ from \cite[page 892]{patzt-central} onwards.
The description of $K_{\bul} V$ in \cite[Proposition 4.2]{patzt-central}, \cite[Proposition 4.3]{patzt-central} matches precisely with the augmented semi-simplicial $U\GG$-module denoted 
\[
  \wt{C}_{\bul}^{\GG}(V)
\]
in \cite[Definition 3.14]{mpw-torelli-H2}. Consequently, the $U\GG$-module $\wt{\co}_{i}^{U\GG}(V)$ in Convention \ref{conv-shift} is isomorphic to $\wt{\co}_{i}^{\GG}(V) = \wt{\co}_{i}(V)$ of \cite[Definition 3.14]{mpw-torelli-H2} and of \cite[Definition 2.9]{mpp-secondary} (note \cite[Remark 2.10]{mpp-secondary}).
\end{rem}

\begin{rem} \label{restrict-FI}
 Suppose $\GG$ is a \textbf{symmetric} stability groupoid, that is, a braided stability groupoid where the braiding is symmetric as in Definition \ref{bm-cat}. Then the maps in \cite[Remark 2.8.5]{tensor-cat-book} patch to a braided monoidal functor $\sym{\!} \to \GG$, which then induces a monoidal functor
 \[
  U\sym{\!} \to U\GG
 \] 
 via Theorem \ref{enhance}. Since this functor is the identity on objects, it preserves the associated shifting contexts, hence under Convention \ref{conv-shift} a $U\GG$-module $V$ is polynomial of degree $r$ starting at $M$ if and only if it is so as a $U\sym{\!}$-module, if and only if it is so as an $\FI$-module in the setting of Appendix \ref{poly-FI}, recalling Example \ref{lift-Sym}.
\end{rem}

\begin{thm}[{\cite[Theorem 3.11]{mpp-secondary}}] \label{mpp-bound}
 Let $\GG$ be a braided stability groupoid under Convention \ref{conv-shift} such that $U\GG$ satisfies $\hth{a,b}$ with $a \geq 2$. If a $U\GG$-module $L$ is polynomial of degree $r$ starting at $N$, then
 \[
  \deg \wt{\co}^{U\GG}_{i}\!(L) \leq \max\{i+N,\,ai+b+r\}
 \]
for every $i \in \zz_{\geq -1}$. 
\end{thm}

\begin{thm}[{\cite{patzt-central},\cite{mpw-torelli-H2}}] \label{H-kriter}
 Let $G$ be an $\atsi$-group with a braided lift $\GG = \pars*{\GG,\oplus,0, \tau}$. For every $\atsi_{G}$-module $V$ with a braided lift  
 \(\pars*{\GG,\wh{V}}\)
 as in Definition \ref{braided-lift-module}, under Convention \ref{conv-shift} the following hold:
\begin{birki}
 \item $V$ has surjective stability degree $\leq d$ (Definition \ref{surj-stab}) if and only if
 \[
  \deg \wt{\co}^{U\GG}_{-1}\!\pars*{\wh{V}} \leq d \, .
 \]
 \item With respect to the tail-central sequence in $G =\GG(0,1)$ from Proposition \ref{tc-from-mon}, $V$ has central stability degree $\leq d$ (Definition \ref{defn:central}) if and only if
 \[
 \max\!\left\{
 \deg\wt{\co}^{U\GG}_{-1}\!\pars*{\wh{V}},\,
 \deg\wt{\co}^{U\GG}_{0}\!\pars*{\wh{V}}
 \right\} \leq d \, . 
 \] 
\end{birki}
\end{thm}
\begin{proof}
 Matching \cite[Definition 3.10]{mpw-torelli-H2} with Definition \ref{surj-stab} and Definition \ref{defn:central} here, the equivalences follow from \cite[Remark 3.16]{mpw-torelli-H2} through Remark \ref{H-match}.
\end{proof}

\begin{thm} \label{H1-ia}
 Consider the braided stability groupoid $\cGL(\zz)$ from Example \ref{lift-GL}. Under Convention \ref{conv-shift}, the $U\cGL(\zz)$-module 
\[
 \co_{1}\pars*{\cia}
\] 
 described in Corollary \ref{torelli-AF-lift} is polynomial of degree $3$ starting at $0$.
\end{thm}
\begin{proof}
 By Remark \ref{restrict-FI}, it suffices to establish polynomiality of $\co_{1}\pars*{\cia}$ as an $\FI$-module. By \cite[Theorem 4.1]{mpw-torelli-H2}, there exists an $\FB$-module $W$ with $\deg(W) = 3$ such that 
 \[
   \co_{1}(\cia) \cong \induce(W)
 \] 
 (see Appendix \ref{poly-FI} for the notation). Hence we are done by Proposition \ref{induced-poly}.
\end{proof}

\begin{thm} \label{H1-torel}
 Consider the braided stability groupoid $\cSp(\zz)$ from Example \ref{lift-Sp}. Under Convention \ref{conv-shift}, the $U\cSp(\zz)$-module 
\[
 \co_{1}\pars*{\ctor^{1}}
\] 
 described in Corollary \ref{torelli-MCG-lift} is polynomial of degree $3$ starting at $4$.
\end{thm}
\begin{proof}
 By Remark \ref{restrict-FI}, it suffices to establish polynomiality of $\co_{1}\pars*{\ctor^{1}}$ as an $\FI$-module. By \cite[Theorem 4.2]{mpw-torelli-H2}, there exists a finitely generated $\FB$-module $W$ with $\deg(W) = 3$ and an isomorphism 
\begin{align*}
 \co_{1}\pars*{\ctor^{1}}_{\!\geq 3} \cong \induce(W)_{\geq 3}
\end{align*}
(see Appendix \ref{poly-FI} for the notation). The $\FI$-module $\induce(W)$ is polynomial of degree $3$ starting at $0$ by Proposition \ref{induced-poly}. Therefore 
\begin{itemize}
 \item by part (2) of Corollary \ref{poly-truncate}, the $\FI$-module $\induce(W)_{\geq 3}$ is polynomial of degree $3$ starting at $\!\max\{0,4\} = 4$, and
 \item by part (1) of Corollary \ref{poly-truncate}, the $\FI$-module $\co_{1}\pars*{\ctor^{1}}$ is polynomial of degree $3$ starting at $\max\{4,3\} = 4$.
\end{itemize}
\end{proof}

\begin{thm} \label{cong-poly}
Let $I$ be a proper ideal in a commutative ring $R$ and $s \geq 1$ such that 
\begin{itemize}
 \item $
 \sr(R) \leq s  
$, and
 \item the mod-$I$ reduction  
$\SL_{n}(R) \rarr \SL_{n}(R/I)$ is surjective for every $n \geq 0$. 
\end{itemize}
Set 
\( \fU := \{x + I : x \in R^{\times}\} \leq (R/I)^{\times}
\), consider the braided stability groupoid $\cSL^{\fU}(R/I)$ from Example \ref{lift-SL-fU}, and fix $k \in \zz_{\geq 0}$. Under Convention \ref{conv-shift}, the $U\cSL^{\fU}(R/I)$-module 
\[
 \co_{k}\pars*{\cGL(R,I)}
\] 
 described in Corollary \ref{cong-lift} is polynomial of degree $2k$ starting at
\begin{align*}
 \begin{cases}
 0 & \text{if $k=0$,}
 \\
 2s+4 & \text{if $k=1$,}
 \\
 4k+2s+1 & \text{if $k \geq 2$.}
\end{cases}
\end{align*}
\end{thm}
\begin{proof}
 By Remark \ref{restrict-FI}, it suffices to establish polynomiality of $\co_{k}\pars*{\cGL(R,I)}$ as an $\FI$-module. This follows from \cite[Theorem 4.17]{bahran-reg} and Theorem \ref{bah-poly}.
\end{proof}

\subsection{ $\co_{< k}$ is polynomial $\imp$ $\co_{k}$ is centrally stable} \label{section:mpw-method}



We prove the main technical theorem of this paper, Theorem \ref{outsource},  which we shall specialize to obtain the theorems from the introduction.

\begin{thm} \label{outsource}
 Let $\pi \colon G \to G^{\wr}$ be a surjective map of $\atsi$-groups with a braided lift
\[
 \mathbf{F} = (F,J,\psi) \colon \pars*{\GG, \oplus, 0, \tau} 
 \to \pars*{\GG^{\wr}, \oplus^{\wr}, 0^{\wr}, \tau^{\wr}}
\] 
as in Definition \ref{braided-lift} and  write $K \coloneqq \ker \pi$. Suppose $k \geq 1$ is a fixed homological degree and
\[
 \weak_{0},\, \dots,\, \weak_{k-1},\,
 \rho_{0},\, \dots,\, \rho_{k-1} \in \zz_{\geq -1} \, ,
\]
such that under Convention \ref{conv-shift},
\begin{itemize}
 \item $U \GG$ satisfies $\hth{a,b}$ with $a \geq 2$,
 \vspace{0.1cm}
  \item $U \GG^{\wr}$ satisfies $\hth{a^{\wr},b^{\wr}}$ with $a^{\wr} \geq 2$,
  \vspace{0.1cm}
 \item for each $0 \leq q < k$, the braided lift
\(
 \pars*{\GG^{\wr}, \co_{q}(\K)} 
\)
 of the $\atsi_{G^{\wr}}$-module $\co_{q}(K)$ of Proposition \ref{Hk-lift} has the property that the $U\GG^{\wr}$-module $\co_{q}(\K)$ is polynomial of degree $\weak_{q}$ starting at $\rho_{q}+1$.
\end{itemize}
Then setting  
\begin{align*}
 A_{k} &:= \max\!\left(\{a(k-1) + b\} \cup \bigcup_{q=0}^{k-1}\left\{
 \rho_{q} - q + k + 1,\, \weak_{q} + a^{\wr}(k-q) + b^{\wr}
 \right\} 
 \right) \,,
 \vspace{0.1cm}
 \\
 B_{k} &:= \max\!\left(\{ak + b\} \cup \bigcup_{q=0}^{k-1}\left\{
 \rho_{q}-q + k + 2,\, \weak_{q} + a^{\wr}(k+1-q) + b^{\wr}
 \right\} 
 \right) \, ,
\end{align*}
the $\atsi_{G^{\wr}}$-module $\co_{k}(K)$ from Proposition \ref{factor-AQ} has 
\begin{birki}
 \item surjective stability degree $\leq A_{k}$, and
 \item central stability degree $\leq B_{k}$, with respect to the tail-central sequence in $G^{\wr} =\GG^{\wr}(0,1)$ from Proposition \ref{tc-from-mon}.
\end{birki}
\end{thm}
\begin{proof}
 By \cite[Proposition 3.39]{mpw-torelli-H2}, there is a spectral sequence of $U\GG^{\wr}$-modules
\begin{itemize}
 \item with $E^{2}_{p,q} = \wt{\co}_{p}^{U\GG^{\wr}}(\co_{q}(\K))$,
 \item supported in the region $\{(p,q) \in \zz^{2} : p \geq -1,\, q \geq 0\}$,
 \item with $\deg E^{\infty}_{p,q} \leq a(p+q) + b$. 
\end{itemize}
Invoking Theorem \ref{mpp-bound}, whenever $0 \leq q <k$ the $U\GG^{\wr}$-module $\co_{q}(\K)$ satisfies
\begin{align*}
 \deg E^{2}_{p,q} = \deg\wt{\co}_{p}^{U\GG^{\wr}}\big(\!
 \co_{q}(\K)\big) \leq \max\{p + \rho_{q} + 1,\,a^{\wr}p + b^{\wr} + \weak_{q}\}
\end{align*}
for every $p \geq -1$.
 Therefore by Lemma \ref{sseq-deg}
\begin{align*}
 \deg\wt{\co}_{-1}^{U\GG^{\wr}}\!\pars*{\co_{k}(\K)}
 &= \deg E^{2}_{-1,k}
 \\
 &\leq \max\!\left(\{\deg E^{\infty}_{-1,k}\} \cup \left\{
 \deg E^{2}_{k-q,q} : 0\leq q < k
 \right\} 
 \right)
 \\
 &\leq A_{k}
\end{align*}
and
\begin{align*}
 \deg\wt{\co}_{0}^{U\GG^{\wr}}\!\pars*{\co_{k}(\K)}
 &= \deg E^{2}_{0,k}
 \\
 &\leq \max\!\left( \{\deg E^{\infty}_{0,k}\} \cup \left\{
 \deg E^{2}_{k+1-q,q} : 0\leq q < k
 \right\} 
 \right)
 \\
 &\leq B_{k} \, .
\end{align*}
Noting that $A_{k} \leq B_{k}$ from their definitions, we are done by Theorem \ref{H-kriter}.
\end{proof}

\subsection{Proofs of the main theorems} \label{main-proofs}

\begin{proof}[Proof of \textbf{\emph{Theorem \ref{H2-IA}}}]
 We work under Convention \ref{conv-shift}. The braided lift 
 \[ 
   \cauf \to \cGL(\zz)
 \]
 of $\auf \to \GL(\zz)$ from Proposition \ref{lift-AF-to-GL} satisfies the following:
\begin{itemize}
  \item  $U \cauf$ satisfies $\hth{2,2}$ by \cite[Proposition 3.20, part (ii)]{mpw-torelli-H2}.
\vspace{0.1cm}
  \item Recalling part (3) of Example \ref{lift-GL}, $U \cGL(\zz)$ satisfies $\hth{2,2}$ by Theorem \ref{VIC-H3}.
\end{itemize}
For every $k \in \zz_{\geq 0}$, the $\atsi_{\GL(\zz)}$-module $\co_{k}(\ia)$ has a braided lift 
\[
 \pars*{\,\cGL(\zz), \co_{k}(\cia)\,}
\] 
from Corollary \ref{torelli-AF-lift} which satisfy the following: 
\begin{itemize}
\vspace{0.1cm}
 \item The $U \cGL(\zz)$-module $\co_{0}(\cia)$ is constant at $\zz$, hence is polynomial of degree $0$ starting at $0$.
 \vspace{0.1cm}
 \item The $U \cGL(\zz)$-module $\co_{1}(\cia)$ is polynomial of degree $3$ starting at $0$ by Theorem \ref{H1-ia}.
\end{itemize}
Therefore we can apply Theorem \ref{outsource} with $G = \auf$ and $G^{\wr} = \GL(\zz)$ at degree $k=2$, together with the parameters $a=b=a^{\wr}=b^{\wr}=2$ and $\weak_{0} = 0$, $\rho_{0} = -1$, $\weak_{1} = 3$, $\rho_{1} = -1$,  which yields
\begin{align*}
 A_{2} &= \max\!\left(\{4\} \cup \bigcup_{q=0}^{1}\left\{
 (-1) - q + 3,\, 3q + 2(2-q) + 2
 \right\} 
 \right)
 \\
 &=
 \max\!\left(\{4\} \cup \bigcup_{q=0}^{1}\left\{
 2 - q,\, q + 6
 \right\} 
 \right) = 7
\end{align*}
and
\begin{align*}
 B_{2} &= \max\!\left(\{6\} \cup \bigcup_{q=0}^{1}\left\{
 (-1)-q + 4,\, 3q + 2(3-q) + 2
 \right\} 
 \right)
 \\
 &= \max\!\left(\{6\} \cup \bigcup_{q=0}^{1}\left\{
 3-q,\, q + 8
 \right\} 
 \right) = 9 \, .
\end{align*}
\end{proof}
%
%


\begin{proof}[Proof of \textbf{\emph{Theorem \ref{H2-torelli}}}]
We work under Convention \ref{conv-shift}. The braided lift 
 \[ 
   \cMod\pars*{\surf^{1}} \to \cSp(\zz)
 \]
 of $\Mod\pars*{\surf^{1}} \to \Sp(\zz)$ from Proposition \ref{lift-Mod-to-Sp} satisfies the following:
\begin{itemize}
  \item  $U \cMod\pars*{\surf^{1}}$ satisfies $\hth{2,2}$ by \cite[Proposition 3.20, part (ii)]{mpw-torelli-H2}.
\vspace{0.1cm}
  \item Recalling part (3) of Example \ref{lift-Sp}, $U \cSp(\zz)$ satisfies $\hth{2,2}$ by Theorem \ref{SI-H3}.
\end{itemize}
For every $k \in \zz_{\geq 0}$, the $\atsi_{\Sp(\zz)}$-module $\co_{k}(\torel^{1})$ has a braided lift 
\[
 \pars*{\,\cSp(\zz), \co_{k}\pars*{\ctor^{1}}\,}
\]
from Corollary \ref{torelli-MCG-lift} which satisfy the following: 
\begin{itemize}
\vspace{0.1cm}
 \item the $U \cSp(\zz)$-module $\co_{0}(\ctor^{1})$ is constant at $\zz$, hence is polynomial of degree $0$ starting at $0$.
 \vspace{0.1cm}
 \item the $U \cSp(\zz)$-module $\co_{1}(\ctor^{1})$ is polynomial of degree $3$ starting at $4$, by Theorem \ref{H1-torel}.
\end{itemize}
Therefore we can apply Theorem \ref{outsource} with $G = \Mod\pars*{\surf^{1}}$ and $G^{\wr} = \Sp(\zz)$ at degree $k=2$, together with the parameters $a=b=a^{\wr}=b^{\wr}=2$ and $\weak_{0} = 0$,  $\rho_{0} = -1$, $\weak_{1} = \rho_{1} = 3$, which yields 
\begin{align*}
 A_{2} &= \max\!\left(\{4\} \cup \bigcup_{q=0}^{1}\left\{
 \rho_{q} - q + 3,\, 3q + 2(2-q) + 2
 \right\} 
 \right)
 \\
 &= \max\!\left(\{4\} \cup \bigcup_{q=0}^{1}\left\{
 \rho_{q} - q + 3,\, q + 6
 \right\} 
 \right) = \max\{4,2,6,5,7\} = 7
\end{align*}
and 
\begin{align*}
 B_{2} 
 &= \max\!\left(\{6\} \cup \bigcup_{q=0}^{1}\left\{
 \rho_{q}-q + 4,\, 3q + 2(3-q) + 2
 \right\} 
 \right)
 \\
 &= \max\!\left(\{6\} \cup \bigcup_{q=0}^{1}\left\{
 \rho_{q}-q + 4,\, q + 8
 \right\} 
 \right) = \max\{6,3,8,6,9\} = 9 \, .
\end{align*}
\end{proof}

\begin{proof}[Proof of \textbf{\emph{Theorem \ref{congruence-larger-action}}}]
We work under Convention \ref{conv-shift}. The braided lift 
\[
 \cGL(R) \to \cSL^{\fU}(R/I)
\]
 of $\GL(R) \to \SL^{\fU}(R/I)$ from Proposition \ref{lift-GL-to-SLfU} satisfies the following:
\begin{itemize}
 \item Recalling part (3) of Example \ref{lift-GL}, $U \cGL(R)$ satisfies $\hth{2,s+1}$ by Theorem \ref{VIC-H3}.
  \item Noting $\sr(R/I) \leq s$ \cite[Proposition 1.5, part (1)]{lam-crash} and recalling part (3) of Example \ref{lift-SL-fU}, $U\cSL^{\fU}(R/I$) satisfies $\hth{2,s+1}$ by Theorem \ref{VIC-fU-H3}.
\end{itemize}

Fix $k \in \zz_{\geq 1}$. The $\atsi_{\SL^{\fU}(R)}$-module $\co_{k}\pars*{\GL(R,I)}$ has a braided lift 
\[
 \pars*{\,\cSL^{\fU}(R/I),\, \co_{k}\pars*{\cGL(R,I)}\,}
\]
from Corollary \ref{cong-lift} which satisfy that 
\begin{itemize}
 \item for each $0 \leq q < k$, the $U \cSL^{\fU}(R/I)$-module $\co_{q}(\cGL(R,I))$ is polynomial of degree $2q$ starting at
\begin{align*}
 M(q,s) := \begin{cases}
 0 & \text{if $q=0$,}
 \\
 2s+4 & \text{if $q=1$,}
 \\
 4q+2s+1 & \text{if $q \geq 2$,}
\end{cases}
\end{align*}
by Theorem \ref{cong-poly}.
\end{itemize}
 Therefore we can apply Theorem \ref{outsource} with $G = \GL(R)$ and $G^{\wr} = \SL^{\fU}(R/I)$ at, together with the parameters $a=a^{\wr}=2$, $b=b^{\wr}=s+1$, \[\weak_{q} = 2q \quad \text{and} \quad \rho_{q} = M(q,s) - 1\] 
for each $0 \leq q < k$, which yields (noting that $q \mapsto (M(q,s) - q)$ is an increasing sequence)
\begin{align*}
 B_{k} 
 &= \max\!\left(\{2k + s+1\} \cup \bigcup_{q=0}^{k-1}\left\{
 M(q,s)-q + k + 1,\, 2q + 2(k+1-q) + s+1
 \right\} 
 \right)
 \\
 &= \max\!\left(\{2k + s+3\} \cup \left\{
 M(q,s)-q + k + 1 : 0\leq q < k
 \right\} 
 \right)
 \\
 &= 
\begin{cases}
 \max\{s+5,\,M(0,s)+2\} & \text{if $k = 1$,}
 \\
 \max\!\left\{s+7,\,
 M(1,s) + 2
 \right\} & \text{if $k=2$,}
 \\
 \max\!\left\{2k + s+3,\,
 M(k-1,s) + 2
 \right\} & \text{if $k \geq 3$.}
\end{cases}
\\
&= C(k,s) \, .
\end{align*}
\end{proof}

\begin{appendices}

\section{Homological stability ranges for related $\atsi$-groups}
As alluded to in the introduction, all of our applications involve a short exact sequence
\[
 1 \to K \to G \to Q \to 1
\] 
of $\atsi$-groups for which we improve the representation stability ranges for the $\atsi_{Q}$-module $\co_{k}(K)$
for a fixed $k \in \zz_{\geq 0}$. In each one of these, the $\atsi$-groups $G,Q$ exhibit (classical) homological stability. In this appendix we collect some of the known stable ranges of such $\atsi$-groups. Although the ranges here are not directly used in the main body of the paper, they provide a decent point of comparison. Furthermore the last few years have seen a flurry of activity in improving these homological stability ranges, which we make an effort to record.

\begin{defn}
 An $\atsi$-module $V$ is \textbf{stable in the range} $> N$ if whenever $n > N$, the induced map
 \[
  V_{(n,n+1)} \colon V_{n} \to V_{n+1}
 \]
of abelian groups is an isomorphism. 
\end{defn}

\begin{defn} \label{defn:stable-homology}
 Given an $\atsi$-group $G$, for every integer $k \geq 0$ we define 
\begin{align*}
  \mathbf{n}_{k}(G) &\coloneqq \inf \{N \geq -1 : \text{the $\atsi$-module }\co_{k}(G) \text{ is stable in the range $> N$}\} \\
 &\in \{-1,0,1,\dots\} \cup \{\infty\} \, ,
\end{align*}
  We say $G$ \textbf{has stable homology} if $\mathbf{n}_{k}(G) < \infty$ for every $k \geq 0$.
\end{defn}

\begin{rem}
The $\atsi$-group $\mathfrak{S}$ of Example \ref{symmetric-A} has stable homology with 
\[\mathbf{n}_{k}(\mathfrak{S}) \leq 2k-1 \quad \text{\cite[Corollary 6.7]{nakaoka-Sn}.}\]
\end{rem}

\begin{ex} \label{braid-A}
 The assignment $n \mapsto \beta_{n}$ where 
\[
 \beta_{n} \coloneqq \left\langle
 \sigma_{1}, \dots, \sigma_{n-1} : 
\begin{array}{l}
 \sigma_{i}\sigma_{j} = \sigma_{j}\sigma_{i} \text{ for } |i-j| > 1 \, ,
 \\
 \sigma_{i}\sigma_{j}\sigma_{i} = \sigma_{j}\sigma_{i}\sigma_{j} \text{ for } |i-j| = 1 \, .
\end{array}
 \right\rangle
\] 
 is the braid group on $n$ strands, together with the standard inclusions, defines an $\atsi$-group $\beta$ which has stable homology \cite{arnold-braids} with 
 \(\mathbf{n}_{k}(\beta) \leq 2k \) \cite[Proposition 1.5]{hatcher-wahl-mcg-3d}.
\end{ex}

\begin{rem}
 The $\atsi$-group $\auf$ from Section \ref{torelli-AF} has stable homology \cite[Theorem, page 39]{hatcher-autf-stab}. It follows from \cite[Theorem D]{hepworth-edge}, \cite[page 93, Corollary]{cv-outer} and fairly elementary algebraic topology that
\(
 \mathbf{n}_{k}(\auf) \leq 2k
\).
\end{rem}

\begin{rem} 
The $\atsi$-group $\GL(\zz)$ has stable homology \cite[Theorem 3.2]{charney-GL-dedekind} with
 \[\mathbf{n}_{k}(\GL(\zz)) \leq \floor*{\frac{3k+2}{2}} \quad \text{\cite[Theorem C]{kmp-improved-GL}.}\]
\end{rem}

\begin{rem} 
The $\atsi$-group $\Sp(\zz)$ has stable homology \cite[Corollary 4.5]{charney-symp} with
\[
 \mathbf{n}_{k}(\Sp(\zz)) 
 \leq \floor*{\frac{3k+1}{2}} \quad \text{\cite[Theorem A]{sierra-wahl-sp}.}\]
\end{rem}

\begin{rem} \label{A-grp-Mod}
 The $\atsi$-group $\Mod(\surf^{1})$ from Section \ref{torelli-MCG} has stable homology \cite[Theorem 0.1]{harer-stab}, \cite[Theorem 1.9]{ivanov-stab} with 
\[
 \mathbf{n}_{k}(\Mod(\surf^{1})) \leq 
\floor*{\frac{3k+1}{2}} \quad \text{\cite[Theorem 1]{boldsen-improved}.}\]
In fact we have equalities $\mathbf{n}_{1}(\Mod(\surf^{1})) = 2$ by (for instance) \cite[Theorem 5.1]{korkmaz-survey} and $\mathbf{n}_{2}(\Mod(\surf^{1})) = 3$ by combining \cite[Theorem 1.1]{korkmaz-stipsicz}, \cite[Theorem 4.9]{sakasai-lagrangian-mcg}. Moreover for every $q \in \zz_{\geq 2}$, \cite[proof of Corollary 5.14]{gkrw-cells-mcg} yields that
\[
 \co_{2q}\!\left(\Mod(\surf^{1}_{3q-1})\right) \to 
 \co_{2q}\!\left(\Mod(\surf^{1}_{3q})\right)
\] 
is not surjective; hence $\mathbf{n}_{2q}(\Mod(\surf^{1})) \in \{3q-1,\,3q\}$.
\end{rem}

\begin{rem} \label{GL-range}
 For a ring $R$ with $\sr(R) \leq s$, the $\atsi$-group $\GL(R)$ has stable homology \cite[Theorem 4.11]{vdk-GL-stab} with 
\[
\mathbf{n}_{k}(\GL(R)) \leq 
\begin{cases}
  k+s-1 & \text{if $k < s$,}    
  \\
  2k & \text{if $k \geq s$,}
\end{cases}
 \quad \text{\cite[Corollary 8.3]{suslin-stab}.}\]
\end{rem}

\begin{rem}
 For a field $\kk$, we can take $s=1$ in Remark \ref{GL-range} to get $\mathbf{n}_{k}(\GL(\kk)) \leq 2k$. On the other hand, this bound is attained only when $\kk = \kk_{2}$ and $k \in \{1,2,3\}$, by \cite[Theorem 1.1 and the preceding paragraph]{szymik-GL6} in conjunction with \cite[Corollary 2, page 579]{quillen-GL}, considering the following:
\begin{itemize}
 \item $\mathbf{n}_{k}(\GL(\kk_{2})) \leq \floor*{\frac{3(k+1)}{2}}$ by \cite[Theorem 3]{quillen-GL} and \cite[Theorem B]{gkrw-cells-GL}. We also note that \cite[Theorem 1.1]{wang-sharp-slope} yields $\mathbf{n}_{k}(\GL(\kk_{2})) = \floor*{\frac{3(k+1)}{2}}$ when $k \equiv 3,4 \!\pmod{8}$ and $\mathbf{n}_{k}(\GL(\kk_{2})) \geq \floor*{\frac{3k+1}{2}}$ when $k \equiv 1,7 \!\pmod{8}$.
 \vspace{0.1cm}
 \item When $q=p^{r} > 2$ for a prime $p$, we have
 \[
 \mathbf{n}_{k}(\GL(\kk_{q})) \leq 
\begin{cases}
 \floor*{\frac{k-1}{2}} & \text{if $k \leq 2r(p-1) - 7$,}
 \\ 
 k-r(p-1) + 2 & \text{otherwise,}
 \\
\end{cases}
\]
 by \cite[Theorem 3]{quillen-GL} and \cite[Theorem A]{gkrw-cells-GL}. In particular $\mathbf{n}_{k}(\GL(\kk_{4})) \leq k$ (in fact $\mathbf{n}_{2}(\GL(\kk_{4})) = 2$ \cite[Remark 6.2,(1)]{sprehn-wahl-stab}).
 \vspace{0.1cm}
 \item $\mathbf{n}_{k}(\GL(\kk)) \leq k-1$ for an infinite field $\kk$ by \cite[Theorem 3.25(b)]{nesterenko-suslin}. We also note that if furthermore the Milnor $K$-theory $\mathbf{K}^{M}_{k}(\kk) \neq 0$ (for instance when $k \geq 2$ and $\kk$ is uncountable \cite[Proposition 3]{springer-milnor-K}), then $\mathbf{n}_{k}(\GL(\kk)) = k-1$ by \cite[Theorem 3.25(a)]{nesterenko-suslin}. On the other hand if $\kk \in \kume*{\ov{\kk_{p}},\, \ov{\qq}}$, \cite[Corollary 9.14]{gkrw-cells-GL-infinite} yields $\mathbf{n}_{k}(\GL(\kk)) \leq \floor*{\frac{3k}{4}}$.
\end{itemize}
\end{rem}

Finally, it is worth mentioning that the $\hth{a,b}$ condition immediately yields homological stability in the following way:

\begin{thm} \label{groupoid-range}
 Let $G$ be an $\atsi$-group with a braided lift 
 \[\GG = \pars*{\GG,\oplus,0,\tau}\] 
as in Definition \ref{braided-lift}. Suppose $U\GG$ satisfies $\hth{a,b}$ with $a \geq 2$ under Convention \ref{conv-shift}. Then $G$ has stable homology (Definition \ref{defn:stable-homology}) with 
\[
 \mathbf{n}_{k}(G) \leq 
\begin{cases}
 2k+b
 & \text{if $a = 2$,}
 \\
 ak+b-1
 & \text{if $a \geq 3$.}
\end{cases}
\] 
\end{thm}
\begin{proof}
 The constant $U\GG$-module $\underline{\zz}$ is polynomial of degree $0$ starting at $0$, so it ``has polynomial degree $\leq 0$ at ranks $>-1$'' in the sense of \cite[Definition 2.40]{mpp-secondary}. Therefore by \cite[Corollary 4.5]{mpp-secondary} for every $i \geq 0$ we have
\begin{align*}
 \mathbf{n}_{i}(G) \leq \max\{-1 + 2i,\,\max\{ai,2i+1\} + b + 0 - 1\}
 &= \max\{ai+b-1,\, 2i+b\}
 \\
 &= 
\begin{cases}
 2i+b
 & \text{if $a = 2$,}
 \\
 ai+b-1
 & \text{if $a \geq 3$.}
\end{cases}
\end{align*}
\end{proof}

\section{Polynomial $\FI$-modules} \label{poly-FI}

Writing $\{*\}$ for a singleton, we equip the monoidal category $(\FI,\sqcup,\empt)$ of finite sets and injections with $\{*\}$-shifting (Definition \ref{x-shift}) and the function 
\begin{align*}
 \rk \colon \Iso(\FI) &\to \zz_{\geq 0}
 \\
 S &\mapsto |S| \, .
\end{align*}
As $\FI$ has a small skeleton and $\lMod{\zz}$ has a zero object, the notions in Section \ref{categorical} are defined for $\FI$-modules. We shortly write $\lMod{\FI} \coloneqq [\FI, \lMod{\zz}]$ for the category of $\FI$-modules.

The main objectives of this appendix is to recast some of the homological interpretations from \cite{bahran-polynomial} and obtain Corollary \ref{poly-truncate}.

\subsection{Relationship with homological invariants of $\FI$-modules}
We shall write $\FB$ for the category of finite sets and \textbf{bijections} (so that the category of $\FB$-modules is $\lMod{\FB} \coloneqq [\FB,\lMod{\zz}]$), and
\[\induce \colon \lMod{\FB} \to \lMod{\FI}\] 
for the left adjoint of the restriction functor $\Res_{\FB}^{\FI} \colon \lMod{\FI} \rarr \lMod{\FB}$.

\begin{prop} \label{induced-poly}
 Let $r \geq -1$ be an integer and $W$ be an $\FB$-module with $\deg W \leq r$. Then the induced $\FI$-module $\induce(W)$ is of polynomial of degree $r$ starting at $0$.
\end{prop}
\begin{proof}
 It can be checked from the description \cite[Definition 2.2.2]{cef} that 
 \[
  \kert\!\left(\induce(W)\right) = 0 \, ,
 \] 
 and by \cite[Lemma 4.4]{ce-homology} we have 
 \[
  \deriv\!\left(\induce(W)\right) \cong 
  \induce\!\left(\shift{W}{}\right) \, .
 \]
 We can now prove the claim for $V \coloneqq \induce(W)$ by induction on $r$:
\begin{itemize}
 \item If $r=-1$, then $W=0$, so $V = 0$ is polynomial of degree $-1$ starting at 0.
 \item If $r \geq 0$, then $\kert V = 0$ is polynomial of degree $-1$ starting at $0$, and since $\deriv V \cong \induce\!\left(\shift{W}{}\right)$ with $\deg \shift{W}{} \leq r-1$, by the induction hypothesis $\deriv V$ is polynomial of degree $r-1$ starting at $0 = \max\{0,-1\}$.
\end{itemize}
Thus $V$ is polynomial of degree $r$ starting at $0$ by Definition \ref{poly-defn}.
\end{proof}

\paragraph{Local cohomology and local degree.}
 An $\FI$-module $M$ is \textbf{torsion} if for every finite set $S$ and $x \in M_{S}$, there exists an injection $\alpha \colon S \emb T$ such that $M_{\alpha}(x) = 0 \in M_{T}$. We write
\begin{align*}
 \locoh{0} \colon \lMod{\FI} \rarr \lMod{\FI}
\end{align*}
for the functor which assigns an $\FI$-module its largest torsion $\FI$-submodule.
 The functor $\locoh{0}$ is left exact and for each $j \geq 1$, we write $\locoh{j} := \operatorname{R}^{j}\!\locoh{0}$ for the $j$-th right derived functor of $\locoh{0}$. Next, for every $j \geq 0$ we write 
\begin{align*}
 h^{j}(V) &:= \deg \locoh{j}(V) 
 \\
 &\in \{-1,0,1,\dots\} \cup \{\infty\}
\end{align*}
for every $\FI$-module $V$. 

\paragraph{Stable degree.} For an $\FI$-module $V$, we set
\begin{align*}
 \weak(V) &:= \min\{r \geq -1 : \deriv^{r+1}(V) \text{ is torsion}\}
 \\
 &\in \{-1,0,1,\dots\} \cup \{\infty\} \, ,
\end{align*}
and call it the \textbf{stable degree} of $V$.

\begin{thm}\label{bah-poly}
 For an $\FI$-module $V$ and integers $r \geq -1$, $M \geq 0$, the following are equivalent:
\begin{birki}
 \item $V$ is polynomial of degree $r$ starting at $M$.
 \item $\weak(V) \leq r$, and for every $j \geq 0$ with $h^{j}(V) \geq 0$ we have $h^{j}(V) + j \leq M-1$.
\end{birki}
\end{thm} 
\begin{proof}
The characterization in \cite[Theorem B]{bahran-polynomial} involves the invariant $\reg(V)$, which is defined via the $\FI$-homology $\cofi{0}$ and its derived functors (see \cite[page 210]{bahran-polynomial}): it says that $V$ satisfying (1) is equivalent to 
 \[\weak(V) \leq r \quad \text{and} \quad \reg(V) \leq M-1 \, .\]
Thus it suffices to show that $\reg(V) \leq M - 1$ if and only if
\[
 \text{for every $j \geq 0$ with $h^{j}(V) \geq 0$, we have $h^{j}(V) + j \leq M - 1$,}
\]
and this follows from \cite[Theorem 1.1]{nss-regularity}.
\end{proof}

\begin{defn} \label{defn-truncate}
 Let $V$ be an $\FI$-module and $N \in \zz_{\geq 0}$. We write $V_{\geq N}$ for the unique $\FI$-submodule of $V$ with 
\begin{align*}
 (V_{\geq N})_{S} = 
\begin{cases}
 0 & \text{if $|S| < N$,}
 \\
 V_{S} & \text{if $|S| \geq N$.}
\end{cases}
\end{align*} 
\end{defn}

\begin{lem} \label{locoh-truncate}
 Given an $\FI$-module $V$ and $N \in \zz_{\geq 0}$. the following hold:
\begin{birki}
 \item $\weak(V) = \weak(V_{\geq N})$.
 \vspace{0.1cm}
 \item $h^{0}(V_{\geq N}) \leq h^{0}(V) \leq \max\{h^{0}(V_{\geq N}),\,N-1\}$.
 \vspace{0.1cm}
 \item $h^{1}(V) \leq h^{1}(V_{\geq N}) \leq \max\{N-1,\,h^{1}(V)\}$.
 \vspace{0.1cm}
 \item $h^{j}(V) = h^{j}(V_{\geq N})$ for every $j \geq 2$.
\end{birki}
\end{lem}
\begin{proof}
Writing $U \coloneqq V/V_{\geq N}$, we have $\deg U \leq N-1$ so that 
\[
 \deg \deriv U \leq \deg \shift{U}{} \leq \max\{-1, N-2\} < \infty
\]
and inductively $\deg \deriv^{a} U \leq \max\{-1, N-a-1\} < \infty$ for every $a \geq 0$.

On the other hand, by taking $p=1$ in \cite[(18)]{ce-homology} there is a short exact sequence
\begin{align*}
 0 \rarr \deriv\!\left(
 \co_{1}^{\deriv^{\!a-1}}\!(U)
 \right) \rarr \co_{1}^{\deriv^{\!a}}\!(U) 
 \rarr \kert\!\left(
 {\deriv^{\!a-1}}(U) 
 \right)\rarr 0 \, .
\end{align*}
for every $a \geq 1$ (here $\co_{1}^{F}$ denotes the first left derived functor of $F$) so that
\begin{align*}
 \deg \co_{1}^{\deriv^{\!a}}\!(U)
 &= \max\!\left\{
 \deg \deriv\!\left(
 \co_{1}^{\deriv^{\!a-1}}\!(U)
 \right),\,
 \deg \kert\!\left(
 {\deriv^{\!a-1}}(U) 
 \right)
 \right\}
 \\
 &\leq \max\!\left\{
 \deg \shift{\!\left(
 \co_{1}^{\deriv^{\!a-1}}\!(U)
 \right)}{},\,
 \deg 
 {\deriv^{\!a-1}}(U) 
 \right\}
 \\
 &\leq \max\!\left\{-1,\,
 \deg 
 \co_{1}^{\deriv^{\!a-1}}\!(U) - 1,\,
 N-a
 \right\}
\end{align*}
and hence by induction $\deg \co_{1}^{\deriv^{\!a}}\!(U) \leq \max\{-1,\,N-a\} < \infty$ for every $a \geq 0$.

Applying $\deriv^{\!a}$ to the short exact sequence 
\begin{align*}
 0 \rarr V_{\geq N} \rarr V \rarr U \rarr 0 ,\tag{$\star$}
\end{align*}
results in a long exact sequence of the form
\[
 \co_{1}^{\deriv^{\!a}}\!(U) \to
 \deriv^{\!a}\!\left(V_{\geq N}\right) \to \deriv^{\!a}V \to \deriv^{\!a}U \to 0 \, .
\]
Since being torsion is preserved under taking submodules, quotients, extensions and $\co_{1}^{\deriv^{\!a}}\!(U)$, $\deriv^{\!a}U$ are already finite degree, we conclude that $\deriv^{\!a}\!\left(V_{\geq N}\right)$ is torsion if and only if $\deriv^{\!a}V$ is torsion: (1) follows.

Applying $\locoh{0}$ to $(\star)$, using \cite[Corollary 2.3]{bahran-polynomial} the associated long exact sequence yields an exact sequence 
\begin{align*}
 0 \rarr \locoh{0}(V_{\geq N}) \rarr \locoh{0}(V) \rarr U \rarr \locoh{1}(V_{\geq N}) \rarr \locoh{1}(V) \rarr 0
\end{align*}
and an isomorphism $\locoh{j}(V_{\geq N}) \cong \locoh{j}(V)$ for every $j \geq 2$. Now (2), (3), (4) follow by taking degrees.
\end{proof}

\begin{cor} \label{poly-truncate}
 For an $\FI$-module $V$ and integers $r \geq - 1$, $M,N \geq 0$, the following hold: 
\begin{birki}
 \item If $V_{\geq N}$ is polynomial of degree $r$ starting at $M$, then $V$ is polynomial of degree $r$ starting at $\max\{M,N\}$.
 \item If $V$ is polynomial of degree $r$ starting at $M$, then $V_{\geq N}$ is polynomial of degree $r$ starting at $\max\{M,N+1\}$.
\end{birki}
\end{cor}
\begin{proof}
(1) Suppose $V_{\geq N}$ is polynomial of degree $r$ starting at $M$. Then by part (1) of Lemma \ref{locoh-truncate} and Theorem \ref{bah-poly}, we have \[\weak(V) = \weak(V_{\geq N}) \leq r \, .\]
 Next, suppose $j \geq 0$ such that $h^{j}(V) \geq 0$. There are three cases:
\begin{itemize}
 \item $j=0$ and $h^{0}(V_{\geq N}) \geq 0$. Then by part (2) of Lemma \ref{locoh-truncate} and Theorem \ref{bah-poly}, we have
 \[h^{0}(V) \leq \max\{M-1,N-1\} = \max\{M,N\} - 1 \, .\]
 \item $j=0$ and $h^{0}(V_{\geq N}) = -1$. Then by part (2) of Lemma \ref{locoh-truncate} we have \[h^{0}(V) \leq N-1 \, .\]
 \item $j \geq 1$. Then by part (3) and (4) of Lemma \ref{locoh-truncate} we have $h^{j}(V_{\geq N}) \geq h^{j}(V) \geq 0$ so by Theorem \ref{bah-poly} we get
 \[h^{j}(V) + j \leq h^{j}(V_{\geq N}) + j \leq M-1 \, .\]
\end{itemize}
In all cases we have $h^{j}(V) + j \leq \max\{M,N\}-1$. We get the desired conclusion by Theorem \ref{bah-poly}.

(2) Suppose $V$ is polynomial of degree $r$ starting at $M$. Then by part (1) of Lemma \ref{locoh-truncate} and Theorem \ref{bah-poly}, we have \[\weak(V_{\geq N}) = \weak(V) \leq r \, .\]
 Next, suppose $j \geq 0$ such that $h^{j}(V_{\geq N}) \geq 0$. There are three cases:
\begin{itemize}
 \item $j=1$ and $h^{1}(V) \geq 0$. Then by part (3) of Lemma \ref{locoh-truncate} and Theorem \ref{bah-poly}, we have
 \[h^{1}(V_{\geq N}) + 1 \leq \max\{N,h^{1}(V) + 1\} \leq \max\{N,M-1\} = \max\{M,N+1\} - 1\, .\]
 \item $j=1$ and $h^{1}(V) = -1$. Then by part (3) of Lemma \ref{locoh-truncate} we have \[h^{1}(V_{\geq N}) + 1 \leq N \, .\]
 \item $j \neq 1$. Then by part (2) and (4) of Lemma \ref{locoh-truncate} we have $h^{j}(V) \geq h^{j}(V_{\geq N}) \geq 0$ so by Theorem \ref{bah-poly} we get
 \[h^{j}(V_{\geq N}) + j \leq h^{j}(V) + j \leq M-1 \, .\]
\end{itemize}
In all cases we have $h^{j}\pars*{V_{\geq N}} + j \leq \max\{M,N+1\}-1$. We get the desired conclusion by Theorem \ref{bah-poly}.
\end{proof}

\section{Euclidean domains}
The objective of this appendix is to observe that the definition of \emph{Euclidean ring} in \cite{maazen-thesis} (and hence \cite{vdk-loo-symp} since it refers to \cite{maazen-thesis}) is equivalent to what is more or less the standard notion of \emph{Euclidean domain} today found in textbooks. Both Theorem \ref{PBC-Euc-conn} and Theorem \ref{hu-conn} are obtained using \cite{maazen-thesis}, making the correspondence between the different definitions in this Appendix necessary. For the main results Theorem \ref{H2-IA} and Theorem \ref{H2-torelli} of the paper, Theorem \ref{PBC-Euc-conn} and Theorem \ref{hu-conn} are invoked only for the Euclidean domain $\zz$.
\begin{defn}[{\cite[Definition 1.1]{graves-gaussian}}]
 Given a commutative ring $R$, a function 
\begin{align*}
 \phi \colon R - \{0\} \rarr \zz_{\geq 0}
\end{align*}
is \textbf{Euclidean} if for every $b \in R-\{0\}$ and $a \in R$, either $b$ divides $a$ or there exists $r \in R- \{0\}$ such that $\phi(r) < \phi(b)$ and $[a] = [r]$ in the quotient ring $R/b$.
\end{defn}

\begin{defn}[{\cite[page 44, Definition 4.1]{maazen-thesis}}] \label{M-E}
 A commutative ring $R$ is called \textbf{M-Euclidean} if there is a function $\mu \colon R \to \zz_{\geq 0}$ such that the following hold:
\begin{birki}
 \item $\mu(1) = 1$, and $\mu(a) = 0$ if and only if $a = 0$.
 \item If $a,b \neq 0$, then $\mu(ab) \geq \mu(b)$.
 \item There are functions $\kappa, \rho \colon R \times (R-\{0\}) \to R$ with the following properties:
\begin{itemize}
 \item For $a \in R$ and $b \in R-\{0\}$, $a = \kappa(a,b)b + \rho(a,b)$ and $\mu(\rho(a,b)) < \mu(b)$.
 \item If $\mu(a) < \mu(b)$, then $\kappa(a,b) = 0$ and $\rho(a,b) = a$.
\end{itemize}
\end{birki}
\end{defn}

\begin{prop} \label{maazen-euc}
 For a commutative ring $R$, the following are equivalent: 
\begin{birki}
 \item $R$ is an integral domain that has a Euclidean function.
 \item $R$ is M-Euclidean.
\end{birki}
\end{prop}
\begin{proof}
 (1) $\imp$ (2): As stated in \cite[Lemma 1.3]{graves-gaussian} (originally due to Motzkin \cite{motzkin-euclidean}), writing $\pi_{b} : R \rarr R/b$ for the canonical projection for each $b \in R$ and setting
\begin{align*}
 \E(R) &:= \{\phi \colon R - \{0\} \rarr \zz_{\geq 0} : \text{$\phi$ is a Euclidean function}\} \, ,
 \\
 A_{R,0} &:= \{0\} \cup R^{\times} \, ,
 \\
 A_{R,j} &:= A_{R,j-1} \cup \{b \in R : \pi_{b}(A_{R,j-1}) = R/b\} \text{ for every } j \geq 1 \, ,
\end{align*}
we have $R = \bigcup_{j=0}^{\infty} A_{R,j}$, and the function 
\begin{align*}
 \phi_{R} \colon R - \{0\} &\rarr \zz_{\geq 0}
 \\
 r &\mapsto \min\{j : r \in A_{R,j}\}
\end{align*}
belongs to $\E(R)$; moreover $\phi_{R}(r) \leq \phi(r)$ for every $\phi \in \E(R)$ and $r \in R-\{0\}$. We now define 
\begin{align*}
 \mu \colon R &\rarr \zz_{\geq 0}
 \\
 r &\mapsto 
\begin{cases}
 \phi_{R}(r) + 1 & \text{if $r \neq 0$,}
 \\
 0 & \text{if $r = 0$.}
\end{cases}
\end{align*}
We immediately see that $\mu$ satisfies part (1) of Definition \ref{M-E}. Further, for every $a,b \in R$ with $b\neq 0$, the set
\begin{align*}
 S_{a,b} &= \{(q,r) \in R^{2} : a = qb +r \text{ and } \mu(r) < \mu(b)\}
 \\
 &= \{(q,r) \in R^{2} : a = qb +r \text{ and } \mu(r) < \phi_{R}(b) + 1\}
 \\
 &= \{(q,0) \in R^{2} : a = qb\} \cup \{(q,r) \in R^{2} : a = qb +r\,, r\neq0, \text{ and } \phi_{R}(r) < \phi_{R}(b)\}
\end{align*}
is nonempty as $\phi_{R} \in \E(R)$. Therefore invoking the axiom of choice, there exist functions $\wt{\kappa}, \wt{\rho} \colon R \times (R-\{0\}) \rarr R$ such that 
\begin{align*}
 (\wt{\kappa}(a,b),\,\wt{\rho}(a,b)) \in S_{a,b} \quad \text{for every } a \in R ,\,b \in R-\{0\} \, .
\end{align*}
Noting that $\mu(a) < \mu(b) \imp (0,a) \in S_{a,b}$, we can define functions
\begin{align*}
 \kappa(a,b) := 
\begin{cases}
 0 & \text{if $\mu(a) < \mu(b)$,}
 \\
 \wt{\kappa}(a,b) & \text{otherwise,}
\end{cases}
\quad \text{and} \quad
\rho(a,b) := 
\begin{cases}
 a & \text{if $\mu(a) < \mu(b)$,}
 \\
 \wt{\rho}(a,b) & \text{otherwise,}
\end{cases}
\end{align*}
and still have
\begin{align*}
 (\kappa(a,b),\,\rho(a,b)) \in S_{a,b} \quad \text{for every } a \in R ,\,b \in R-\{0\} \, .
\end{align*}
It follows that $\mu$ satisfies part (3) of Definition \ref{M-E}.

Finally we claim that the function 
\begin{align*}
 \phi' \colon R - \{0\} &\rarr \zz_{\geq 0}
 \\
 x &\mapsto \min\{\phi_{R}(xy) : y \in R - \{0\}\}
\end{align*}
belongs to $\E(R)$ (note that $\phi'$ is well-defined because $R$ is a domain). To that end, fix $a,b \in R$ such that $b \neq 0$ and $b \nmid a$. We know that $\phi'(b) = \phi_{R}(bc)$ for some $c \in R-\{0\}$. We have $bc \neq 0$ as $R$ is a domain, and $bc \nmid a$ since $b \nmid a$. Therefore as $\phi_{R} \in \E(R)$, there exist $q,r \in R$ with $r \neq 0$ such that 
\begin{align*}
 a = q(bc) + r = (qc)b + r \quad \text{with $\phi_{R}(r) < \phi_{R}(bc) = \phi'(b)$.}
\end{align*}
But evidently $\phi'(r) \leq \phi_{R}(r)$, so $\phi'(r) < \phi'(b)$, proving our claim. By the minimality of $\phi_{R}$ in $\E(R)$, we conclude that $\phi_{R}(x) \leq \phi'(x)$ for every $x \in R - \{0\}$, which means that $\phi_{R}(x) \leq \phi_{R}(xy)$ and hence $\mu(x) \leq \mu(xy)$ for every $x,y \in R - \{0\}$. This verifies part (2) of Definition \ref{M-E}.

(2) $\imp$ (1): To see $R$ has to be a domain, note that if $a,b \in R-\{0\}$, then $\mu(ab) \geq \mu(b) \geq 1$ by parts (1) and (2) of Definition \ref{M-E}. Hence $ab \neq 0$ by part (1) of Definition \ref{M-E}. It is also evident by part (3) of Definition \ref{M-E} that $\mu|_{R - \{0\}} \in \E(R)$.
\end{proof}

\begin{defn}
 A commutative ring $R$ is called a \textbf{Euclidean domain} if it satisfies one (hence both) of the equivalent conditions in Proposition \ref{maazen-euc}.
\end{defn}

\section{Some spaces associated to rings} \label{specific}
Let us recall that our rings are, by Convention \ref{ring-conv}, nonzero and stably finite. This has some convenient consequences.

\begin{prop} \label{rank-surj}
 Let $R$ be a ring (with Convention \ref{ring-conv}). If $m,n \in \zz_{\geq 0}$ such that there is an $R$-linear surjection
 \[
  R^{m} \to R^{n}
 \]
of free right (or left) $R$-modules, then $m \geq n$.
\end{prop}
\begin{proof}
 This follows from \cite[Proposition 1.22]{lam-lectures}, considering \cite[Definition 1.20]{lam-lectures}.
\end{proof}

\begin{rem} \label{wd-free-rank}
 Thanks to Proposition \ref{rank-surj}, every finitely generated free right (or left) free $R$-module $F$ has a well-defined $\rank_{R}(F) \in \zz_{\geq 0}$ : every pair of $R$-linear isomorphisms
\[
 R^{\oplus X} \cong F \cong R^{\oplus Y}
\] 
for some sets $X,Y$ has to satisfy $|X| = |Y| < \infty$, so we can declare this common finite cardinality as $\rank_{R}(F)$. To see this, as $F$ is finitely generated, there exists $m \in \zz_{\geq 0}$ and a surjection $R^{m} \to R^{\oplus X}$. If $X$ were infinite, we could pick a finite subset $X_{0} \subseteq X$ with $|X_{0}| > m$ and get a surjection $R^{m} \to R^{\oplus X_{0}}$, contradicting Proposition \ref{rank-surj}. Thus $X$ is finite, and similarly $Y$ is finite. Writing $m = |X|$ and $n = |Y|$, we get an isomorphism $R^{m} \cong R^{n}$ so that applying Proposition \ref{rank-surj} twice yields $m \leq n$ and $n \leq m$, so $m=n$.
\end{rem}

Recall that if (P) is a property defined for right $R$-modules (for a ring $R$), a right $R$-module $M$ is said to be \textbf{stably} (P) if for some $a \in \zz_{\geq 0}$ the right $R$-module $M \oplus R^{a}$ is (P).

\begin{rem}
 As a generalization of Remark \ref{wd-free-rank}, every finitely generated free right (or left) \textbf{stably free} $R$-module $M$ has a well-defined $\rank_{R}(M) \in \zz_{\geq 0}$ : every pair of $R$-linear isomorphisms
 \[
  M \oplus R^{a} \cong R^{\oplus X} \quad \text{and} \quad M \oplus R^{b} \cong R^{\oplus Y}
 \]
for some $a,b \in \zz_{\geq 0}$ and sets $X,Y$ has to satisfy $|X|, |Y| < \infty$ and 
\[
 0 \leq |X| - a = |Y| - b \, ,
\]
so we can declare this common non-negative integer as $\rank_{R}(M)$. To see this, as $M \oplus R^{a}$ is finitely generated, an argument similar to the one in Remark \ref{wd-free-rank} yields $|X| < \infty$ and similarly $|Y| < \infty$. Writing $a' \coloneqq |X|$ and $b' \coloneqq |Y|$, we have
\[
 R^{a'+b} \cong R^{a'} \oplus R^{b} \cong M \oplus R^{a} \oplus R^{b} \cong M \oplus R^{b} \oplus R^{a} \cong R^{b'} \oplus R^{a} \cong R^{a + b'} \, ,
\]
so that applying Proposition \ref{rank-surj} twice yields $a' + b = a+b'$, and hence $a' - a = b'-b$. A third application of Proposition \ref{rank-surj} to the surjection $R^{a'} \cong R^{\oplus X} \to R^{a}$ finally yields $a \leq a'$.
\end{rem}

\begin{lem} \label{glr0}
 Let $R$ be a ring (with Convention \ref{ring-conv}). If $M$ is a stably free right $R$-module with $\rank_{R}(M) = 0$, then $M=0$.
\end{lem}
\begin{proof}
 There exists $n \in \zz_{\geq 0}$ with $M \oplus R^{n} \cong R^{n}$. We get $M=0$ by \cite[Proposition 1.7, part (2)]{lam-lectures}.
\end{proof}

\begin{prop} \label{rank-split-inj}
  Let $R$ be a ring (with Convention \ref{ring-conv}). If $m,n \in \zz_{\geq 0}$ such that there is a \textbf{split} $R$-linear injection
 \[
  R^{n} \to R^{m}
 \]
of free right (or left) $R$-modules, then $n \leq m$.
\end{prop}
\begin{proof}
 The splitting is necessarily an $R$-linear surjection $R^{m} \to R^{n}$. Now use Proposition \ref{rank-surj}.
\end{proof}

Given a monoidal category $(\me,\oplus,0)$ where the unit object $0$ is initial with objects $x,y$ in $\me$, there is a semi-simplicial set given by the composite
 \[
  \OI_{+}^{\opp} \xrightarrow{F_{x}^{\opp}} \me^{\opp}
  \xrightarrow{\Hom_{\me}(-,y)} \Set \tag{$\dagger$} \label{sfx}
 \]
where $F_{x}$ is (the restriction of) the functor in Theorem \ref{initial}.

\begin{defn} \label{pb-u-defn}
 For a ring $R$, let $\VI(R)$ be the category of finitely generated free right $R$-modules and \textbf{splittable} $R$-linear maps, where a splittable map is one with a left inverse. Let $\VI'(R)$ be the wide subcategory of $\VI(R)$ given by 
\[
 \Hom_{\VI'(R)}(W,V) \coloneqq \left\{
 \alpha \in \Hom_{\VI(R)}(W,V) : \coker(\alpha) \text{ is $R$-free}
 \right\} \, .
\] 
 
Both $\VI(R)$ and $\VI'(R)$ are monoidal via direct sum $\oplus$ where the unit object $0$ is initial. For every free right $R$-module $V$ of finite rank, taking $x$ to be $R$ and $y$ to be $V$ in (\ref{sfx}), we write $\UuU_{\star}(V)$ and $\PB_{\star}(V)$ for the corresponding semi-simplicial set, respectively for $\VI(R)$ and $\VI'(R)$. More explicitly, they have the description
\begin{align*}
 \UuU_{\star}(V) \colon \OI_{+}^{\opp} &\to \Set \\
 [k] &\mapsto \Hom_{\VI(R)}\!\left(
 R^{\oplus [k]},V\right)
\end{align*}
and
\begin{align*}
 \PB_{\star}(V) \colon \OI_{+}^{\opp} &\to \Set \\
 [k] &\mapsto \Hom_{\VI'(R)}\!\left(
 R^{\oplus [k]},V\right) \, .
\end{align*}
\end{defn}

We shall sometimes make use of an invariant that measures how proper the inclusion of semi-simplicial sets $\PB_{\star}(V) \emb \UuU_{\star}(V)$ is:

\begin{defn} \label{glr-defn}
 Let $R$ be a ring and $s \in \zz_{\geq 1}$. We write $\glr(R) \leq s$ if every finitely generated stably free right $R$-module of rank $\geq s$ is free.
\end{defn}

\begin{rem}
 Stably free modules which are \textbf{not} finitely generated are in fact always free \cite[Proposition 1.2]{gabel-thesis}, \cite[Proposition 4.2]{lam-serre}.
\end{rem}

\begin{thm} \label{glr-eqv}
 Given a ring $R$ and 
 $s \in \zz_{\geq 1}$, the following are equivalent:
\begin{birki}
 \item $\glr(R) \leq s$.
 \item Whenever $M \oplus R \cong R^{n}$ and $n > s$, we have $M \cong R^{n-1}$.
 \item For every $n > s$, the left action of $\GL_{n}(R)$ on the set of unimodular column vectors of size $n$ is transitive.
 \item For every $n > s$, the right action of $\GL_{n}(R)$ on the set of unimodular row vectors of size $n$ is transitive. 
 \item A unimodular column vector of size $n$ can be completed to a matrix in $\GL_{n}(R)$ whenever $n > s$.
 \item For every free right $R$-module $V$ of finite rank, the inclusion
 \[
 \PB_{\star}(V) \emb \UuU_{\star}(V)
 \]
 of semi-simplicial sets is full on the $k$-simplices with $0 \leq k \leq \rank_{R}(V) - s - 1$, that is, $\PB_{k}(V) = \UuU_{k}(V)$ in this range.
\end{birki}
\end{thm}
\begin{proof}
For every $n \in \zz_{\geq 1}$, consider the statement
\[
 T(n) : \text{Whenever $M \oplus R \cong R^{n}$ we have $M \cong R^{n-1}$.}
\] 
In \cite[page 407]{mr-nonc-book}, $T(n)$ being true means that  that $n-1$ is in the ``general linear range'' of $R$, hence in \cite[11.1.14]{mr-nonc-book} the inequality in (1) is essentially \emph{defined} as the statement
\[
\text{``$T(n)$ holds for every $n \geq s-1$''}
\] (which is (2) here). The theorem right after \cite[11.1.14]{mr-nonc-book} shows this is equivalent to our Definition \ref{glr-defn} of (1), and in addition equivalent to (5). The equivalence with (3) and (4) follows from \cite[Proposition 1.12]{mr-nonc-book}.

(1) $\imp$ (6): For every $k \leq \rank_{R}(V) - s- 1$, every 
\[ \alpha \in  \UuU_{k}(V) 
  = \Hom_{\VI(R)}\!\left(
 R^{\oplus [k]},V\right)
 \\
 = \Hom_{\VI(R)}\!\left(
 R^{k+1},V\right)
\]
satisfies $\coker(\alpha) \oplus R^{k+1} \cong V$, so the right $R$-module $\coker(\alpha)$ is stably free of rank $\rank_{R}(V) - (k + 1) \geq s$ and hence already free: this means
\[ \alpha \in  \Hom_{\VI'(R)}\!\left(
 R^{k+1},V\right) =
 \Hom_{\VI'(R)}\!\left(
 R^{\oplus [k]},V\right) = \PB_{k}(V) \, .
 \\
\]

(6) $\imp$ (2): Suppose $M \oplus R \cong R^{n}$ and $n > s$. Let us write $\alpha$ for this isomorphism precomposed with the inclusion $R \emb M \oplus R$ into the second coordinate. Then $\alpha$ is a morphism in $\VI(R)$ from $R$ to $R^{n}$, in other words $\alpha \in \UuU_{0}(R^{n})$ with
\[
 \rank_{R}(R^{n}) - s - 1 = (n -1) - s \geq 0 \, .
\]
Then $\alpha \in \PB_{0}(R^{n})$ as well, which means $\coker(\alpha) \cong M$ is free: $M \cong R^{n-1}$.

\end{proof}

\begin{prop} \label{glr-leq-sr}
 For every ring $R$ and $s \in \zz_{\geq 1}$, we have
 \[
  \sr(R) \leq s \,\,\, \imp \,\,\, \glr(R) \leq s \, .
 \]
\end{prop}
\begin{proof}
 See for instance \cite[page 417]{mr-nonc-book} or  \cite[Proposition 5.9]{rw-wahl-stab}.
\end{proof}
\begin{rem} \label{match-vdk}
 Specifying a splittable map $f \colon R^{\oplus [k]} \to V$ is equivalent to specifying a sequence
 \[
  (v_{0}, \dots,  v_{k})
 \]
 of elements in $V$ which form a basis of a free direct summand in $V$, via the mapping $f \mapsto (f(e_{0}), \dots, f(e_{k}))$ where $(e_{0}, \dots, e_{k})$ is the standard ordered $R$-basis of $R^{\oplus [k]}$. As a result, the face poset $\Face(\UuU_{\star}(V))$ is isomorphic to the poset $\mathcal{O}(V) \cap \UU$ of \cite[2.3, 2.4]{vdk-GL-stab}.
\end{rem}

\begin{rem} \label{match-maazen}
 Specifying a splittable map $R^{\oplus [k]} \to V$ with $R$-free cokernel is equivalent to specifying a sequence
 \[
  (v_{0}, \dots,  v_{k})
 \]
 of elements in $V$ which \textbf{can be completed to} a basis of the whole $V$, or with more common terminology an ordered \textbf{partial basis} of $V$. As a result, the face poset $\Face(\PB_{\star}(R^{n}))$ is isomorphic to the poset denoted
\begin{itemize}
 \item $\mathcal{O}(n,R)$ in \cite[page 27, Examples.(2)]{maazen-thesis},
 \vspace{0.1cm}
 \item $\mathscr{O}(n)$ in \cite[Appendix]{vdk-loo-symp}.
\end{itemize}
\end{rem}

\begin{defn} \label{u-link-defn}
 Fix a ring $R$, a free right $R$-module $V$ of finite rank, $a \in \zz_{\geq 0}$, and a splittable $R$-linear map $f \colon R^{\oplus [a]} \to V$. We define a semi-simplicial subset $\UuU_{\star}(V)_{f}$ of $\UuU_{\star}(V)$ as follows:
\begin{align*}
 \UuU_{\star}(V)_{f} \colon \OI_{+}^{\opp} &\to \Set \\
 [k] &\mapsto \left\{g \colon R^{\oplus [k]} \to V \text{ splittable} :
 f\oplus g \colon R^{\oplus [a]} \oplus R^{\oplus [k]} \to V \text{ still splittable}
 \right\} \, .
\end{align*}
If in addition $\coker(f)$ is $R$-free, we define a semi-simplicial subset $\PB_{\star}(V)_{f}$ of $\PB_{\star}(V)$ as follows:
\begin{align*}
 \PB_{\star}(V)_{f} \colon \OI_{+}^{\opp} &\to \Set \\
 [k] &\mapsto \left\{\!\!\!\!\text{  
\begin{tabular}{l}
 $g \colon R^{\oplus [k]} \to V$
 splittable \\
 and $\coker(g)$ is $R$-free
\end{tabular}
}
    \!\!\!\!:\!\!\!\!
\text{
\begin{tabular}{l}
 $f\oplus g \colon R^{\oplus [a]} \oplus R^{\oplus [k]} \to V$ splittable
 \\
 and $\coker(f \oplus g)$ is $R$-free
\end{tabular}
}
\!\!\!\!\right\} \, .
\end{align*}
\end{defn}


\begin{rem} \label{match-vdk-link}
 In a similar vein to Remark \ref{match-vdk}, for a splittable $R$-linear $f \colon R^{\oplus [a]} \to V$, the face poset $\Face(\UuU_{\star}(V)_{f})$ is isomorphic to the poset denoted
 \[ \mathcal{O}(V) \cap \UU_{(f(e_{0}), \,\dots\,, f(e_{a}))} \]
in \cite[2.3, 2.4]{vdk-GL-stab}.
\end{rem}

\begin{rem} \label{match-maazen-link}
 In a similar vein to Remark \ref{match-maazen}, for a splittable $R$-linear $f \colon R^{\oplus [a]} \to R^{n}$ with $R$-free cokernel, the face poset $\Face\pars*{\PB_{\star}(R^{n})_{f}}$ is isomorphic to the poset denoted
\begin{itemize}
 \item $\mathcal{O}(n,R)_{(f(e_{0}), \,\dots\,, f(e_{a}))}$ in \cite[page 27, Examples.(2) and page 28]{maazen-thesis},
 \vspace{0.1cm}
 \item $\mathscr{O}(n)_{(f(e_{0}), \,\dots\,, \,f(e_{a}))}$ in \cite[Appendix]{vdk-loo-symp}.
\end{itemize}
\end{rem}

\begin{defn} \label{u-simp-defn}
 Let $R$ be a ring and $V$ be a free right $R$-module of finite rank. Building off of Remark \ref{match-vdk}, we define an unordered version of $\UuU_{\star}(V)$: we set
\[
 \UuU_{\circ}(V) \coloneqq \{\empt \neq F \subseteq V : \text{$F$ is a basis of a free direct summand of $V$}\}
\]
so that $\UuU_{\circ}(V)$ is a set of nonempty finite subsets of $V$ with the property that 
\[
 \empt \neq F' \subseteq F \in \UuU_{\circ}(V) 
 \quad \text{implies} \quad
 F' \in \UuU_{\circ}(V) \, ,
\]
in other words an (abstract) simplicial complex whose vertex set is a subset of $V$.
\end{defn}

\begin{rem} \label{match-U}
 The simplicial complex $\UuU_{\circ}(R^{n})$ of Definition \ref{u-simp-defn} is 
\(
 U(R^{n})
\)
in the notation of \cite[the proof of Theorem 5.10]{rw-wahl-stab}.
\end{rem}

\begin{defn} \label{su-pbc-defn}
  Let $R$ be a ring and $\VIS(R)$ be the category of free right $R$-modules of finite rank and \textbf{split} $R$-linear maps, where a split map from $W$ to $V$ is a pair $(\alpha,\alpha')$ such that $\alpha \colon W \to V$, $\alpha' \colon V\to W$ with $\alpha'\circ\alpha = \id_{W}$. Let $\VIC(R)$ be the wide subcategory of $\VIS(R)$ given by 
\[
 \Hom_{\VIC(R)}(W,V) \coloneqq \left\{
 (\alpha,\alpha') \in \Hom_{\VIS(R)}(W,V) : \coker(\alpha) \text{ is $R$-free}
 \right\} \, .
\] 
 
Both $\VIS(R)$ and $\VIC(R)$ are monoidal via direct sum $\oplus$ where the unit object $0$ is initial. For every free right $R$-module $V$ of finite rank, taking $x$ to be $R$ and $y$ to be $V$ in (\ref{sfx}), we write $\SU_{\star}(V)$ and $\PBC_{\star}(V)$ for the corresponding semi-simplicial set, respectively for $\VIS(R)$ and $\VIC(R)$. More explicitly, they have the description
\begin{align*}
 \SU_{\star}(V) \colon \OI_{+}^{\opp} &\to \Set \\
 [k] &\mapsto \Hom_{\VIS(R)}\pars*{R^{\oplus [k]},V}
\end{align*}
and
\begin{align*}
 \PBC_{\star}(V) \colon \OI_{+}^{\opp} &\to \Set \\
 [k] &\mapsto \Hom_{\VIC(R)}\pars*{R^{\oplus [k]},V} \, .
\end{align*}
\end{defn}

\begin{thm} \label{glr-pbc}
 Given a ring $R$ and $s \in \zz_{\geq 1}$, the following are equivalent:
\begin{birki}
 \item $\glr(R) \leq s$.
 \item For every free right $R$-module $V$ of finite rank, the inclusion
 \[
 \PBC_{\star}(V) \emb \SU_{\star}(V)
 \]
 of semi-simplicial sets is full on the $k$-simplices with $k \leq \rank_{R}(V) - s - 1$, that is, $\PBC_{k}(V) = \SU_{k}(V)$ in this range.
\end{birki}
\end{thm}
\begin{proof}
 The proof of the equivalence involving the inclusion $\PB_{\star}(V) \emb \UuU_{\star}(V)$ in Theorem \ref{glr-eqv} goes through essentially verbatim.
\end{proof}

\begin{rem} \label{conv-to-fd}
 Let $R$ be a ring and $V$ be a nonzero finitely generated free right $R$-module. By Proposition \ref{rank-split-inj}, all four of the semi-simplicial sets 
 \[
  \PB_{\star}(V),\, \UuU_{\star}(V),\, \PBC_{\star}(V),\, \SU_{\star}(V)
 \]
 have dimension equal to $\rank_{R}(V) - 1$.
\end{rem}

\begin{cor} \label{glr-one}
 Given a ring $R$, the following are equivalent:
\begin{birki}
 \item $\glr(R) \leq 1$.
 \item For every free right $R$-module $V$ of finite rank, $\PB_{\star}(V) = \UuU_{\star}(V)$.
 \item For every free right $R$-module $V$ of finite rank, $\PBC_{\star}(V) = \SU_{\star}(V)$. 
\end{birki}
\end{cor}
\begin{proof}
For a fixed free right $R$-module $V$ with $\rank_{R}(V) < \infty$ and $k \in \zz_{\geq 0}$, let us write $F(V,k)$ for the claim that the inclusion
 \[
 \iota \colon 
 \PB_{\star}(V) \emb \UuU_{\star}(V)
 \]
 of semi-simplicial sets is full on the $k$-simplices; so that whenever $V \neq 0$ the statement $F(V,\rank_{R}(V)-1)$ holds by Proposition \ref{glr0}.

By the equivalence (1) $\Leftrightarrow$ (6) in Theorem \ref{glr-eqv}, the statement (1) is equivalent to the following bullet point:
\begin{itemize}
 \item For every free right $R$-module $V$ of finite rank and $0 \leq k \leq \rank_{R}(V) - 2$, the statement $F(V,k)$ holds. 
\end{itemize}
Recalling Remark \ref{conv-to-fd}, we get the equivalence (1) $\Leftrightarrow$ (2). A similar argument in tandem with Theorem \ref{glr-pbc} (instead of Theorem \ref{glr-eqv}) shows (1) $\Leftrightarrow$ (3).
\end{proof}

\begin{thm} \label{glr-dedekind}
 For every Dedekind domain $R$ we have $\glr(R) \leq 1$.
\end{thm}
\begin{proof}
 It follows from (for instance) \cite[Theorem]{gmr-dedekind} that part (5) of Theorem \ref{glr-eqv} holds with $s=1$.
\end{proof}

\begin{rem} \label{match-W}
 Fix integers $n \geq s \geq 1 $. Then the semi-simplicial set $\SU_{\star}(R^{n})$ of Definition \ref{su-pbc-defn} is 
\[
 W_{n-s}(R^{s},R)_{\bul}
\]
in the notation of \cite{rw-wahl-stab}, when the $W$-construction \cite[Definition 2.1]{rw-wahl-stab} is applied in the monoidal category $U \lMod{fR}$ \cite[Section 5.3]{rw-wahl-stab}. \end{rem}

\begin{rem} \label{su-elts}
 Writing $V^{*} \coloneqq \Hom_{R}(V,R)$, specifying a split map $(\alpha, \alpha') \colon R^{\oplus [k]} \to V$ is equivalent to specifying a sequence
 \[
  \big( (v_{0},g_{0}), \dots,  (v_{k},g_{k}) \big)
 \]
 of elements in $V \times V^{*}$ such that $g_{i}(v_{j}) = \delta_{ij}$.
\end{rem}

\begin{defn} \label{su-simp-defn}
 Let $R$ be a ring and $V$ be a free right $R$-module of finite rank. Building off of Remark \ref{su-elts}, we define an unordered version of $\SU_{\star}(V)$: we set
\[
 \SU_{\circ}(V) \coloneqq \left\{
 \empt \neq F \subseteq V \times V^{*} :  \text{ 
\begin{tabular}{l}
 $g(v) = 1$ whenever $(v,g) \in F$, and
 \\
 $g(w) = 0$ whenever $(v,g) \neq (w,h)$ both in $F$
\end{tabular}
 }
 \!\right\} \, .
\]
Note that $\SU_{\circ}(V)$ is a set of nonempty finite subsets of $V \times V^{*}$ with the property that 
\[
 \empt \neq F' \subseteq F \in \SU_{\circ}(V) 
 \quad \text{implies} \quad
 F' \in \SU_{\circ}(V) \, ,
\]
in other words an (abstract) simplicial complex whose vertex set is a subset of $V \times V^{*}$.
\end{defn}

\begin{rem} \label{match-S}
 Fix integers $n \geq s \geq 1 $. Then the simplicial complex $\SU_{\circ}(R^{n})$ of Definition \ref{su-simp-defn} is isomorphic to
\[
 S_{n-s}(R^{s},R)
\]
in the notation of \cite{rw-wahl-stab}, when the $S$-construction \cite[Definition 2.8]{rw-wahl-stab} is applied in the monoidal category $U \lMod{fR}$ \cite[Section 5.3]{rw-wahl-stab}.
\end{rem}

\begin{defn} \label{pbc-U}
 For a \textbf{commutative} ring $R$ and a subgroup $\fU \leq R^{\times}$, let $\VIC(R,\fU)$ be the subcategory of $\VIC(R)$ with $\{R^{n} : n\in \zz_{\geq 0}\}$ as its set of objects so that 
 \[
  \Hom_{\VIC(R,\fU)}\!\left(R^{a},R^{b}\right) \coloneqq 
\begin{cases}
 \Hom_{\VIC(R)}\!\left(R^{a},R^{b}\right) & \text{if $a \neq b$,}
 \\
 \{(\alpha,\alpha') \in
 \Hom_{\VIC(R)}\!\left(R^{a},R^{a}\right) :
 \det(\alpha) \in \fU\} & \text{if $a=b$.}
\end{cases}
 \]
 See \cite[pages 2528-2529]{putman-sam} for alternative definitions of this category. The category $\VIC(R,\fU)$ is monoidal via direct sum $\oplus$ where the unit object $0$ is initial. For every $n \in \zz_{\geq 0}$, taking $x$ to be $R$ and $y$ to be $R^{n}$ in (\ref{sfx}), we write $\PBC^{\fU}_{\star}(R^{n})$ for the corresponding semi-simplicial set for $\VIC(R,\fU)$. More explicitly, it has the description
\begin{align*}
 \PBC^{\fU}_{\star}(R^{n}) \colon \OI_{+}^{\opp} &\to \Set \\
 [k] &\mapsto \Hom_{\VIC(R,\fU)}\!\left(
 R^{\oplus [k]},R^{n}\right) \, .
\end{align*}
\end{defn}

\begin{rem} \label{all-but-top}
 Let $R$ be a commutative ring with a subgroup $\fU \leq R^{\times}$. For every $n \in \zz_{\geq 0}$, the inclusion
 \[
 \PBC^{\fU}_{\star}(R^{n}) \emb \PBC_{\star}(R^{n})
 \]
 of semi-simplicial sets is full on the $k$-simplices with $k \leq n - 2$, that is, $\PBC^{\fU}_{k}(R^{n}) = \PBC_{k}(R^{n})$ in this range.
\end{rem}

\begin{defn} \label{hu-defn}
 For a \textbf{commutative} ring $R$, let $\SI(R)$ be the category of finitely generated free $R$-modules equipped with a symplectic form (shortly symplectic $R$-modules) and $R$-linear maps preserving the symplectic form (shortly symplectic maps). The main example of a symplectic $R$-module is $H_{R}$ which has rank $2$ with an $R$-basis $\{x,y\}$ equipped with the unique symplectic form given by 
\begin{align*}
  \omega \colon H_{R} \times H_{R} &\to R \\
  (x,y) &\mapsto 1 \, .
\end{align*}
$\SI(R)$ is monoidal via direct sum $\oplus$ where the unit object $0$ is initial. For every symplectic $R$-module $V$, taking $x$ to be $H_{R}$ and $y$ to be $V$ in (\ref{sfx}), we write $\HU_{\star}(V)$ for the corresponding semi-simplicial set. More explicitly, it has the description
\begin{align*}
 \HU_{\star}(V) \colon \OI_{+}^{\opp} &\to \Set \\
 [k] &\mapsto \Hom_{\SI(R)}\!\left(
 H_{R}^{\oplus [k]},V\right) \, .
\end{align*}
\end{defn}

\begin{rem} \label{match-hu}
 Specifying a symplectic map $H_{R}^{\oplus [k]} \to V$ is equivalent to specifying a sequence
 \[
  (x_{0}, y_{0}, \dots,  x_{k}, y_{k})
 \]
 of elements in $V$ such that the symplectic form in $V$ satisfies
\begin{align*}
 \gen{x_{i},y_{j}} &= \delta_{ij} = -\gen{y_{i},x_{j}} \\
 \gen{x_{i},x_{j}} &= 0 = \gen{y_{i},y_{j}}
\end{align*}
As a result, the face poset $\Face(\HU_{\star}(V))$ is isomorphic to the poset denoted
\begin{itemize}
 \item $HU(V)$ in \cite[Section 3]{charney-symp},
 \vspace{0.1cm}
 \item $\mathcal{HU}(R^{2n})$ in \cite[Section 7]{mirzaii-vdk-unitary} when $V = H_{R}^{\oplus n}$,
 \vspace{0.1cm}
 \item $\mathscr{HU}_{g}(R)$ in \cite[Section 4]{vdk-loo-symp} when $V = H_{R}^{\oplus g}$.
\end{itemize}
\end{rem}

\subsection{Connectivity} 
Recalling Remark \ref{match-vdk}, Remark \ref{match-vdk-link}, the argument in \cite[proof of Lemma 5.10]{rw-wahl-stab} can be compartmentalized as the following three results:
\begin{prop} \label{join}
 Let $R$ be a ring and $n \in \zz_{\geq 0}$. Then in the sense of \emph{\cite[Definition 3.2]{hatcher-wahl-mcg-3d}}, the simplicial complex $\SU_{\circ}(R^{n})$ of Definition \ref{su-simp-defn} is a ``join complex over'' the simplicial complex $\UuU_{\circ}(R^{n})$ of Definition \ref{u-simp-defn} through the natural forgetful map
 \[
 \pi \colon \SU_{\circ}(R^{n}) \to \UuU_{\circ}(R^{n}) \, .
 \]
\end{prop}
\begin{proof}
 Recalling Remark \ref{match-U} and Remark \ref{match-S}, this follows from the argument that starts at the bottom two paragraphs of \cite[page 603]{rw-wahl-stab}.
\end{proof}

\begin{prop}[\cite{rw-wahl-stab}] \label{wCM}
 Let $R$ be a ring and $s \in \zz_{\geq 1}$ such that $\glr(R) \leq s$. Consider the following statement:\\
 $B(n)$: In the sense of \emph{\cite[Definition 2.11]{rw-wahl-stab}}, the simplicial complex $\UuU_{\circ}(R^{n})$ is ``weakly Cohen--Macaulay of dimension $n-s$''.
 
If $B(n)$ holds for every integer $n \geq s$, then for every $n \in \zz_{\geq 0}$ both 
\begin{itemize}
 \item the simplicial complex $\SU_{\circ}(R^{n})$ of Definition \ref{su-simp-defn}, and
 \item the semi-simplicial set $\SU_{\star}(R^{n})$ of Definition \ref{su-pbc-defn},
\end{itemize}
 are
\ds{
\floor*{\frac{n-s-2}{2}}
}-connected.
\end{prop}
\begin{proof}
For $n=0$ we have
\(
\floor*{\frac{n-s-2}{2}} \leq  \floor*{\frac{-3}{2}} = -2
\)
so there is nothing to show. For $1 \leq n \leq s$, we have 
\(
\floor*{\frac{n-s-2}{2}} \leq  -1
\)
and each three of the spaces is nonempty hence $(-1)$-connected. Consequently we may assume from now on that $n > s$.

The penultimate paragraph of \cite[proof of Lemma 5.10]{rw-wahl-stab} shows that for every $p$-simplex $\sigma$ of $\SU_{\circ}(R^{n})$ with $0 \leq p < n-s$, we have
\[
 \Link_{\SU_{\circ}(R^{n})}(\sigma) \cong \SU_{\circ}(R^{n-p-1})
\]
provided that the implication
\[
 E(p,n): \quad R^{p+1} \oplus H \cong R^{n} \,\,\imp\,\, H \cong R^{n-p-1}
\]
holds: indeed when $0 \leq p < n-s$, such $H$ is stably free of rank $n-p-1 \geq s \geq \glr(R)$ so has to be free. 

Consequently, under the forgetful map $ \pi \colon \SU_{\circ}(R^{n}) \to \UuU_{\circ}(R^{n})
$, every $p$-simplex $\sigma$ of $\SU_{\circ}(R^{n})$ with $0 \leq p < n-s$ satisfies
\[
 \pi\! \left(
 \Link_{\SU_{\circ}(R^{n})}(\sigma)
 \right) \cong \UuU_{\circ}(R^{n-p-1}) \, ,
\]
which is weakly Cohen--Macaulay of dimension $n-p-1-s$ by $B(n-p-1)$ (which holds because $n-p > s$). In particular, this link is weakly Cohen--Macaulay of dimension $n-s-p-2$. Together with the validity of $B(n)$ and Proposition \ref{join}, we may apply  \cite[Theorem 3.6]{hatcher-wahl-mcg-3d} to $\pi$ and conclude that $\SU_{\circ}(R^{n})$ is $\floor*{\frac{n-s-2}{2}}$-connected.

To conclude the same connectivity for $\SU_{\star}(R^{n})$, recalling Remark \ref{match-W} and Remark \ref{match-S}, thanks to \cite[Theorem 2.10]{rw-wahl-stab} it suffices to verify the following with the notation of the paper \cite{rw-wahl-stab}:
\begin{birki}
 \item The monoidal category $(U \lMod{fR}, \oplus, 0)$ of \cite[Section 5.3]{rw-wahl-stab} is ``locally homogeneous'' \cite[Definition 1.2]{rw-wahl-stab} and ``locally standard'' \cite[Definition 2.5]{rw-wahl-stab} at $(R^{s},R)$.
 \item ``Condition (A)'' \cite[page 558]{rw-wahl-stab} holds for all $n$.
\end{birki}
Here (2) can be checked by hand, or alternatively it follows from (1) and \cite[Proposition 2.9]{rw-wahl-stab} as $U\lMod{fR}$ is symmetric monoidal. 

For the ``locally homogeneous``'' part of (1), we need to verify the two properties \textbf{LH1} and \textbf{LH2} in \cite[Definition 1.2]{rw-wahl-stab} which can be reduced to statements about the monoidal groupoid $(\lMod{fR}, \oplus , 0)$ itself (before the $U$-construction) via \cite[Theorem 1.10]{rw-wahl-stab}. The if clause of \cite[Theorem 1.10, part (b)]{rw-wahl-stab} holds in $\lMod{fR}$ at every pair of objects. The if clause of \cite[Theorem 1.10, part (a)]{rw-wahl-stab} for the pair $(R^{s},R)$ is (see \cite[Definition 1.9]{rw-wahl-stab}) that the implication
\[
 Y \oplus R^{p+1} \cong R^{s+n} \quad \imp
 \quad Y \cong R^{s + n - p - 1}
\]
should hold for every finitely generated right $R$-module $Y$ and whenever $0 \leq p  < n$. This implication is precisely the $E(p,n+s)$ we introduced in the beginning of the proof which we have shown to hold whenever
\[
 0 \leq p < (n+s) - s = n \, ,
\]
as desired.

For the ``locally standard'' part of (1), we can invoke \cite[Proposition 2.6]{rw-wahl-stab} which amounts to checking that for every $n$ the semi-simplicial set $\SU_{\star}(R^{n})$ is determined by its vertices and has distinct vertices, which are evident.
\end{proof}

\begin{prop}[\cite{rw-wahl-stab}] \label{ular}
 Let $R$ be a ring and $1 \leq s \leq n$ integers such that the following hold:
\begin{birki}
 \item The semi-simplicial set $\UuU_{\star}(R^{n})$ of Definition \ref{pb-u-defn} is $(n-s-1)$-connected.
 \item For every $a \in \zz_{\geq 0}$ and splittable $R$-linear map $f \colon R^{\oplus [a]} \to V$, the semi-simplicial set $\UuU_{\star}(R^{n})_{f}$ of Definition \ref{u-link-defn} is $(n-s-a-2)$-connected.
\end{birki} 
Then condition $B(n)$ of Proposition \ref{wCM} holds. 
\end{prop}
\begin{proof}
 This follows from the argument in the paragraphs adjacent to the commutative diagram in \cite[page 603]{rw-wahl-stab}, which we shall now expand in our notation. Inspecting the involved definitions, assuming (1) and (2), we need to verify that the simplicial complex $\UuU_{\circ}(R^{n})$ satisfy the following:
\begin{birki}[(i)]
 \item $\UuU_{\circ}(R^{n})$ is $(n-s-1)$-connected.
 \item For every $a \geq 0$ and $a$-simplex $F \in \UuU_{\circ}(R^{n})$, the link $\Link_{\UuU_{\circ}(R^{n})}(F)$ is 
\vspace{-0.1cm} 
 \[
 \text{$(n-s-a-2)$-connected.}
 \]
\end{birki}
For every free right $R$-module $V$ of finite rank, writing $e_{0}, \dots, e_{k}$ for the standard basis elements of $R^{\oplus [k]}$, the assignments
\begin{align*}
  p_{k} \colon \Hom_{\VI(R)}(R^{\oplus[k]}, V) &\to \{\text{$k$-simplices of $\UuU_{\circ}(V)$}\} \\
 f &\mapsto \{f(e_{i}) : i \in [k]\} \, ,
\end{align*}
as $k$ varies, patch to a cellular map 
\[
 p \colon \left| \UuU_{\star}(V) \right| \to \left| \UuU_{\circ}(V) \right|
\]
between the geometric realizations. Moreover, choosing a total order on the set of vertices of $\UuU_{\circ}(V)$ defines a cellular map 
\[
 j \colon \left| \UuU_{\circ}(V) \right| \to \left| \UuU_{\star}(V) \right|
\]
such that $pj = \id_{\left|\UuU_{\circ}(V)\right|}$. Therefore (i) follows from (1) and the functoriality of homotopy groups.

Next, given $b \geq 0$ and a $b$-simplex $F \in \UuU_{\circ}(V)$, we have
\begin{align*}
 \Link_{\UuU_{\circ}(V)}(F)
 &= \left\{
 G \in \UuU_{\circ}(V) : G \cup F \in \UuU_{\circ}(V)
 \text{ and } G \cap F = \empt
 \right\} \, ,
\end{align*}
so by letting $f = j_{b}(F)$, the maps $p,i$ above restrict to maps
\begin{align*}
  p &\colon \left| \UuU_{\star}(V)_{f} \right| \to 
  \left| \Link_{\UuU_{\circ}(V)}(F) \right| 
  \\ 
  j &\colon \left| \Link_{\UuU_{\circ}(V)}(F) \right| \to 
  \left| \UuU_{\star}(V)_{f} \right|   
\end{align*}
such that $pj = \id_{\left|\UuU_{\circ}(V)_{f}\right|}$. Therefore (ii) follows from (2) and the functoriality of homotopy groups.
\end{proof}

\begin{cor} \label{nihai}
  Let $R$ be a ring and $s \in \zz_{\geq 1}$ such that $\glr(R) \leq s$. Consider the following statements:\\
  $C(n)$: The semi-simplicial set $\UuU_{\star}(R^{n})$ is $(n-s-1)$-connected.\\
  $C'(n)$: For every $q \in \zz_{\geq 0}$ and splittable $R$-linear map $f \colon R^{\oplus [q]} \to R^{n}$ between right $R$-modules, the semi-simplicial set $\UuU_{\star}(R^{n})_{f}$ is $(n-s-q-2)$-connected.
  
  If both $C(n)$ and $C'(n)$ hold for every integer $n \geq s$, then for every $n \in \zz_{\geq 0}$ the semi-simplicial set $\SU_{\star}(R^{n})$ is
 \[
\text{$\floor*{\frac{n-s-2}{2}}$-connected.}
\]
\end{cor}
\begin{proof}
  As
\(
 \floor*{\frac{-s-2}{2}} \leq -2
\), there is nothing to show when $n=0$. By Proposition \ref{ular}, condition $B(n)$ of Proposition \ref{wCM} holds for every $n \geq s$. We may now apply Proposition \ref{wCM} and conclude.
\end{proof}

\begin{thm}[{\cite{rw-wahl-stab}}] \label{su-conn}
 Let $R$ be a ring and $s \in \zz_{\geq 1}$ such that \(\sr(R) \leq s\). Then for every $n \in \zz_{\geq 0}$, the semi-simplicial set $\SU_{\star}(R^{n})$ is
\[
\text{$\floor*{\frac{n-s-2}{2}}$-connected.}
\]
\end{thm}
\begin{proof}
 We have $\glr(R) \leq s$ by Proposition \ref{glr-leq-sr}. Therefore it suffices to verify the statements $C(n)$, $C'(n)$ of Corollary \ref{nihai} for every $n \geq s$. Fixing an arbitrary $n \geq s$, we have the following:
\begin{itemize}
 \item The face poset $\Face(\UuU_{\star}(R^{n}))$ is 
 \[
  \text{$(n-2-(s-1)) = (n-s-1)$-connected}
 \]
  by Remark \ref{match-vdk}, \cite[Convention 2.2]{vdk-GL-stab}, and \cite[Theorem 2.6, part (i) with $\delta = 0$]{vdk-GL-stab}. Hence $C(n)$ holds by Theorem \ref{regular-ok}.
 \item For every $q \in \zz_{\geq 0}$ and splittable $R$-linear map $f \colon R^{\oplus [q]} \to R^{n}$ between right $R$-modules, the face poset $\Face(\UuU_{\star}(R^{n})_{f})$ is
 \[
  \text{$(n-2-(s-1) - (q+1)) = (n-s-q-2)$-connected}
 \]
 by Remark \ref{match-vdk-link}, \cite[Convention 2.2]{vdk-GL-stab}, and \cite[Theorem 2.6, part (ii) with $\delta = 0$]{vdk-GL-stab}. Hence $C'(n)$ holds by Theorem \ref{regular-ok}.
\end{itemize}
\end{proof}

\begin{prop} \label{su-pbc-conn}
 Let $R$ be a ring and $s \in \zz_{\geq 1}$ such that $\glr(R) \leq s$. Then for every $n \in \zz_{\geq 1}$, the inclusion of the nonempty semi-simplicial sets 
 \[
 \PBC_{\star}(R^{n}) \emb \SU_{\star}(R^{n})
 \]
 is $(n-s-1)$-connected.
 \end{prop}
\begin{proof}
 This follows from Theorem \ref{glr-pbc} and Proposition \ref{inclusion-conn}.
\end{proof}

\begin{cor} \label{pbc-conn}
 Let $R$ be a ring and $s \in \zz_{\geq 1}$ such that \(\sr(R) \leq s\). Then for every $n \geq 0$, the semi-simplicial set $\PBC_{\star}(R^{n})$ is
\[
\text{$\floor*{\frac{n-s-2}{2}}$-connected.}
\]
\end{cor}
\begin{proof}
  If $n = 0$, then $\PBC_{\star}(R^{0}) = \empt$ is $(-2)$-connected and 
\[
 \floor*{\frac{0-s-2}{2}} \leq \floor*{\frac{-3}{2}} = -2 \, .
\]
 Fix $n \geq 1$ and note that $\glr(R) \leq s$ by Proposition \ref{glr-leq-sr}. The larger semi-simplicial set $\SU_{\star}(R^{n})$ is $\floor*{\frac{n-s-2}{2}}$-connected by Theorem \ref{su-conn}. Now Proposition \ref{su-pbc-conn}, the implication
 \[
 0 \leq j \leq \floor*{\frac{n-s-2}{2}} \quad \imp \quad
 0 \leq j < n-s-1 \, ,
 \]
and Proposition \ref{conn-transfer} yield the result.
\end{proof}

\begin{prop} \label{pbc-unit-conn}
 Let $R$ be a commutative ring with a subgroup $\fU \leq R^{\times}$. Then for every $n \in \zz_{\geq 1}$, the inclusion of the nonempty semi-simplicial sets
 \[
  \PBC^{\fU}_{\star}(R^{n}) \emb \PBC_{\star}(R^{n})
 \]
 is $(n-2)$-connected.
\end{prop}
\begin{proof}
 This follows from Remark \ref{all-but-top} and Proposition \ref{inclusion-conn}.
\end{proof}

\begin{cor} \label{PBC-fU-conn}
 Let $R$ be a commutative ring with a subgroup $\fU \leq R^{\times}$, and $s \in \zz_{\geq 1}$ such that \(\sr(R) \leq s\). Then for every $n \geq 0$, the semi-simplicial set $\PBC^{\fU}_{\star}(R^{n})$ is
\[
\text{$\floor*{\frac{n-s-2}{2}}$-connected.}
\]
\end{cor}
\begin{proof}
 If $n = 0$, then $\PBC^{\fU}_{\star}(R^{0}) = \empt$ is $(-2)$-connected and 
\[
 \floor*{\frac{0-s-2}{2}} \leq \floor*{\frac{-3}{2}} = -2 \, .
\]
 Fix $n \geq 1$. The larger semi-simplicial set $\PBC_{\star}(R^{n})$ is $\floor*{\frac{n-s-2}{2}}$-connected by Corollary \ref{pbc-conn}. Now Proposition \ref{pbc-unit-conn}, the implication
 \[
 0 \leq j \leq \floor*{\frac{n-s-2}{2}} \quad \imp \quad
 0 \leq j < n-2
 \]
and Proposition \ref{conn-transfer}  yield the result.
\end{proof}

\begin{thm} \label{PBC-Euc-conn}
 Let $R$ be a Euclidean domain. Then for every $n \geq 0$, the semi-simplicial set $\PBC_{\star}(R^{n}) = \SU_{\star}(R^{n})$ is
\[
\text{$\floor*{\frac{n-3}{2}}$-connected.}
\]
\end{thm}
\begin{proof}
 Being a Euclidean domain, $R$ is also a Dedekind domain, so $\glr(R) \leq 1$ by Theorem \ref{glr-dedekind}. We get $\PB_{\star}(R^{n}) = \UuU_{\star}(R^{n})$ and $\PBC_{\star}(R^{n}) = \SU_{\star}(R^{n})$ by Corollary \ref{glr-one}. Therefore it suffices to verify the statements $C(n)$, $C'(n)$ of Corollary \ref{nihai} with $s \coloneqq 1$ for every $n \geq 1$. Fixing an arbitrary $n \geq 1$, we have the following:
 \begin{itemize}
 \item The face poset $\Face(\UuU_{\star}(R^{n})) = \Face(\PB_{\star}(R^{n}))$ is 
 \(
  \text{$(n-2)$-connected}
 \)
  by Remark \ref{match-maazen} and \cite[Theorem 6.1(i)]{vdk-loo-symp}. Hence $C(n)$ holds by Theorem \ref{regular-ok}.
 \item For every $q \in \zz_{\geq 0}$ and splittable $R$-linear map $f \colon R^{\oplus [q]} \to R^{n}$ between right $R$-modules, the face poset $\Face(\UuU_{\star}(R^{n})_{f}) =\Face(\PB_{\star}(R^{n})_{f})$ is
 \[
  \text{$(n-(q+1) - 2) = (n-q-3)$-connected}
 \]
 by Remark \ref{match-maazen-link} and \cite[Theorem 6.1(ii)]{vdk-loo-symp}. Hence $C'(n)$ holds by Theorem \ref{regular-ok}.
\end{itemize}
\end{proof}

\begin{thm}[{\cite{vdk-loo-symp}}] \label{hu-conn}
 Let $R$ be a Euclidean domain and write $Y \coloneqq H_{R}$ as in Definition \ref{hu-defn}. Then for every $g \geq 0$, the semi-simplicial set $\HU_{\star}(Y^{g})$ is
 \[
\text{$\floor*{\frac{g-3}{2}}$-connected.}
\]
\end{thm}
\begin{proof}
 The face poset $\Face(\HU_{\star}(Y^{g}))$ is $\floor*{\frac{g-3}{2}}$-connected by Remark \ref{match-hu} and \cite[Theorem 4.1]{vdk-loo-symp}. Now invoke Theorem \ref{regular-ok}.
\end{proof}

\section{The $\GG \mapsto U\GG$ construction and its functoriality} \label{UG-ulan}

While the following is a special case of Quillen's bracket construction \cite[page 219]{grayson-higher-k}, we follow the treatment in \cite{rw-wahl-stab} more closely.

\begin{defn}[{\cite[Section 1.1]{rw-wahl-stab}}] \label{U-const}
 Let $(\GG, \oplus, 0)$ be a monoidal groupoid. The category $U\GG$ has the same objects with $\GG$ and given  $A,B \in \Obj(\GG)$, we set
 \[
  \Hom_{U \GG}(A,B) \coloneqq \bigslant{\kume*{(X,f) : f \in \Hom_{\GG}(X \oplus A, B)}}{\sim}
 \]
 where we declare $(X,f) \sim (X',f')$ with $f \colon X \oplus A \to B$ and $f' \colon X' \oplus A \to B$ if there exists $\alpha \colon X \to X'$ in $\GG$ such that $f' \circ (\alpha \oplus \id_{A}) = f$ (this is an equivalence relation as $\GG$ is a groupoid); with composition  
\begin{align*}
  \Hom_{U \GG}(A,B) \times \Hom_{U \GG}(B,C) &\to \Hom_{U\GG}(A,C)
  \\
  \big([X,f],[Y,g]\big) &\mapsto [Y \oplus X,\, g \circ (\id_{Y} \oplus f)] \, .
\end{align*}
and identity morphisms $[0,\id_{A}] \in \End_{U\GG}(A)$. 
\end{defn} 

\begin{defn}\label{defn-ker}
 Let $F \colon \GG \to \GG^{\wr}$ be a functor between groupoids $\GG,\GG^{\wr}$. We write $\ker F$ for the subcategory of $\GG$ given by
\begin{itemize}
 \item $\Obj(\ker F) = \Obj(\GG)$,
 \item $\Hom_{\ker F}(X,Y) = 
\begin{cases}
 \empt & \text{if $X \neq Y$,}
 \\
 \{\vphi \in \Hom_{\GG}(X,X) : F(\vphi) = \id_{F(X)}\} & \text{if $X= Y$.}
\end{cases}
$
\end{itemize}
\end{defn}

\begin{rem}
 In the setting of Definition \ref{defn-ker}, $\ker F$ is a subgroupoid of $\GG$ as it contains the identity morphisms and is closed under composition and taking inverses. Moreover if $F$ is a monoidal functor between monoidal groupoids, $\ker F$ inherits a monoidal structure from $\GG$.
\end{rem}

\begin{prop} \label{conj-action-gpd}
 Let $(\GG, \oplus, 0)$ and $(\GG^{\wr}, \oplus^{\wr}, 0^{\wr})$ be monoidal groupoids and 
\[
 (F,J,\psi) \colon \GG \to \GG^{\wr}
\] 
  be a monoidal functor as in Definition \ref{monoidal-F}. Writing $\K \coloneqq \ker F$, 
 the assignments
\begin{align*}
 \Hom_{U\GG}(A,B) 
 &\to \Hom_{\Grp}(\Aut_{\K}(A),\,\Aut_{\K}(B))
 \\
 [X,f] &\mapsto f \circ \pars*{\id_{X} \oplus -} \circ f^{-1}
\end{align*}
for every $A,B \in \Obj(\GG)$ define a functor
\begin{align*}
  \Aut_{\K} \colon U\GG &\to \Grp
  \\
  A &\mapsto \Aut_{\K}(A) \, .
\end{align*}
\end{prop}

\begin{rem}
 Since the proofs of the statements in this appendix took about 25 pages in total, I have put them in a separate place: section \ref{UG-proofs}. These proofs form the least readable part of this paper and can be skipped in first (or all) readings. Hopefully the statements themselves have become easier to navigate this way.
\end{rem}

\begin{prop} \label{A-to-U}
 Let $(\GG, \oplus, 0)$ be a monoidal groupoid and $c,x \in \Obj(\GG)$ so that we have an associated $\atsi$-group $\GG(c,x)$ from Definition \ref{monoidal-to-A-grp} and hence the category $\atsi_{\GG(c,x)}$ from Definition \ref{grothendieck-const}. The assignments
\begin{align*}
 \Hom_{\atsi_{\GG(c,x)}}(m,n) &\to \Hom_{U\GG^{\circ}}\pars*{c \oplus x^{\oplus m},c \oplus x^{\oplus n}}  
 \\
 \pars*{f,m,n} &\mapsto 
 \left[x^{\oplus n-m}, f \right]
\end{align*}
for every pair of integers $0 \leq m \leq n$ and $f \in \GG(c,x)_{n} = \Aut_{\GG}\pars*{c \oplus x^{\oplus n}}$ constitute a functor 
\[T_{\GG,c,x} \colon \atsi_{\GG(c,x)} \to U\GG^{\circ} \, . \]
Here $\GG^{\circ}$ is the monoidal opposite of $\GG$ as in Definition \ref{defn:mon-opp}.
\end{prop}

\begin{prop} \label{UF-oldu}
 Let $(\GG, \oplus, 0)$ and $(\GG^{\wr}, \oplus^{\wr}, 0^{\wr})$ be monoidal groupoids and 
\[
 (F,J,\psi) \colon \GG \to \GG^{\wr}
\] 
  be a monoidal functor as in Definition \ref{monoidal-F}. Then the assignments
\begin{align*}
 U_{F,J,\psi} \colon \Hom_{U\GG}(A,B) &\to \Hom_{U\GG^{\wr}}(F(A),F(B))
 \\
 [X,f] &\mapsto \left[F(X),\, F(f) \circ J_{X,A}\right]
\end{align*}
patch to a functor $U_{F,J,\psi} \colon U\GG \to U\GG^{\wr}$. 

Furthermore, writing $\cat$ for the category of (small) categories and $\mongpd$ for the category of monoidal groupoids and monoidal functors, the above assignments for various $\GG, \GG^{\wr}$  patch to a functor
\(
 U \colon \mongpd \to \cat
\).
\end{prop}

\begin{lem} \label{cat-epi}
Let $\ce_{1}, \ce_{2}, \ce_{3}$ be categories and $F \colon \ce_{1} \to \ce_{2}$ a functor which is essentially surjective and full. If $H, \wh{H} \colon \ce_{2} \to \ce_{3}$ are two functors such that $H \circ F \cong \wh{H} \circ F$, then $H \cong \wh{H}$.
\end{lem}

\begin{thm} \label{enhance}
We have the following monoidal enhancements:
\begin{birki}
 \item For every braided monoidal groupoid $(\GG,\oplus,0, \tau)$, the operation defined as
 \[
  [X,f] \oplus [Y,g] \coloneqq 
  \left[X \oplus Y, (f \oplus g) \circ (\id_{X} \oplus \tau_{A,Y}^{-1} \oplus \id_{C})\right]
 \]
 for every $[X,f] \in \Hom_{U\GG}(A,B)$ and $[Y,g] \in \Hom_{U\GG}(C,D)$ is monoidal on $U\GG$.
 \item Given a braided monoidal functor $(F,J,\psi) \colon (\GG,\oplus,0,\tau) \to (\GG^{\wr},\oplus^{\wr},0^{\wr},\tau^{\wr})$ between braided monoidal groupoids as in Definition \ref{bm-functor}, the triple 
 \[ (U_{F,J,\psi}, [0^{\wr},J], [0^{\wr},\psi]) \colon (U\GG, \oplus, 0) \to (U\GG^{\wr}, \oplus^{\wr}, 0^{\wr})\]
 is a monoidal functor as in Definition \ref{monoidal-F} where 
\begin{itemize}
 \item the functor $U_{F,J,\psi}$ is as in Proposition \ref{UF-oldu},
 \item $[0^{\wr},J]$ stands for the collection of isomorphisms 
 \[ \{[0^{\wr}, J_{A,B}] \colon F(A) \oplus^{\wr} F(B) \to F(A \oplus B)\}
 \]
 in $U\GG^{\wr}$,
 \item $[0^{\wr},\psi] \colon F(0) \to 0^{\wr}$ in $U\GG^{\wr}$.
\end{itemize}
\item 
 Via the assignments in (1) and (2), writing $\moncat$ for the category of monoidal categories and $\bmgpd$ for the category of braided monoidal groupoids and braided monoidal functors , the functor $U$ lifts to complete the square
 \[
  \xymatrixcolsep{2.5cm}
  \xymatrix{
  \bmgpd \ar@{-->}[r]^-{U} \ar[d] & \moncat \ar[d]
  \\
  \mongpd \ar[r]^-{U} & \cat
  }
 \]
where the vertical functors are forgetful.
\end{birki}
\end{thm}

\begin{prop} \label{conj-action-Hk}
 Let $(\GG, \oplus, 0)$ and $(\GG^{\wr}, \oplus^{\wr}, 0^{\wr})$ be monoidal groupoids and 
\[
 \mathbf{F} = (F,J,\psi) \colon \GG \to \GG^{\wr}
\] 
  be a monoidal functor so we have $U_{\mathbf{F}} \colon U\GG \to U\GG^{\wr}$ from Proposition \ref{UF-oldu}. Write $\K \coloneqq \ker F$, and let $\Aut_{\K} \colon U\GG \to \Grp$ be as in Proposition \ref{conj-action-gpd}.
If $F$ is essentially surjective and full, then for every $k \in \zz_{\geq 0}$, writing $\co_{k} \colon \Grp \to \lMod{\zz}$ for the $k$-th group homology functor, there is a unique functor
\(
 \mathbf{H}_{k} 
\)
which makes the diagram
\[
 \xymatrixcolsep{2.1cm}
 \xymatrixrowsep{1.4cm} 
 \xymatrix{
 U\GG \ar[r]^-{U_{\mathbf{F}}} \ar[d]_{\Aut_{\K}} &  U{\GG^{\wr}} 
 \ar@{-->}[d]^{\exists !\,\mathbf{H}_{k}}
 \\
 \Grp \ar[r]_{\co_{k}} & \lMod{\zz} \, ,
 }
\] 
commute up to natural isomorphism; in other words, the $U\GG$-module $\co_{k}(\Aut_{\K})$  descends to a $U\GG^{\wr}$-module in a unique way up to natural isomorphism.
\end{prop}

\begin{thm} \label{essek}
Let $(\GG,\oplus,0,\tau)$ and $(\GG^{\wr},\oplus^{\wr},0^{\wr},\tau^{\wr})$ be braided monoidal groupoids, and $\mathbf{F} = (F,J,\psi) \colon \GG \to \GG^{\wr}$ be a braided monoidal functor  as in Definition \ref{bm-functor} and write $\K \coloneqq \ker F$. Given
\begin{itemize}
 \item $c,x \in \Obj(\GG)$,
 \item $c^{\wr}, x^{\wr} \in \Obj(\GG)$,
 \item $\gamma \colon F(c) \to c^{\wr}$ and $\xi \colon F(x) \to x^{\wr}$ in $\GG^{\wr}$,
\end{itemize}
consider the map $\mathbf{F}[\gamma,\xi] \colon \GG(c,x) \to \GG^{\wr}(c^{\wr},x^{\wr})$ of $\atsi$-groups from Remark \ref{Gcx-maps} and set 
\[
 \K(c,x) \coloneqq \ker \mathbf{F}[\gamma,\xi] \, .
\]
Consider the following functors/constructions:
\begin{itemize}
 \item The $\atsi$-group $\K(c,x)$ extends to an $\atsi_{\GG(c,x)}$-group via Proposition \ref{conj-action}.
 \vspace{0.1cm}
 \item Proposition \ref{A-to-U} yields a functor \(T_{\GG,c,x} \colon \atsi_{\GG(c,x)} \to U\GG^{\circ}\).
 \vspace{0.1cm}
 \item The braided monoidal functor $\pmb{\tau} \coloneqq (\id_{\GG},\tau,\id_{0}) \colon \GG^{\circ} \to \GG$ in Proposition \ref{exer} induces a (monoidal) functor $U_{\pmb{\tau}} \colon U\GG^{\circ} \to U\GG$ via Proposition \ref{UF-oldu}.
 \vspace{0.1cm}
 \item Writing $\K \coloneqq \ker F$, we have $\Aut_{\K} \colon U\GG \to \Grp$ from Proposition \ref{conj-action-gpd}.
\end{itemize}
 The following hold: 
\begin{birki}
 \item The diagram
\[
 \xymatrixcolsep{2.1cm}
 \xymatrixrowsep{1.4cm} 
 \xymatrix{
 \atsi_{\GG(c,x)} \ar[r]^-{T_{\GG,c,x}} \ar[drr]_-{\K(c,x)} & 
 U\GG^{\circ} 
 \ar[r]^-{U_{\pmb{\tau}}} & U\GG \ar[d]^-{\Aut_{\K}}
 \\
 & & \Grp 
 }
\] 
 with these four functors commutes.
\vspace{0.1cm}
\item If furthermore the functor $F \colon \GG \to \GG^{\wr}$ is essentially surjective and full, then for every $k \in \zz_{\geq 0}$, writing $\co_{k} \colon \Grp \to \lMod{\zz}$ for the $k$-th group homology functor, there exist unique functors
\begin{align*}
 \co_{k}\pars*{\K(c,x)} &\colon \atsi_{\GG^{\wr}(c^{\wr},x^{\wr})} \to \lMod{\zz}
 \\
 \mathbf{H}_{k} &\colon U\GG^{\wr} \to \lMod{\zz}
\end{align*} that make the diagram
\[
 \xymatrixcolsep{2.1cm}
 \xymatrixrowsep{1.4cm} 
 \xymatrix{
 \atsi_{\GG(c,x)} \ar[r]^-{T_{\GG,c,x}} \ar[drr]_-{\K(c,x)} 
 \ar[ddr]_{\atsi_{\mathbf{F}[\gamma,\xi]}} & 
 U\GG^{\circ} \ar[ddr]_-{U_{\mathbf{F}^{\circ}}}
 \ar[r]^-{U_{\pmb{\tau}}} & U\GG \ar[d]_-{\Aut_{\K}} \ar[ddr]^{U_{\mathbf{F}}}
 \\
 & & \Grp \ar[ddr]^<<<<<<<<<<<<{\co_{k}}
 \\
 & \atsi_{\GG^{\wr}(c^{\wr},x^{\wr})} \ar[r]^-{T_{\GG^{\wr},c^{\wr},x^{\wr}}} 
 \ar[drr]_-{\co_{k}\pars*{\K(c,x)}\,\,\,\,\,\,} & 
 U{\GG^{\wr}}^{\circ} 
 \ar[r]^>>>>>>>{U_{\pmb{\tau}^{\wr}}} & U\GG^{\wr} \ar[d]^-{\mathbf{H}_{k}}
 \\
 & & & \lMod{\zz}  
 } \tag{$\triangle$} \label{prism}
\] 
commute up to natural isomorphisms.
\end{birki}
\end{thm}

\subsection{Proofs} \label{UG-proofs}
\begin{proof}[Proof of \textbf{\emph{Proposition \ref{conj-action-gpd}}}]
 First let us observe that for every morphism $f \colon X \oplus A \to B$ in $\GG$, there is a well-defined map
\begin{align*}
 f \circ \pars*{\id_{X} \oplus -} \circ f^{-1} \colon \Aut_{\K}(A) &\to \Aut_{\K}(B)
 \\
 \kappa &\mapsto f \circ \pars*{\id_{X} \oplus \kappa} \circ f^{-1} \, ,
\end{align*}
that it is a group homomorphism is then straightforward. Indeed, if $\kappa \colon A \to A$ in $\GG$ satisfies $F(\kappa) = \id_{F(A)}$, then
\begin{align*}
 F\!\left(
 f \circ (\id_{X} \oplus \kappa) \circ f^{-1}
 \right)
 &= 
 F(f) \circ F(\id_{X} \oplus \kappa) \circ F(f^{-1}) \, .
\end{align*}
Here due to the natural isomorphism $J$, we have a commutative diagram
\[ 
 \xymatrixrowsep{1.5cm}
 \xymatrixcolsep{2.5cm}
 \xymatrix{
 F(X) \oplus^{\wr} F(A) \ar[d]_-{J_{X,A}} \ar[r]^-{F(\id_{X}) \oplus^{\wr} F(\kappa)}
 &
 F(X) \oplus^{\wr} F(A) \ar[d]^-{J_{X,A}}
 \\
 F(X \oplus A) \ar[r]_-{F\pars*{\id_{X} \oplus \kappa}} & F(X \oplus A) \, ,
 }
\]
so that
\begin{align*}
 F\pars*{\id_{X} \oplus \kappa} &= 
 J_{X,A} \circ \pars*{F(\id_{X}) \oplus^{\wr} F(\kappa)} 
 \circ J_{X,A}^{-1}
 \\
 &=
 J_{X,A} \circ \pars*{\id_{F(X)} \oplus^{\wr} \id_{F(A)}} 
 \circ J_{X,A}^{-1}
 \\
 &=
 J_{X,A} \circ \id_{F(X) \oplus^{\wr} F(A)} 
 \circ J_{X,A}^{-1}
 \\
 &= J_{X,A} \circ J_{X,A}^{-1} = \id_{F(X \oplus A)}
\end{align*}
and hence
\begin{align*}
 F\!\left(
 f \circ (\id_{X} \oplus \kappa) \circ f^{-1}
 \right) 
 &=
 F(f) \circ F(\id_{X} \oplus \kappa) \circ F(f^{-1})
 \\
 &=
 F(f) \circ \id_{F(X \oplus A)}\circ F(f)^{-1} 
 \\
 &= F(f) \circ F(f)^{-1} = \id_{F(B)}
\end{align*}
as desired.

Second, we check that the assignment 
\[
 [X,f] \mapsto f \circ \pars*{\id_{X} \oplus -} \circ f^{-1}
\]
itself is well-defined. Suppose $[X,f] = [\wt{X},\wt{f}]$ in $U\GG$, so there is an isomorphism $\alpha \colon X \to \wt{X}$ in $\GG$ such that
\[
 \wt{f} \circ (\alpha \oplus \id_{A}) = f \, ;
\]
hence
\begin{align*}
  f \circ (\id_{X} \oplus -) \circ f^{-1} 
  &= \left( \wt{f} \circ (\alpha \oplus \id_{A}) \right)
  \circ (\id_{X} \oplus -) \circ
  \left( \wt{f} \circ (\alpha \oplus \id_{A}) \right)^{-1}
  \\
  &= \wt{f} \circ (\alpha \oplus \id_{A})
  \circ (\id_{X} \oplus -) \circ
  (\alpha \oplus \id_{A})^{-1} \circ \wt{f}^{-1}
  \\
  &= \wt{f} \circ (\alpha \oplus-)
   \circ
  (\alpha^{-1} \oplus \id_{A}) \circ \wt{f}^{-1} 
  \\
  &= \wt{f} \oplus (\id_{\wt{X}} \oplus -) \circ \wt{f}^{-1} \, .
\end{align*}

Third, we check the functoriality of $\Aut_{\K}$. Preservation of identities follows from
\[
 \Aut_{\K}[0,\id_{A}] = \id_{A} \circ \pars*{\id_{0} \oplus -} \circ \pars*{\id_{A}}^{-1} = \id_{\Aut_{\K}(A)} \, .
\]
To check preservation of composition, given
\[
 A \xrightarrow{[X,f]} B \xrightarrow{[Y,g]} C
\]
in $U\GG$, for every $\kappa \colon A \to A$ in $\GG$ with $F(\kappa) = \id_{F(A)}$, we observe
\begin{align*}
 \pars*{\Aut_{\K}[Y,g] \circ \Aut_{\K}[X,f]}(\kappa) 
 &=
 \Aut_{\K}[Y,g]\pars*{f \circ \pars*{\id_{X} \oplus \kappa} \circ f^{-1}}
 \\
 &=
 g \circ \pars*{\id_{Y} \oplus 
 \pars*{f \circ \pars*{\id_{X} \oplus \kappa} \circ f^{-1}}} \circ g^{-1}
 \\
 &=
 g \circ 
 \pars*{\id_{Y} \oplus f} \circ \pars*{\id_{Y \oplus X} \oplus \kappa} \circ 
 \pars*{\id_{Y} \oplus f^{-1}} \circ g^{-1}
 \\
 &=
 \pars*{g \circ \pars*{\id_{Y} \oplus f}}
 \circ 
 \pars*{\id_{Y \oplus X} \oplus \kappa} \circ 
 \pars*{g \circ \pars*{\id_{Y} \oplus f}}^{-1}
 \\
 &=
 \Aut_{\K}[Y \oplus X,\, g \circ (\id_{Y} \oplus f)](\kappa)
 \\ 
 &=
 \Aut_{\K}\!\big([Y,g] \circ [X,f]\big)(\kappa) \, ,
\end{align*}
as desired.
\end{proof}

\begin{proof}[Proof of \textbf{\emph{Proposition \ref{A-to-U}}}]
To check the preservation of composition, let
 \[
  m \xrightarrow{\pars*{f,m,n}} n \xrightarrow{\pars*{g,n,p}} p
 \]
 be morphisms in $\atsi_{\GG(c,x)}$, so we want to establish
\begin{align*}
 T_{\GG,c,x}\pars*{g,n,p} \circ T_{\GG,c,x}\pars*{f,m,n} &= 
 T_{\GG,c,x}\pars*{\pars*{g,n,p} \circ \pars*{f,m,n}} \, ,
 \\
 \left[x^{\oplus p-n}, g \right] \circ 
 \left[x^{\oplus n-m}, f \right] &= 
 T_{\GG,c,x}{\pars*{g \,\circ\, \GG(c,x)_{(n,p)}(f),\,m,p}} \, .
\end{align*}
in $U\GG^{\circ}$. By construction we have $\GG(c,x)_{(n,p)}(f) = f \oplus \id_{x^{\oplus p-n}}$, so what we want to establish is
\begin{align*}
 \left[
 x^{\oplus m-p}, g \circ \pars*{f \oplus \id_{x^{\oplus p-n}}}
 \right] &=
 \left[x^{\oplus p-n}, g \right] \circ 
 \left[x^{\oplus n-m}, f \right]
 \\
 &=
 \left[
 x^{\oplus p-n} \oplus^{\circ} x^{\oplus n-m},\,
 g \circ \pars*{\id_{x^{\oplus p-n}} \oplus^{\circ} f}
 \right] 
 \\
 &=
 \left[
 x^{\oplus n-m} \oplus x^{\oplus p-n},\,
 g \circ \pars*{f \oplus \id_{x^{\oplus p-n}}}
 \right] \, ,
\end{align*}
which is evidently true.
\end{proof}

\begin{proof}[Proof of \textbf{\emph{Proposition \ref{UF-oldu}}}]
 To see that $U_{F,J,\psi}$ is well-defined, suppose $(X,f) \sim_{\GG} (X',f')$ so that there is $\alpha \colon X \to X'$ in $\GG$ with $f' \circ (\alpha \oplus \id_{A}) = f$. Applying the functor $F$ we have
 \[
  F(f') \circ F(\alpha \oplus \id_{A}) = F(f)
 \]
 as a morphism $F(X \oplus A) \to F(B)$ in $\GG^{\wr}$. Precomposing with $J_{X,A}$ we get
\begin{align*}
 F(f) \circ J_{X,A} &= F(f') \circ F(\alpha \oplus \id_{A}) \circ J_{X,A}
 \\
 &= F(f') \circ J_{X',A} \circ \pars*{F(\alpha) \oplus^{\wr} F(\id_{A})}
 \\
 &= \pars*{F(f') \circ J_{X',A}} \circ 
 \pars*{F(\alpha) \oplus^{\wr} \id_{F(A)}} \, ,
\end{align*}
by the naturality of $J$, establishing
\[
 \pars*{F(X), \,F(f) \circ J_{X,A}} \sim_{\GG^{\wr}} 
 \pars*{F(X'),\, F(f') \circ J_{X',A}}
\]
as desired.

To check preservation of composition, suppose $[X,f] \colon A \to B$ and $[Y,g] \colon B \to C$ are morphisms in $U\GG$ so we want to show
\begin{align*}
 U_{F,J,\psi}\pars*{[Y,g] \circ [X,f]} &=
 U_{F,J,\psi}[Y,g] \,\circ \, U_{F,J,\psi}[X,f]
 \\
 U_{F,J,\psi}[Y \oplus X,\, g \circ (\id_{Y} \oplus f)] &=
 \left[F(Y),\, F(g) \circ J_{Y,B}\right] \,\circ \, \left[F(X),\, F(f) \circ J_{X,A}\right]
\end{align*}
as morphisms $F(A) \to F(C)$ in $U\GG^{\wr}$. The left hand side in the last equation above is
\[
 \left[
 F(Y \oplus X),\, F\pars*{g \circ (\id_{Y} \oplus f)} \circ J_{Y\oplus X,A}
 \right] \, ,
\]
whereas the right hand side is
\[
 \left[
 F(Y) \oplus^{\wr} F(X),\,
 \pars*{F(g) \circ J_{Y,B}} \circ 
 \pars*{\id_{F(Y)} \oplus^{\wr} \pars*{F(f) \circ J_{X,A}}}
 \right] \, .
\]
The obvious candidate to realize the equality of both sides is $J_{Y,X}$. 
To that end, the monoidality of $(F,J)$ yields
\begin{align*}
 J_{Y \oplus X, A} \circ \pars*{J_{Y,X} \oplus^{\wr} \id_{F(A)}} 
 &= J_{Y,X \oplus A} \circ \pars*{\id_{F(Y)} \oplus^{\wr} J_{X,A}} 
\end{align*}
and the naturality of $J$ yields
\begin{align*}
 F(\id_{Y} \oplus f) \circ J_{Y,X \oplus A} 
 &= J_{Y,B} \circ \pars*{F(\id_{Y}) \oplus^{\wr} F(f)}
 \\
 &= J_{Y,B} \circ \pars*{\id_{F(Y)} \oplus^{\wr} F(f)} 
\end{align*}
so that
\begin{align*}
 F(\id_{Y} \oplus f) \circ J_{Y \oplus X, A} \circ \pars*{J_{Y,X} \oplus^{\wr} \id_{F(A)}}
 &=
 J_{Y,B} \circ \pars*{\id_{F(Y)} \oplus^{\wr} F(f)} 
 \circ \pars*{\id_{F(Y)} \oplus^{\wr} J_{X,A}} 
 \\
 &=
 J_{Y,B} \circ \pars*{\id_{F(Y)} \oplus^{\wr} \pars*{F(f) \circ J_{X,A}}} 
\end{align*}
as desired. To check preservation of identities, note that
\[ 
 U_{F,J,\psi}[0,\id_{A}] = [F(0), F(\id_{A}) \circ J_{0,A}] = [F(0), \id_{F(A)} \circ J_{0,A}] 
 = [F(0), J_{0,A}] = [0^{\wr}, \id_{F(A)}]
\]
where the last equality follows from the property of the isomorphism $\psi \colon F(0) \to 0^{\wr}$ in $\GG^{\wr}$ that \( \id_{F(A)} \circ \pars*{\psi \oplus^{\wr} \id_{F(A)}} = \psi \oplus^{\wr} \id_{F(A)} = J_{0,A}\).

Moving on to the functoriality of \( U \colon \mongpd \to \cat\), let us first note that given 
\[
 (\GG,\oplus,0) \xrightarrow{(F,J,\psi)} (\GG^{\wr},\oplus^{\wr},0^{\wr}) \xrightarrow{(H,L,\eta)} (\GG^{\square},\oplus^{\square},0^{\square}) \, ,
\]
in $\mongpd$, we compose them as
\[
 (H,L,\eta) \circ (F,J,\psi) = (H\circ F,W,\theta)
\]
where $W_{X,Y}$ is the composite
\[
 H(F(X)) \oplus^{\square} H(F(Y)) \xrightarrow{L_{F(X),F(Y)}} 
 H\pars*{F(X) \oplus^{\wr} F(Y)} \xrightarrow{H\pars*{J_{X,Y}}} H(F(X \oplus Y))
\]
for every $X,Y \in \Obj(\GG)$, and $\theta$ is the composite $H(F(0)) \xrightarrow{H(\psi)} H(0^{\wr}) \xrightarrow{\eta} 0^{\square}$. To check preservation of composition, we hence need to verify
\[
 U_{H,L,\eta} \circ U_{F,J,\psi}  = U_{H \circ F, W, \theta} \, .
\]
Indeed, for every morphism $[X,f] \colon A \to B$ in $U\GG$ we have
\begin{align*}
 \pars*{U_{H,L,\eta} \circ U_{F,J,\psi}}[X,f]
 &= U_{H,L,\eta}\left[
 F(X),\, F(f) \circ J_{X,A}
 \right]
 \\
 &= \left[
 H(F(X)),\, H\pars*{F(f) \circ J_{X,A}} \circ L_{F(X),F(A)}
 \right]
 \\
 &= \left[
 H(F(X)),\, H\pars*{F(f)} \circ H\pars*{J_{X,A}} \circ L_{F(X),F(A)}
 \right] 
 \\
 &= \left[
 (H \circ F)(X),\, (H\circ F)(f) \circ W_{X,A}
 \right]  
 \\
 &= U_{H \circ F,W,\theta}[X,f] \, .
\end{align*}
\end{proof}
 
\begin{proof}[Proof of \textbf{\emph{Lemma \ref{cat-epi}}}]
 A relatively routine diagram chase which we omit.
\end{proof} 
 
\begin{proof}[Proof of \textbf{\emph{Theorem \ref{enhance}}}]
 (1): This is in \cite[Proposition 1.8 and its proof]{rw-wahl-stab}. 
 
 (2): Let us first observe the naturality of the collection
 \[ \left\{
 \left[0^{\wr}, J_{A,B}\right] \!\colon F(A) \oplus^{\wr} F(B) \to F(A \oplus B)
 \right\}
 \]
 in $A,B$. Suppose $[X,f] \colon A \to \wh{A}$ and $[Y,g] \colon B \to \wh{B}$ are morphisms in $U\GG$, so we want to show that
 \[
  \xymatrixrowsep{1.5cm}
  \xymatrixcolsep{4.5cm}  
  \xymatrix{
  F(A) \oplus^{\wr} F(B) \ar[d]_-{\left[0^{\wr}, J_{A,B}\right]}
  \ar[r]^-{U_{F,J,\psi}[X,f] \,\,\oplus^{\wr}\,\, U_{F,J,\psi}[Y,g]} &
  F(\wh{A}) \oplus^{\wr} F(\wh{B}) 
  \ar[d]^-{\left[0^{\wr}, J_{\wh{A},\wh{B}}\right]}
  \\
  F(A \oplus B)
  \ar[r]^-{U_{F,J,\psi}\pars*{[X,f] \,\oplus\, [Y,g]}} &
  F(\wh{A} \oplus \wh{B})  
  }
 \]
commutes in $U\GG^{\wr}$. Note that post-composing the morphism
\begin{align*}
 &\,\,\,\,\,U_{F,J,\psi}[X,f] \,\oplus^{\wr}\, U_{F,J,\psi}[Y,g] 
 \\
 &=
 \left[F(X),\,F(f) \circ J_{X,A}\right] \,\oplus^{\wr}\,
 \left[F(Y),\,F(g) \circ J_{Y,B}\right]
 \\
 &=
 \left[
 F(X) \oplus^{\wr} F(Y),
 \pars*{\pars*{F(f) \circ J_{X,A}} \oplus^{\wr} 
 \pars*{F(g) \circ J_{Y,B}}
 }
 \!\circ\!
 \pars*{
 \id_{F(X)} \oplus^{\wr} \pars*{\tau^{\wr}_{F(A),F(Y)}}^{-1} \!\!\oplus^{\wr} \id_{F(B)}
 }
 \right]
 \\
 &=
 \left[
 F(X) \oplus^{\wr} F(Y),
 \pars*{F(f) \oplus^{\wr} F(g)} \!\circ\!
 \pars*{J_{X,A} \oplus^{\wr} J_{Y,B}} \!\circ\!
 \pars*{
 \id_{F(X)} \oplus^{\wr} \pars*{\tau^{\wr}_{F(A),F(Y)}}^{-1} \!\!\oplus^{\wr} \id_{F(B)}
 }
 \right]
\end{align*}
with $\left[0^{\wr}, J_{\wh{A},\wh{B}}\right]$ results in $\left[F(X) \oplus^{\wr} F(Y),\,\Xi \right]$, where
\begin{align*}
 \Xi 
 &= 
 J_{\wh{A},\wh{B}} \circ
 \pars*{F(f) \oplus^{\wr} F(g)} \circ
 \pars*{J_{X,A} \oplus^{\wr} J_{Y,B}} \circ
 \xi
 \\
 &= 
 F(f \oplus g) \circ J_{X \oplus A, Y \oplus B} \circ
 \pars*{J_{X,A} \oplus^{\wr} J_{Y,B}} \circ
 \xi
\end{align*}
and
\[
 \xi =
 \id_{F(X)} \oplus^{\wr} \pars*{\tau^{\wr}_{F(A),F(Y)}}^{-1} \!\!\oplus^{\wr} \id_{F(B)}
\]
in $\GG^{\wr}$. On the other hand, pre-composing the morphism
\begin{align*}
 U_{F,J,\psi}\pars*{[X,f] \,\oplus\, [Y,g]}
 &=
 U_{F,J,\psi}
 \left[
 X \oplus Y,\, (f \oplus g) \circ \pars*{\id_{X} \oplus \tau^{-1}_{A,Y} \oplus \id_{B}}
 \right] 
 \\
 &=
 \left[
 F(X \oplus Y),\, 
 F\pars*{(f \oplus g) \circ \pars*{\id_{X} \oplus \tau^{-1}_{A,Y} \oplus \id_{B}}
 } \circ J_{X\oplus Y,A \oplus B}
 \right] 
 \\
 &=
 \left[
 F(X \oplus Y),\, 
 F\pars*{f \oplus g} \circ F\pars*{\id_{X} \oplus \tau^{-1}_{A,Y} \oplus \id_{B}}
 \circ J_{X\oplus Y,A \oplus B}
 \right] 
\end{align*}
with $\left[0^{\wr}, J_{A,B}\right]$ results in $\left[ F(X \oplus Y),\, \Theta\right]$, where
\begin{align*}
 \Theta &= F\pars*{f \oplus g} \circ F\pars*{\id_{X} \oplus \tau^{-1}_{A,Y} \oplus \id_{B}}
 \circ J_{X\oplus Y,A \oplus B}
 \circ (\id_{F(X \oplus Y)} \oplus^{\wr} J_{A,B})
\end{align*}
in $\GG^{\wr}$. The obvious candidate to realize $\left[F(X) \oplus^{\wr} F(Y),\,\Xi \right] = \left[ F(X \oplus Y),\, \Theta\right]$ is $J_{X,Y}$, that is, it suffices to show that
\[
 \Theta \circ \pars*{J_{X,Y} \oplus^{\wr} \id_{F(A) \oplus^{\wr} F(B)}} = \Xi \, .
\]
To that end, first note that 
\begin{align*}
 (\id_{F(X \oplus Y)} \oplus^{\wr} J_{A,B})
 \circ \pars*{J_{X,Y} \oplus^{\wr} \id_{F(A) \oplus^{\wr} F(B)}}
 &= J_{X,Y} \oplus^{\wr} J_{A,B}
 \\
 &= \pars*{J_{X,Y} \oplus^{\wr} \id_{F(A \oplus B)}}
 \circ \pars*{\id_{F(X) \oplus^{\wr} F(Y)} \oplus^{\wr} J_{A,B}} \, .
\end{align*}
Writing $\mathbf{D}(X,Y,Z)$ for the commutative diagram in Definition \ref{monoidal-F}, we can invoke \[\mathbf{D}(X,Y, A \oplus B)\] to get
\begin{align*}
 &\,\,\,\,\,
 F\pars*{\id_{X} \oplus \tau^{-1}_{A,Y} \oplus \id_{B}}
 \circ J_{X\oplus Y,A \oplus B}
 \circ (\id_{F(X \oplus Y)} \oplus^{\wr} J_{A,B})
 \circ \pars*{J_{X,Y} \oplus^{\wr} \id_{F(A) \oplus^{\wr} F(B)}}
 \\
 &=
 F\pars*{\id_{X} \oplus \tau^{-1}_{A,Y} \oplus \id_{B}}
 \circ J_{X\oplus Y,A \oplus B}
 \circ \pars*{J_{X,Y} \oplus^{\wr} \id_{F(A \oplus B)}}
 \circ \pars*{\id_{F(X) \oplus^{\wr} F(Y)} \oplus^{\wr} J_{A,B}}
 \\
 &=
 F\pars*{\id_{X} \oplus \tau^{-1}_{A,Y} \oplus \id_{B}}
 \circ 
 J_{X,Y \oplus A \oplus B} 
 \circ
 \pars*{\id_{F(X)} \oplus^{\wr} J_{Y,A \oplus B}}
 \circ \pars*{\id_{F(X) \oplus^{\wr} F(Y)} \oplus^{\wr} J_{A,B}} \, .
\end{align*}
We can now use the naturality of $J$ to get
\begin{align*}
 &\,\,\,\,\,
 F\pars*{\id_{X} \oplus \tau^{-1}_{A,Y} \oplus \id_{B}}
 \circ 
 J_{X,Y \oplus A \oplus B} 
 \circ
 \pars*{\id_{F(X)} \oplus^{\wr} J_{Y,A \oplus B}}
 \circ \pars*{\id_{F(X) \oplus^{\wr} F(Y)} \oplus^{\wr} J_{A,B}}
 \\
 &=
 J_{X, A \oplus Y \oplus B}
 \circ
 \pars*{F\pars*{\id_{X}} \oplus^{\wr} F\pars*{\tau_{A,Y}^{-1} \oplus \id_{B}}}
 \circ
 \pars*{\id_{F(X)} \oplus^{\wr} J_{Y,A \oplus B}}
 \circ \pars*{\id_{F(X) \oplus^{\wr} F(Y)} \oplus^{\wr} J_{A,B}} 
 \\
 &=
 J_{X, A \oplus Y \oplus B}
 \circ
 \pars*{\id_{F(X)} \oplus^{\wr} F\pars*{\tau_{A,Y}^{-1} \oplus \id_{B}}}
 \circ
 \pars*{\id_{F(X)} \oplus^{\wr} J_{Y,A \oplus B}}
 \circ \pars*{\id_{F(X) \oplus^{\wr} F(Y)} \oplus^{\wr} J_{A,B}}  
 \\
 &=
 J_{X, A \oplus Y \oplus B}
 \circ
 \bigg(\id_{F(X)} \oplus^{\wr}
 \pars*{\,
 F\pars*{\tau_{A,Y}^{-1} \oplus \id_{B}}
 \circ
 J_{Y,A \oplus B}
 \circ \pars*{\id_{F(Y)} \oplus^{\wr} J_{A,B}}
 \,}\bigg) \, .
\end{align*}
Next, we invoke $\mathbf{D}(Y,A,B)$, the naturality of $J$ and the commutative diagram in Definition \ref{bm-functor} for $A,Y$ to get
\begin{align*}
 &\,\,\,\,\,F\pars*{\tau_{A,Y}^{-1} \oplus \id_{B}}
 \circ
 J_{Y,A \oplus B}
 \circ \pars*{\id_{F(Y)} \oplus^{\wr} J_{A,B}}
 \\
 &=
 F\pars*{\tau_{A,Y}^{-1} \oplus \id_{B}}
 \circ
 J_{Y \oplus A, B}
 \circ 
 \pars*{J_{Y,A} \oplus^{\wr} \id_{F(B)}}
 \\ 
 &=
 J_{A \oplus Y, B}
 \circ 
 \pars*{F\pars*{\tau^{-1}_{A,Y}} \!\oplus^{\wr}\! F\pars*{\id_{B}}}
 \circ
 \pars*{J_{Y,A} \oplus^{\wr} \id_{F(B)}}  
 \\ 
 &=
 J_{A \oplus Y, B}
 \circ 
 \pars*{\pars*{F\pars*{\tau^{-1}_{A,Y}} \circ J_{Y,A}} \oplus^{\wr} \id_{F(B)}}
 \\
 &=
 J_{A \oplus Y, B}
 \circ 
 \pars*{\pars*{J_{A,Y} \circ \pars*{\tau^{\wr}_{F(A),F(Y)}}^{-1}} \oplus^{\wr} \id_{F(B)}} \, .
\end{align*}
Thus via $\mathbf{D}(A,Y,B)$, we have
\begin{align*}
 &\,\,\,\,\,
 \id_{F(X)} \oplus^{\wr}
 \pars*{\,
 F\pars*{\tau_{A,Y}^{-1} \oplus \id_{B}}
 \circ
 J_{Y,A \oplus B}
 \circ \pars*{\id_{F(Y)} \oplus^{\wr} J_{A,B}}
 \,}
 \\
 &=
 \id_{F(X)} \oplus^{\wr}
 \pars*{\,
 J_{A \oplus Y, B}
 \circ 
 \pars*{\pars*{J_{A,Y} \circ \pars*{\tau^{\wr}_{F(A),F(Y)}}^{-1}} \oplus^{\wr} \id_{F(B)}} \,} 
 \\
 &=
 \bigg(\id_{F(X)} \oplus^{\wr}
 \pars*{
 J_{A \oplus Y, B}
 \circ 
 \pars*{
 J_{A,Y} \oplus^{\wr} \id_{F(B)}}
 }\bigg) \circ \xi
 \\
 &=
 \bigg(\id_{F(X)} \oplus^{\wr}
 \pars*{
 J_{A,Y \oplus B}
 \circ 
 \pars*{\id_{F(A)} \oplus^{\wr} J_{Y,B}}
 }\bigg) \circ \xi 
 \\
 &=
 \pars*{\id_{F(X)} \oplus^{\wr} J_{A,Y \oplus B}}
 \circ 
 \pars*{\id_{F(X) \oplus^{\wr} F(A)} \oplus^{\wr} J_{Y,B}}
 \circ \xi  
\end{align*}
and hence via $\mathbf{D}(X,A,Y\oplus B)$, we have
\begin{align*}
 &\,\,\,\,\,
 J_{X, A \oplus Y \oplus B}
 \circ
 \bigg(\id_{F(X)} \oplus^{\wr}
 \pars*{\,
 F\pars*{\tau_{A,Y}^{-1} \oplus \id_{B}}
 \circ
 J_{Y,A \oplus B}
 \circ \pars*{\id_{F(Y)} \oplus^{\wr} J_{A,B}}
 \,}\bigg)
 \\
 &=
 J_{X, A \oplus Y \oplus B}
 \circ
 \pars*{\id_{F(X)} \oplus^{\wr} J_{A,Y \oplus B}}
 \circ 
 \pars*{\id_{F(X) \oplus^{\wr} F(A)} \oplus^{\wr} J_{Y,B}}
 \circ \xi
 \\
 &=
 J_{X \oplus A, Y \oplus B}
 \circ
 \pars*{J_{X,A} \oplus^{\wr} \id_{Y \oplus B}}
 \circ 
 \pars*{\id_{F(X) \oplus^{\wr} F(A)} \oplus^{\wr} J_{Y,B}}
 \circ \xi
 \\
 &= J_{X \oplus A, Y \oplus B} \circ
 \pars*{J_{X,A} \oplus^{\wr} J_{Y,B}} \circ \xi \, .
\end{align*}
It follows that
\begin{align*}
  &\,\,\,\,\,
 F\pars*{\id_{X} \oplus \tau^{-1}_{A,Y} \oplus \id_{B}}
 \circ J_{X\oplus Y,A \oplus B}
 \circ (\id_{F(X \oplus Y)} \oplus^{\wr} J_{A,B})
 \circ \pars*{J_{X,Y} \oplus^{\wr} \id_{F(A) \oplus^{\wr} F(B)}}
 \\
 &= J_{X \oplus A, Y \oplus B} \circ
 \pars*{J_{X,A} \oplus^{\wr} J_{Y,B}} \circ \xi \, .
\end{align*}
Post-composing the above with $F(f \oplus g)$ yields
\[ \Theta \circ \pars*{J_{X,Y} \oplus^{\wr} \id_{F(A) \oplus^{\wr} F(B)}}
 = F(f \oplus g)
 \circ
 J_{X \oplus A, Y \oplus B} \circ
 \pars*{J_{X,A} \oplus^{\wr} J_{Y,B}} \circ \xi
 = \Xi 
\]
as desired.

Second, we shall show that the diagram
\[
 \xymatrixcolsep{3.6cm}
 \xymatrixrowsep{1.6cm} 
 \xymatrix{
 F(A) \oplus^{\wr} F(B) \oplus^{\wr} F(C) 
 \ar[r]^-{\left[0^{\wr},J_{A,B}\right] \,\oplus^{\wr}\, \left[0^{\wr},\id_{F(C)}\right]} 
 \ar[d]_-{\left[0^{\wr},\id_{F(A)}\right] \,\oplus^{\wr}\, \left[0^{\wr},J_{B,C}\right]} 
 &
 F(A \oplus B) \oplus^{\wr} F(C) 
 \ar[d]^-{\left[0^{\wr},\, J_{A \oplus B, C}\right]}  
 \\
 F(A) \oplus^{\wr} F(B \oplus C) 
 \ar[r]^-{\left[0^{\wr},\, J_{A,B \oplus C}\right]}  
 & F(A \oplus B \oplus C)
 }
 \]
in $U\GG^{\wr}$ commutes: first note that
\begin{align*}
 &\,\,\,\,\,\left[0^{\wr},J_{A,B}\right] \,\oplus^{\wr}\, \left[0^{\wr},\id_{F(C)}\right]
 \\
 &=
 \left[
  0^{\wr} \oplus^{\wr} 0^{\wr},\, \pars*{J_{A,B} \oplus^{\wr} \id_{F(C)}} \circ
  \pars*{\id_{0^{\wr}} \oplus^{\wr} \pars*{\tau^{\wr}_{F(A) \oplus^{\wr} F(B),\,0^{\wr}}}^{-1}
  \oplus^{\wr} \id_{F(C)}
  }
 \right]
 \\
 &=
 \left[
  0^{\wr},\, \pars*{J_{A,B} \oplus^{\wr} \id_{F(C)}} \circ
  \pars*{
  \id_{F(A) \oplus^{\wr} F(B)}
  \oplus^{\wr} \id_{F(C)}
  }
 \right] 
 \\
 &=
 \left[
  0^{\wr},\, J_{A,B} \oplus^{\wr} \id_{F(C)}
 \right]  
\end{align*}
and
\begin{align*}
 &\,\,\,\,\,\left[0^{\wr},\id_{F(A)}\right] \,\oplus^{\wr}\, \left[0^{\wr},J_{B,C}\right]
 \\
 &= \left[
 0^{\wr} \oplus^{\wr} 0^{\wr},\,
 \pars*{\id_{F(A)} \oplus^{\wr} J_{B,C}} 
 \circ
 \pars*{\id_{0^{\wr}} \oplus^{\wr} \pars*{\tau_{F(A),0^{\wr}}}^{-1} \oplus^{\wr} 
 \id_{F(B) \oplus^{\wr} F(C)}}
 \right]
 \\
 &= \left[
 0^{\wr},\,
 \pars*{\id_{F(A)} \oplus^{\wr} J_{B,C}} 
 \circ
 \pars*{\id_{F(A)} \oplus^{\wr} 
 \id_{F(B) \oplus^{\wr} F(C)}}
 \right] 
 \\
 &= \left[
 0^{\wr},\,
 \id_{F(A)} \oplus^{\wr} J_{B,C} 
 \right]  
\end{align*}
so we reduce to showing
\[
 \xymatrixcolsep{3.6cm}
 \xymatrixrowsep{1.6cm} 
 \xymatrix{
 F(A) \oplus^{\wr} F(B) \oplus^{\wr} F(C) 
 \ar[r]^-{\left[
  0^{\wr},\, J_{A,B} \oplus^{\wr} \id_{F(C)}
 \right]} 
 \ar[d]_-{\left[
 0^{\wr},\,
 \id_{F(A)} \oplus^{\wr} J_{B,C} 
 \right]} 
 &
 F(A \oplus B) \oplus^{\wr} F(C) 
 \ar[d]^-{\left[0^{\wr},\, J_{A \oplus B, C}\right]}  
 \\
 F(A) \oplus^{\wr} F(B \oplus C) 
 \ar[r]^-{\left[0^{\wr},\, J_{A,B \oplus C}\right]}  
 & F(A \oplus B \oplus C)
 }
\]
commutes in $U\GG^{\wr}$, which follows immediately from the commutativity of
\[
 \xymatrixcolsep{3.6cm}
 \xymatrixrowsep{1.6cm} 
 \xymatrix{
 F(A) \oplus^{\wr} F(B) \oplus^{\wr} F(C) 
 \ar[r]^-{
  J_{A,B} \oplus^{\wr} \id_{F(C)}} 
 \ar[d]_-{
 \id_{F(A)} \oplus^{\wr} J_{B,C} 
 } 
 &
 F(A \oplus B) \oplus^{\wr} F(C) 
 \ar[d]^-{J_{A \oplus B, C}}  
 \\
 F(A) \oplus^{\wr} F(B \oplus C) 
 \ar[r]^-{J_{A,B \oplus C}}  
 & F(A \oplus B \oplus C)
 }
\]
in $\GG^{\wr}$ as $(F,J,\psi)$ is monoidal (Definition \ref{monoidal-F}). Similarly, the properties of $[0^{\wr},\psi]$ follow from those of $\psi$.

(3): Given 
\[
 \pars*{\GG,\oplus,0,\tau} \xrightarrow{(F,J,\psi)} 
 \pars*{\GG^{\wr},\oplus^{\wr},0^{\wr},\tau^{\wr}} \xrightarrow{(H,L,\eta)} 
 \pars*{\GG^{\square},\oplus^{\square},0^{\square}, \tau^{\square}} \, ,
\]
in $\bmgpd$, since the composition in $\bmgpd$ is simply inherited from $\mongpd$ (functors being braided is a property, not extra structure), we need to verify
\[
 \pars*{U_{H,L,\eta}, \left[0^{\square},L\right], \left[0^{\square},\eta\right]}
 \circ 
 \pars*{U_{F,J,\psi}, \left[0^{\wr},J\right], \left[0^{\wr},\psi\right]} 
 = 
 \pars*{U_{H \circ F, W, \theta}, \left[0^{\square},W\right], \left[0^{\square},\theta\right]} 
 \tag{$\star$} \label{alla}
\]
where $W_{X,Y}$ is the composite
\[
 H(F(X)) \oplus^{\square} H(F(Y)) \xrightarrow{L_{F(X),F(Y)}} 
 H\pars*{F(X) \oplus^{\wr} F(Y)} \xrightarrow{H\pars*{J_{X,Y}}} H(F(X \oplus Y))
\]
for every $X,Y \in \Obj(\GG)$, and $\theta$ is the composite $H(F(0)) \xrightarrow{H(\psi)} H(0^{\wr}) \xrightarrow{\eta} 0^{\square}$ (both composites happen in $\GG^{\square}$). On the other hand, the left hand side of (\ref{alla}) is
\[
 \pars*{U_{H,L,\eta} \circ U_{F,J,\psi},\, \mathbf{W}, \mathbf{t}}
\]
where $\mathbf{W}_{X,Y}$ is the composite
\[
 H(F(X)) \oplus^{\square} H(F(Y)) \xrightarrow{\left[0^{\square}, L_{F(X),F(Y)}\right]} 
 H\pars*{F(X) \oplus^{\wr} F(Y)} \xrightarrow{U_{H,L,\eta}\left[0^{\wr},J_{X,Y}\right]} H(F(X \oplus Y))
\]
for every $X,Y \in \Obj(U\GG) = \Obj(\GG)$, and $\mathbf{t}$ is the composite 
\[
  H(F(0)) \xrightarrow{U_{H,L,\eta}[0^{\wr},\psi]} H(0^{\wr}) 
  \xrightarrow{[0^{\square},\eta]} 0^{\square}
\]
(both composites happen in $U\GG^{\square}$). By construction (Proposition \ref{UF-oldu}) we have
\[
 U_{H,L,\eta}[0^{\wr},J_{X,Y}] = \left[
 H(0^{\wr}), H(J_{X,Y}) \circ L_{0^{\wr},F(X) \oplus^{\wr} F(Y)}
 \right]
\]
and
\[
 U_{H,L,\eta}[0^{\wr},\psi] = \left[
 H(0^{\wr}), H(\psi) \circ L_{0^{\wr},F(0)}
 \right] \, .
\]
Since $L_{0^{\wr},F(0)} = \eta \oplus^{\square} \id_{H(F(0))}$, the diagram
\[
 \xymatrixcolsep{3cm}
 \xymatrixrowsep{1.2cm} 
 \xymatrix{
 H(0^{\wr}) \oplus^{\square} H(F(0)) 
 \ar[r]^-{
  L_{0^{\wr},F(0)}} 
 \ar[d]_-{
 \eta \,\oplus^{\square}\, \id_{H(F(0))} 
 } 
 &
 H(0^{\wr} \oplus^{\wr} F(0)) = H(F(0)) 
 \ar[d]^-{H(\psi)}  
 \\
 0^{\square} \oplus^{\square} H(F(0)) = H(F(0)) 
 \ar[r]^-{H(\psi)}  
 & H(0^{\wr})
 }
\]
commutes in $\GG^{\square}$, and hence
\begin{align*}
 U_{H,L,\eta}[0^{\wr},\psi] = \left[
 H(0^{\wr}), H(\psi) \circ L_{0^{\wr},F(0)}
 \right] = \left[ 0^{\square}, H(\psi)\right]
\end{align*}
in $U\GG^{\square}$. In a similar fashion, we also have
\begin{align*}
 U_{H,L,\eta}[0^{\wr},J_{X,Y}] 
 &= \left[
 H(0^{\wr}), H(J_{X,Y}) \circ L_{0^{\wr},F(X) \oplus^{\wr} F(Y)}
 \right]
 \\
 &=\left[
 0^{\square}, H(J_{X,Y})
 \right]
\end{align*}
From here, we get the equalities $\mathbf{W} = \left[0^{\square}, W\right]$ and $\mathbf{t} = \left[0^{\square}, \theta \right]$. The equality 
\[
  U_{H,L,\eta} \circ U_{F,J,\psi} = U_{H \circ F, W, \theta}
\]
was already checked in the proof of Proposition \ref{UF-oldu}. These three equalities together establish (\ref{alla}).
\end{proof}

\begin{proof}[Proof of \textbf{\emph{Proposition \ref{conj-action-Hk}}}]
 For every $A \in \Obj(\GG^{\wr})$, choose $\ov{A} \in \Obj(\GG)$ with an isomorphism
 \[
  \mu_{A} \colon F(\ov{A}) \to A \, .
 \]
in $\GG^{\wr}$. Hence on objects, we can define
\begin{align*}
  \mathbf{H}_{k} \colon U\GG^{\wr} &\to \lMod{\zz}
  \\
  A &\mapsto \co_{k}(\Aut_{\K}(\ov{A}))
\end{align*}
 Moreover, for every $g \colon A \to B$ in $\GG^{\wr}$, as $F$ is full, we may pick $\ov{g} \colon \ov{A} \to \ov{B}$ in $\GG$ such that
\[
 \xymatrixcolsep{2.1cm}
 \xymatrixrowsep{1.4cm} 
 \xymatrix{
 F(\ov{A}) \ar[r]^{F(\ov{g})} \ar[d]_-{\mu_{A}} &
 F(\ov{B}) \ar[d]^-{\mu_{B}}
 \\
 A \ar[r]_{g} & B 
 }
\] 
commutes in $\GG^{\wr}$. Moreover due to the chain of (iso)morphisms
\[
 \xymatrixcolsep{2.1cm}
 \xymatrixrowsep{1.4cm} 
 \xymatrix{
 F\pars*{\ov{A} \oplus \ov{B}}  &
 F\pars*{\ov{A}} \oplus^{\wr} F\pars*{\ov{B}} \ar[l]_-{J_{\ov{A},\ov{B}}} 
 \ar[r]^-{\mu_{A} \oplus^{\wr} \mu_{B}} & A \oplus^{\wr} B &
 F\pars*{\ov{A \oplus^{\wr} B}} \ar[l]_{\mu_{A \oplus^{\wr} B}}
 }
\] 
in $\GG^{\wr}$ and $F$ being full, there exists (for every $A,B \in \Obj(\GG^{\wr})$) a morphism
\[
 \eta_{A,B} \colon \ov{A} \oplus \ov{B} \to \ov{A \oplus^{\wr} B}
\]
in $\GG$ such that
\[
 F(\eta_{A,B}) = \mu_{A \oplus^{\wr} B}^{-1} \circ 
 \pars*{\mu_{A} \oplus^{\wr} \mu_{B}} \circ J_{\ov{A},\ov{B}}^{-1}
\]
Hence given $[X,f] \colon A \to B$ in $U\GG^{\wr}$, that is a morphism $f \colon X \oplus^{\wr} A \to B$ in $\GG^{\wr}$, we can form the composite
\[
 \xymatrixcolsep{2.1cm}
 \xymatrixrowsep{1.4cm} 
 \xymatrix{
 \ov{X} \oplus \ov{A} \ar[r]^-{\eta_{X,A}}
 &
 \ov{X \oplus^{\wr} A} \ar[r]^-{\ov{f}}  &
 \ov{B}
 }
\] 
in $\GG$, hence a morphism $\left[\,\ov{X}, \ov{f} \circ \eta_{X,A}\right] \colon \ov{A} \to \ov{B}$ in $U\GG$. Therefore on morphisms, we can consider the assignment
\begin{align*}
 \mathbf{H}_{k} \colon \Hom_{U\GG^{\wr}}(A,B) &\to 
 \Hom_{\zz}\pars*{\co_{k}\pars*{\Aut_{\K}(\ov{A})},\,
 \co_{k}\pars*{\Aut_{\K}(\ov{B})}}
 \\
 [X,f] &\mapsto \co_{k}\pars*{\Aut_{\K}\!\left[\,\ov{X}, \ov{f} \circ \eta_{X,A}\right]} \, .
\end{align*}
Let us check that the above is well-defined: say $[X,f] = [X',f']$ in $U\GG^{\wr}$, so there is $\alpha \colon X \to X'$ in $\GG^{\wr}$ with $f' \circ (\alpha \oplus^{\wr} \id_{A}) = f$. Here given
\[
 \kappa \in \Aut_{\K}\pars*{\ov{A}} \, ,
\]
we have
\begin{align*}
 \Aut_{\K}\!\left[\,\ov{X}, \ov{f} \circ \eta_{X,A}\right](\kappa)
 &=
 \pars*{\ov{f} \circ \eta_{X,A}} \circ \pars*{\id_{\ov{X}} \oplus \kappa} 
 \circ \pars*{\ov{f} \circ \eta_{X,A}}^{-1}
 \\
 &=
 \ov{f} \circ \eta_{X,A} \circ \pars*{\id_{\ov{X}} \oplus \kappa} 
 \circ \eta_{X,A}^{-1} \circ \ov{f}^{-1}
\end{align*}
as an element of $\Aut_{\K}(\ov{B})$, which we want to compare with
\begin{align*}
 \Aut_{\K}\!\left[\,\ov{X'}, \ov{f'} \circ \eta_{X',A}\right](\kappa)
 =
 \ov{f'} \circ \eta_{X',A} \circ \pars*{\id_{\ov{X'}} \oplus \kappa} 
 \circ \eta_{X',A}^{-1} \circ \ov{f'}^{-1} \, .
\end{align*}
We shall show that the composite
\[
 \xymatrixcolsep{2.1cm}
 \xymatrixrowsep{1.4cm} 
 \xymatrix{
 \ov{X} \oplus \ov{A} \ar[d]_-{\ov{\alpha} \oplus \id_{\ov{A}}}
 & \ov{X \oplus^{\wr} A} \ar[l]_-{\eta_{X,A}^{-1}} 
 & \ov{B} \ar[l]_-{\ov{f}^{-1}} 
 \\
 \ov{X'} \oplus \ov{A} \ar[r]^-{\eta_{X',A}}
 & \ov{X' \oplus^{\wr} A} \ar[r]^-{\ov{f'}} 
 & \ov{B}
 }
\] 
 in $\GG$, that we shall denote by $\lambda \in \Aut_{\GG}(\ov{B})$, actually lies in $\K = \ker F$. Indeed, 
\begin{align*}
 F\pars*{\pars*{\ov{\alpha} \oplus \id_{\ov{A}}} \circ \eta_{X,A}^{-1}} &=
 F\pars*{\ov{\alpha} \oplus \id_{\ov{A}}} \circ F\pars*{\eta_{X,A}}^{-1}
 \\
 &=
 F\pars*{\ov{\alpha} \oplus \id_{\ov{A}}} \circ J_{\ov{X},\ov{A}}
 \circ \pars*{\mu_{X} \oplus^{\wr} \mu_{A}}^{-1} \circ \mu_{X \oplus^{\wr} A}
\end{align*}
Here the naturality of $J$ yields a commutative diagram
\[
 \xymatrixcolsep{3.1cm}
 \xymatrixrowsep{1.4cm} 
 \xymatrix{
 F\pars*{\ov{X}} \oplus^{\wr} F\pars*{\ov{A}} 
 \ar[r]^-{F\pars*{\ov{\alpha}} \oplus^{\wr} F\pars*{\id_{\ov{A}}}} 
 \ar[d]_-{J_{\ov{X},\ov{A}}} &
 F\pars*{\ov{X'}} \oplus^{\wr} F\pars*{\ov{A}}
 \ar[d]^-{J_{\ov{X'},\ov{A}}} 
 \\
 F\pars*{\ov{X} \oplus \ov{A}} \ar[r]^-{F\pars*{\ov{\alpha} \oplus \id_{\ov{A}}}} &
 F\pars*{\ov{X'} \oplus \ov{A}} }
\] 
so that
\begin{align*}
 F\pars*{\pars*{\ov{\alpha} \oplus \id_{\ov{A}}} \circ \eta_{X,A}^{-1}}
 &=
 J_{\ov{X'},\ov{A}} 
 \circ \pars*{F\pars*{\ov{\alpha}} \oplus^{\wr} F\pars*{\id_{\ov{A}}}}
 \circ \pars*{\mu_{X}^{-1} \oplus^{\wr} \mu_{A}^{-1}} \circ \mu_{X \oplus^{\wr} A}
 \\
 &=
 J_{\ov{X'},\ov{A}} 
 \circ \pars*{\pars*{F\pars*{\ov{\alpha}} \circ \mu_{X}^{-1}} \oplus^{\wr} \mu_{A}^{-1}}
 \circ \mu_{X \oplus^{\wr} A} \, .
\end{align*}
Since
\[
 \xymatrixcolsep{2.1cm}
 \xymatrixrowsep{1.4cm} 
 \xymatrix{
 F\pars*{\ov{X}} \ar[r]^-{F(\ov{\alpha})} \ar[d]_-{\mu_{X}} &
 F(\ov{X'}) \ar[d]^-{\mu_{X'}}
 \\
 X \ar[r]_-{\alpha} & X' 
 }
\]
commutes in $\GG^{\wr}$ by construction, we get
\[
 F\pars*{\pars*{\ov{\alpha} \oplus \id_{\ov{A}}} \circ \eta_{X,A}^{-1}}
 =
 J_{\ov{X'},\ov{A}} 
 \circ \pars*{\pars*{\mu_{X'}^{-1} \circ \alpha} \oplus^{\wr} \mu_{A}^{-1}}
 \circ \mu_{X \oplus^{\wr} A} \, .
\]
Next, note that
\[ 
 F\pars*{\eta_{X',A}} \circ J_{\ov{X'},\ov{A}} 
 = \mu_{X' \oplus^{\wr} A}^{-1} \circ 
 \pars*{\mu_{X'} \oplus^{\wr} \mu_{A}} \, ,
\]
so that
\begin{align*}
 F\pars*{\eta_{X',A} 
 \circ \pars*{\ov{\alpha} \oplus \id_{\ov{A}}} \circ \eta_{X,A}^{-1}}
 &=
 \mu_{X' \oplus^{\wr} A}^{-1} \circ 
 \pars*{\mu_{X'} \oplus^{\wr} \mu_{A}} 
 \circ \pars*{\pars*{\mu_{X'}^{-1} \circ \alpha} \oplus^{\wr} \mu_{A}^{-1}}
 \circ \mu_{X \oplus^{\wr} A} 
 \\
 &=
 \mu_{X' \oplus^{\wr} A}^{-1} 
 \circ \pars*{\alpha \oplus^{\wr} \id_{A}}
 \circ \mu_{X \oplus^{\wr} A} \, .
\end{align*}
Since
\[
 \xymatrixcolsep{2.1cm}
 \xymatrixrowsep{1.4cm} 
 \xymatrix{
 F\pars*{\ov{X' \oplus^{\wr} A}} \ar[r]^-{F(\ov{f'})} \ar[d]_-{\mu_{X' \oplus^{\wr} A}} &
 F(\ov{B}) \ar[d]^-{\mu_{B}}
 \\
 X' \oplus^{\wr} A \ar[r]_-{f'} & B 
 }
\]
commutes in $\GG^{\wr}$ by construction, we have
\begin{align*}
 F\pars*{\ov{f'} \circ \eta_{X',A} 
 \circ \pars*{\ov{\alpha} \oplus \id_{\ov{A}}} \circ \eta_{X,A}^{-1}}
 &=
 F\pars*{\ov{f'}} \circ
 \mu_{X' \oplus^{\wr} A}^{-1} 
 \circ \pars*{\alpha \oplus^{\wr} \id_{A}}
 \circ \mu_{X \oplus^{\wr} A}
 \\
 &=
 \mu_{B}^{-1} \circ f' 
 \circ \pars*{\alpha \oplus^{\wr} \id_{A}}
 \circ \mu_{X \oplus^{\wr} A} 
 \\
 &=
 \mu_{B}^{-1} \circ f
 \circ \mu_{X \oplus^{\wr} A} \, .
\end{align*}
Since
\[
 \xymatrixcolsep{2.1cm}
 \xymatrixrowsep{1.4cm} 
 \xymatrix{
 F\pars*{\ov{X \oplus^{\wr} A}} \ar[r]^-{F(\ov{f})} \ar[d]_-{\mu_{X \oplus^{\wr} A}} &
 F(\ov{B}) \ar[d]^-{\mu_{B}}
 \\
 X \oplus^{\wr} A \ar[r]_-{f} & B 
 }
\]
commutes in $\GG^{\wr}$ by construction, we have
\begin{align*}
 F(\lambda) &= 
 F\pars*{\ov{f'} \circ \eta_{X',A} 
 \circ \pars*{\ov{\alpha} \oplus \id_{\ov{A}}} \circ \eta_{X,A}^{-1}}
 \circ F\pars*{\ov{f}}^{-1}
 \\
 &=
 \mu_{B}^{-1} \circ f
 \circ \mu_{X \oplus^{\wr} A} \circ F\pars*{\ov{f}}^{-1}
 \\
 &= \mu_{B}^{-1} \circ f \circ f^{-1} \circ \mu_{B} = \id_{F\pars*{\ov{B}}} \, ,
\end{align*}
verifying 
\[
 \ov{f'} \circ \eta_{X',A} 
 \circ \pars*{\ov{\alpha} \oplus \id_{\ov{A}}} \circ \eta_{X,A}^{-1}
 \circ \ov{f}^{-1} =
 \lambda \in \Aut_{\K}\pars*{\ov{B}}
\]
 as desired. Considering the conjugation map
\begin{align*}
 c_{\lambda} \colon \Aut_{\K}(\ov{B}) &\to \Aut_{\K}(\ov{B})
 \\
 h &\mapsto \lambda \circ h \circ \lambda^{-1} \, ,
\end{align*}
for every $\kappa \in \Aut_{\K}\pars*{\ov{A}}$ we have
\begin{align*}
 \pars*{c_{\lambda} \circ \Aut_{\K}\!\left[\,\ov{X}, \ov{f} \circ \eta_{X,A}\right]}(\kappa)
 &=
 \lambda \circ \ov{f} \circ \eta_{X,A} \circ \pars*{\id_{\ov{X}} \oplus \kappa} 
 \circ \eta_{X,A}^{-1} \circ \ov{f}^{-1} \circ \lambda^{-1}
 \\
 &= \pars*{\lambda \circ \ov{f} \circ \eta_{X,A}} 
 \circ \pars*{\id_{\ov{X}} \oplus \kappa}
 \circ \pars*{\lambda \circ \ov{f} \circ \eta_{X,A}}^{-1}
 \\
 &= \pars*{\ov{f'} \circ \eta_{X',A} 
 \circ \pars*{\ov{\alpha} \oplus \id_{\ov{A}}}} 
 \circ \pars*{\id_{\ov{X}} \oplus \kappa}
 \circ \pars*{\lambda \circ \ov{f} \circ \eta_{X,A}}^{-1}
 \\
 &= \ov{f'} \circ \eta_{X',A} 
 \circ \pars*{\ov{\alpha} \oplus \kappa}
 \circ \pars*{\lambda \circ \ov{f} \circ \eta_{X,A}}^{-1}
 \\
 &= \ov{f'} \circ \eta_{X',A} 
 \circ \pars*{\ov{\alpha} \oplus \kappa}
 \circ \pars*{\ov{f'} \circ \eta_{X',A} 
 \circ \pars*{\ov{\alpha} \oplus \id_{\ov{A}}}}^{-1} 
 \\
 &= \ov{f'} \circ \eta_{X',A} 
 \circ \pars*{\ov{\alpha} \oplus \kappa}
 \circ \pars*{\ov{\alpha}^{-1} \oplus \id_{\ov{A}}} \circ \eta_{X',A}^{-1}
 \circ \ov{f'}^{-1}
 \\
 &= \ov{f'} \circ \eta_{X',A} 
 \circ \pars*{\id_{\ov{X'}} \oplus \kappa}
 \circ \eta_{X',A}^{-1}
 \circ \ov{f'}^{-1} 
 \\
 &= \Aut_{\K}\!\left[\,\ov{X'}, \ov{f'} \circ \eta_{X',A}\right](\kappa) \, ,
\end{align*} 
and hence
\[
 c_{\lambda} \circ \Aut_{\K}\!\left[\,\ov{X}, \ov{f} \circ \eta_{X,A}\right]
 = \Aut_{\K}\!\left[\,\ov{X'}, \ov{f'} \circ \eta_{X',A}\right] \, .
\]
Since 
\[
  \co_{k}(c_{\lambda}) \colon \co_{k}\pars*{\Aut_{\K}\pars*{\ov{B}}} 
  \to \co_{k}\pars*{\Aut_{\K}\pars*{\ov{B}}}
\]
is the identity map by \cite[Corollary II.6.2]{brown-coh-book}, we have
\begin{align*}
 \co_{k}\pars*{\Aut_{\K}\!\left[\,\ov{X'}, \ov{f'} \circ \eta_{X',A}\right]}
 &=
 \co_{k}\pars*{c_{\lambda} \circ \Aut_{\K}\!\left[\,\ov{X}, \ov{f} \circ \eta_{X,A}\right]}
 \\
 &=
 \co_{k}\pars*{c_{\lambda}} \circ 
 \co_{k}\pars*{\Aut_{\K}\!\left[\,\ov{X}, \ov{f} \circ \eta_{X,A}\right]}
 \\
 &=
 \co_{k}\pars*{\Aut_{\K}\!\left[\,\ov{X}, \ov{f} \circ \eta_{X,A}\right]} 
\end{align*}
as desired. 

Now that we have established the well-definition of $\mathbf{H}_{k}$, let us observe its functoriality. The identity map 
\(
 [0^{\wr},\id_{A}] \colon A \to A
\) in $U\GG^{\wr}$ is sent to
\[
 \mathbf{H}_{k}[0^{\wr},\id_{A}] = \co_{k}\pars*{\Aut_{\K}\!\left[
 \ov{0^{\wr}}, \ov{\id_{A}} \circ \eta_{0^{\wr},A}
 \right]}
\]
Since $F$ is full, there exists $\zeta \colon 0 \to \ov{0^{\wr}}$ in $\GG$ such that $F(\zeta)$ is the composite
\[
 F(0) \xrightarrow{\psi} 0^{\wr} \xrightarrow{\mu_{0^{\wr}}^{-1}} F\pars*{\ov{0^{\wr}}}
\]
in $\GG^{\wr}$. The morphism
\[
 \eta_{0^{\wr},A} \colon \ov{0^{\wr}} \oplus \ov{A} \to \ov{0^{\wr} \oplus^{\wr} A} = \ov{A}
\]
in $\GG$ by construction satisfies
\begin{align*}
 F(\eta_{0^{\wr},A}) &= \mu_{0^{\wr} \oplus^{\wr} A}^{-1} 
 \circ 
 \pars*{\mu_{0^{\wr}} \oplus^{\wr} \mu_{A}} \circ J_{\ov{0^{\wr}},\ov{A}}^{-1}
 \\ &=
 \mu_{A}^{-1} 
 \circ 
 \pars*{\mu_{0^{\wr}} \oplus^{\wr} \mu_{A}} \circ J_{\ov{0^{\wr}},\ov{A}}^{-1} \, .
\end{align*}
The naturality of $J$ yields a commutative diagram
\[
 \xymatrixcolsep{3.1cm}
 \xymatrixrowsep{1.4cm} 
 \xymatrix{
 F(0) \oplus^{\wr} F\pars*{\ov{A}} 
 \ar[r]^-{F(\zeta) \oplus^{\wr} F\pars*{\id_{\ov{A}}}} 
 \ar[d]_-{J_{0,\ov{A}}} &
 F\pars*{\ov{0^{\wr}}} \oplus^{\wr} F\pars*{\ov{A}}
 \ar[d]^-{J_{\ov{0^{\wr}},\ov{A}}} 
 \\
 F\pars*{\ov{A}} = F\pars*{0 \oplus \ov{A}} \ar[r]^-{F\pars*{\zeta \oplus \id_{\ov{A}}}} &
 F\pars*{\ov{0^{\wr}} \oplus \ov{A}}\, , } 
\] 
hence 
\begin{align*}
 F\pars*{\eta_{0^{\wr},A} \circ \pars*{\zeta \oplus \id_{\ov{A}}}}
 &=
 \mu_{A}^{-1} 
 \circ 
 \pars*{\mu_{0^{\wr}} \oplus^{\wr} \mu_{A}} \circ J_{\ov{0^{\wr}},\ov{A}}^{-1}
 \circ F\pars*{\zeta \oplus \id_{\ov{A}}}
 \\
 &=
 \mu_{A}^{-1} 
 \circ 
 \pars*{\mu_{0^{\wr}} \oplus^{\wr} \mu_{A}} 
 \circ \pars*{F(\zeta) \oplus^{\wr} F\pars*{\id_{\ov{A}}}}
 \circ J_{0,\ov{A}}^{-1} 
 \\
 &=
 \mu_{A}^{-1} 
 \circ 
 \pars*{\pars*{\mu_{0^{\wr}} \circ F(\zeta)} \oplus^{\wr} \mu_{A}} 
 \circ J_{0,\ov{A}}^{-1}  
 \\
 &=
 \pars*{\id_{0^{\wr}} \oplus^{\wr} \mu_{A}^{-1}}
 \circ 
 \pars*{\psi \oplus^{\wr} \mu_{A}} 
 \circ J_{0,\ov{A}}^{-1}   
 \\
 &=
 \pars*{\psi \oplus^{\wr} \id_{F\pars*{\ov{A}}}} 
 \circ J_{0,\ov{A}}^{-1} = \id_{F\pars*{\ov{A}}}
\end{align*}
by the defining properties of $(F, J,\psi)$ in Definition \ref{monoidal-F}. On the other hand, by construction
\[
 \xymatrixcolsep{2.1cm}
 \xymatrixrowsep{1.4cm} 
 \xymatrix{
 F(\ov{A}) \ar[r]^{F\pars*{\ov{\id_{A}}}} \ar[d]_-{\mu_{A}} &
 F(\ov{A}) \ar[d]^-{\mu_{A}}
 \\
 A \ar[r]_{\id_{A}} & A 
 }
\] 
commutes in $\GG^{\wr}$. As $\mu_A$ (like every morphism in $\GG^{\wr}$) is an isomorphism, we have 
\[ F\pars*{\ov{\id_{A}}} = \id_{F\pars*{\ov{A}}} \, , \]
hence $\ov{\id_{A}} \in \Aut_{\K}\pars*{\ov{A}}$.
It follows that the morphism 
\[
 \kappa \coloneqq \ov{\id_{A}} \circ 
 \eta_{0^{\wr},A} \circ \pars*{\zeta \oplus \id_{\ov{A}}} \colon \ov{A} \to \ov{A}
\]
in $\GG$ is in fact in $\K$, with
\[
 \left[ \ov{0^{\wr}}, \ov{\id_{A}} \circ \eta_{0^{\wr},A} \right] 
 = \left[0,\kappa\right] \colon \ov{A} \to \ov{A}
\]
in $U\GG$; consequently the group homomorphism
\[
 \Aut_{\K}\!\left[ \ov{0^{\wr}}, \ov{\id_{A}} \circ \eta_{0^{\wr},A} \right] = \Aut_{\K}[0,\kappa] \colon \Aut_{\K}\pars*{\ov{A}} \to \Aut_{\K}\pars*{\ov{A}}
\]
is conjugation by $\kappa \in \Aut_{\K}\pars*{\ov{A}}$. Therefore
\[
 \mathbf{H}_{k}[0^{\wr},\id_{A}] = \co_{k}\pars*{\Aut_{\K}\!\left[
 \ov{0^{\wr}}, \ov{\id_{A}} \circ \eta_{0^{\wr},A}
 \right]} = \id_{\co_{k}\pars*{\Aut_{\K}\pars*{\ov{A}}}}
\]
by \cite[Corollary II.6.2]{brown-coh-book} as desired.

To check $\mathbf{H}_{k}$ preserves composition, let 
\[
A \xrightarrow{[X,f]} B \xrightarrow{[Y,g]} C
\] 
be morphisms in $U\GG^{\wr}$. On one hand, by the functoriality of $\co_{k}$ and $\Aut_{\K}$ we have
\begin{align*}
 \mathbf{H}_{k}[Y,g] \circ \mathbf{H}_{k}[X,f] 
 &=
 \co_{k}\pars*{\Aut_{\K}\!\left[\,\ov{Y}, \ov{g} \circ \eta_{Y,B}\right]}
 \circ
 \co_{k}\pars*{\Aut_{\K}\!\left[\,\ov{X}, \ov{f} \circ \eta_{X,A}\right]}
 \\
 &=
 \co_{k}\pars*{
 \Aut_{\K} \pars*{
 \left[\,\ov{Y}, \ov{g} \circ \eta_{Y,B}\right] \circ
 \left[\,\ov{X}, \ov{f} \circ \eta_{X,A}\right]
 }
 }
\end{align*}
where the morphisms
\[
  \ov{A} \xrightarrow{\left[\ov{X},\, \ov{f} \circ \eta_{X,A}\right]} \ov{B} 
  \xrightarrow{\left[\ov{Y},\, \ov{g} \circ \eta_{Y,B}\right]} \ov{C}
\] 
in $U\GG$ compose as
\begin{align*}
 \left[\,\ov{Y}, \ov{g} \circ \eta_{Y,B}\right] \circ
 \left[\,\ov{X}, \ov{f} \circ \eta_{X,A}\right]
 &=
 \left[
 \ov{Y} \oplus \ov{X},\,
 \ov{g} \circ \eta_{Y,B} \circ
 \pars*{\id_{\ov{Y}} \oplus \pars*{\ov{f} \circ \eta_{X,A}}}
 \right] \, .
\end{align*}
On the other hand,
\begin{align*}
 \mathbf{H}_{k}\pars*{
 [Y,g] \circ [X,f]}
 &=
 \mathbf{H}_{k}\pars*{
 [Y \oplus^{\wr} X,\, g \circ (\id_{Y} \oplus^{\wr} f)]
 }
 \\
 &= \co_{k}\pars*{
 \Aut_{\K}\!
 \left[
 \ov{Y \oplus^{\wr} X},\, \ov{g \circ (\id_{Y} \oplus^{\wr} f)}
 \circ \eta_{Y \oplus^{\wr} X, A}
 \right]
 } \, .
\end{align*}
Consider the composite
\[
 F\pars*{\ov{Y \oplus^{\wr} X}} \xrightarrow{\mu_{Y \oplus^{\wr} X}} Y \oplus^{\wr} X
 \xrightarrow{\mu_{Y}^{-1} \oplus^{\wr} \mu_{X}^{-1}}
 F\pars*{\ov{Y}} \oplus^{\wr} F\pars*{\ov{X}}
 \xrightarrow{J_{\ov{Y},\ov{X}}} F\pars*{\ov{Y} \oplus \ov{X}}
\]
in $\GG^{\wr}$. Since $F$ is full, there exists $\rho \colon \ov{Y \oplus^{\wr} X} \to \ov{Y} \oplus \ov{X}$ in $\GG$ such that
\[
 F(\rho) = J_{\ov{Y},\ov{X}} \circ \pars*{\mu_{Y}^{-1} \oplus^{\wr} \mu_{X}^{-1}}
  \circ \mu_{Y \oplus^{\wr} X} \, .
\]
We would like to compare the composite
\[
 \ov{Y \oplus^{\wr} X} \oplus \ov{A} \xrightarrow{\rho \oplus \id_{\ov{A}}}
 \ov{Y} \oplus \ov{X} \oplus \ov{A} 
 \xrightarrow{\ov{g}\, \circ\, \eta_{Y,B} \,\circ\,
 \pars*{\id_{\ov{Y}} \oplus \pars*{\ov{f} \circ \eta_{X,A}}}}
 \ov{C}
\]
with
\[
 \ov{Y \oplus^{\wr} X} \oplus \ov{A} 
 \xrightarrow{\ov{g \circ (\id_{Y} \oplus^{\wr} f)}
 \,\circ\, \eta_{Y \oplus^{\wr} X, A}} \ov{C} \, ,
\]
both morphisms in $\GG$ for which we shortly write $u$ and $v$, respectively. Taking the image of $u$ under $F$, we get
\begin{align*}
 F(u) &= F\pars*{\ov{g} \circ \eta_{Y,B} \circ
 \pars*{\id_{\ov{Y}} \oplus \pars*{\ov{f} \circ \eta_{X,A}}}
 \circ \pars*{\rho \oplus \id_{\ov{A}}}
 }
 \\
 &= F\pars*{\ov{g}} \circ F\pars*{\eta_{Y,B}} \circ 
 F\pars*{\id_{\ov{Y}} \oplus \pars*{\ov{f} \circ \eta_{X,A}}} \circ
 F\pars*{\rho \oplus \id_{\ov{A}}}
 \\
 &= F\pars*{\ov{g}} \circ \mu_{Y \oplus^{\wr} B}^{-1} \circ 
 \pars*{\mu_{Y} \oplus^{\wr} \mu_{B}} \circ J_{\ov{Y},\ov{B}}^{-1} \circ 
 F\pars*{\id_{\ov{Y}} \oplus \pars*{\ov{f} \circ \eta_{X,A}}} \circ
 F\pars*{\rho \oplus \id_{\ov{A}}} 
\end{align*}
The naturality of $J$ yields a commutative diagram
\[
 \xymatrixcolsep{3.4cm}
 \xymatrixrowsep{1.4cm} 
 \xymatrix{
 F\pars*{\ov{Y}} \oplus^{\wr} F\pars*{\ov{X} \oplus \ov{A}} 
 \ar[r]^-{F\pars*{\id_{\ov{Y}}} \oplus^{\wr} F\pars*{\ov{f} \circ \eta_{X,A}}} 
 \ar[d]_-{J_{\ov{Y},\ov{X} \oplus \ov{A}}} &
 F\pars*{\ov{Y}} \oplus^{\wr} F\pars*{\ov{B}}
 \ar[d]^-{J_{\ov{Y},\ov{B}}} 
 \\
 F\pars*{\ov{Y} \oplus \ov{X} \oplus \ov{A}} 
 \ar[r]^-{F\pars*{\id_{\ov{Y}} \oplus \pars*{\ov{f} \circ \eta_{X,A}}}} &
 F\pars*{\ov{Y} \oplus \ov{B}}\, , } 
\] 
hence 
\begin{align*}
 F(u) &= F\pars*{\ov{g}} \circ F\pars*{\eta_{Y,B}} \circ 
 F\pars*{\id_{\ov{Y}} \oplus \pars*{\ov{f} \circ \eta_{X,A}}} \circ
 F\pars*{\rho \oplus \id_{\ov{A}}}
 \\
 &= F\pars*{\ov{g}} \circ \mu_{Y \oplus^{\wr} B}^{-1} \circ 
 \pars*{\mu_{Y} \oplus^{\wr} \mu_{B}} \circ 
 \pars*{F\pars*{\id_{\ov{Y}}} \oplus^{\wr} F\pars*{\ov{f} \circ \eta_{X,A}}} \circ 
 J_{\ov{Y},\ov{X} \oplus \ov{A}}^{-1} \circ
 F\pars*{\rho \oplus \id_{\ov{A}}} 
 \\
 &= F\pars*{\ov{g}} \circ \mu_{Y \oplus^{\wr} B}^{-1} \circ 
 \pars*{\mu_{Y} \oplus^{\wr} \pars*{\mu_{B} \circ F\pars*{\ov{f} \circ \eta_{X,A}}}} 
 \circ 
 J_{\ov{Y},\ov{X} \oplus \ov{A}}^{-1} \circ
 F\pars*{\rho \oplus \id_{\ov{A}}} \, .
\end{align*}
Here
\begin{align*}
 \mu_{B} \circ F\pars*{\ov{f} \circ \eta_{X,A}}
 &= \mu_{B} \circ F\pars*{\ov{f}}
 \circ
 \mu_{X \oplus^{\wr} A}^{-1} \circ 
 \pars*{\mu_{X} \oplus^{\wr} \mu_{A}} \circ J_{\ov{X},\ov{A}}^{-1} \, .
\end{align*}
Since
\[
 \xymatrixcolsep{2.1cm}
 \xymatrixrowsep{1.4cm} 
 \xymatrix{
 F\pars*{\ov{X \oplus^{\wr} A}} \ar[r]^-{F(\ov{f})} \ar[d]_-{\mu_{X \oplus^{\wr} A}} &
 F(\ov{B}) \ar[d]^-{\mu_{B}}
 \\
 X \oplus^{\wr} A \ar[r]_-{f} & B 
 }
 \text{ \, and \, }
 \xymatrix{
 F\pars*{\ov{Y \oplus^{\wr} B}} \ar[r]^-{F(\ov{g})} \ar[d]_-{\mu_{Y \oplus^{\wr} B}} &
 F(\ov{C}) \ar[d]^-{\mu_{C}}
 \\
 Y \oplus^{\wr} B \ar[r]_-{g} & C 
 } 
\]
commutes in $\GG^{\wr}$ by construction, we have
\begin{align*}
 \mu_{B} \circ F\pars*{\ov{f} \circ \eta_{X,A}}
 &= f \circ 
 \pars*{\mu_{X} \oplus^{\wr} \mu_{A}} \circ J_{\ov{X},\ov{A}}^{-1} \, ,
\end{align*}
and hence
\begin{align*}
 F(u) &= \mu_{C}^{-1} \circ g \circ 
 \pars*{\mu_{Y} \oplus^{\wr} \pars*{f \circ 
 \pars*{\mu_{X} \oplus^{\wr} \mu_{A}} \circ J_{\ov{X},\ov{A}}^{-1}}} 
 \circ 
 J_{\ov{Y},\ov{X} \oplus \ov{A}}^{-1} \circ
 F\pars*{\rho \oplus \id_{\ov{A}}} \, .
\end{align*}
Pre-composing the above with $J_{\ov{Y \oplus^{\wr} X},\ov{A}}$ and invoking the commutative diagram
\[
 \xymatrixcolsep{3.4cm}
 \xymatrixrowsep{1.4cm} 
 \xymatrix{
 F\pars*{\ov{Y \oplus^{\wr} X}} \oplus^{\wr} F\pars*{\ov{A}} 
 \ar[r]^-{F\pars*{\rho} \oplus^{\wr} F\pars*{\id_{\ov{A}}}} 
 \ar[d]_-{J_{\ov{Y \oplus^{\wr} X},\ov{A}}} &
 F\pars*{\ov{Y} \oplus \ov{X}} \oplus^{\wr} F\pars*{\ov{A}}
 \ar[d]^-{J_{\ov{Y} \oplus \ov{X},\ov{A}}} 
 \\
 F\pars*{\ov{Y \oplus^{\wr} X} \oplus \ov{A}} 
 \ar[r]^-{F\pars*{\rho \oplus \id_{\ov{A}}}} &
 F\pars*{\ov{Y} \oplus \ov{X} \oplus \ov{A}}\, , } 
\] 
we get 
\begin{align*}
 &\,\,\,\,F(u) \circ J_{\ov{Y \oplus^{\wr} X},\ov{A}} 
 \\
 &= \mu_{C}^{-1} \circ g \circ 
 \pars*{\mu_{Y} \oplus^{\wr} \pars*{f \circ 
 \pars*{\mu_{X} \oplus^{\wr} \mu_{A}} \circ J_{\ov{X},\ov{A}}^{-1}}} 
 \circ 
 J_{\ov{Y},\ov{X} \oplus \ov{A}}^{-1} \circ
 J_{\ov{Y} \oplus \ov{X},\ov{A}} \circ 
 \pars*{F\pars*{\rho} \oplus^{\wr} F\pars*{\id_{\ov{A}}}}
 \, .
\end{align*}
Since
 \[
 \xymatrixcolsep{2.4cm}
 \xymatrixrowsep{1.4cm} 
 \xymatrix{
 F\pars*{\ov{Y}} \oplus^{\wr} F\pars*{\ov{X}} \oplus^{\wr} F\pars*{\ov{A}} 
 \ar[r]^-{J_{\ov{Y},\ov{X}} \,\oplus^{\wr}\, \id_{F\pars*{\ov{A}}}} 
 \ar[d]_-{\id_{F\pars*{\ov{Y}}} \oplus^{\wr} J_{\ov{X},\ov{A}}} 
 &
 F\pars*{\ov{Y} \oplus \ov{X}} \oplus^{\wr} F\pars*{\ov{A}} 
 \ar[d]^-{J_{\ov{Y} \oplus \ov{X}, \ov{A}}}  
 \\
 F\pars*{\ov{Y}} \oplus^{\wr} F\pars*{\ov{X} \oplus \ov{A}}
 \ar[r]^-{J_{\ov{Y},\ov{X} \oplus \ov{A}}}  
 & F\pars*{\ov{Y} \oplus \ov{X} \oplus \ov{A}}
 }
 \]
commutes in $\GG^{\wr}$ by the defining properties of $(F, J,\psi)$ in Definition \ref{monoidal-F}, we have
\begin{align*}
 &\,\,\,\,\,J_{\ov{Y},\ov{X} \oplus \ov{A}}^{-1} \circ
 J_{\ov{Y} \oplus \ov{X},\ov{A}} \circ 
 \pars*{F\pars*{\rho} \oplus^{\wr} F\pars*{\id_{\ov{A}}}}
 \\
 &=\pars*{\id_{F\pars*{\ov{Y}}} \oplus^{\wr} J_{\ov{X},\ov{A}}} 
 \circ \pars*{J_{\ov{Y},\ov{X}} \oplus^{\wr} \id_{F\pars*{\ov{A}}}}^{-1}
 \circ
 \pars*{F\pars*{\rho} \oplus^{\wr} F\pars*{\id_{\ov{A}}}}
 \\
 &=\pars*{\id_{F\pars*{\ov{Y}}} \oplus^{\wr} J_{\ov{X},\ov{A}}} 
 \circ \pars*{\pars*{J_{\ov{Y},\ov{X}}^{-1} \circ F(\rho)} \oplus^{\wr} \id_{F\pars*{\ov{A}}}}
 \\
 &=\pars*{\id_{F\pars*{\ov{Y}}} \oplus^{\wr} J_{\ov{X},\ov{A}}} 
 \circ \pars*{\pars*{\pars*{\mu_{Y}^{-1} \oplus^{\wr} \mu_{X}^{-1}}
  \circ \mu_{Y \oplus^{\wr} X}} \oplus^{\wr} \id_{F\pars*{\ov{A}}}} \, .
\end{align*}
As a result,
\begin{align*}
 &\,\,\,\,F(u) \circ J_{\ov{Y \oplus^{\wr} X},\ov{A}} 
 \\
 &= \mu_{C}^{-1} \circ g \circ 
 \pars*{\mu_{Y} \oplus^{\wr} \pars*{f \circ 
 \pars*{\mu_{X} \oplus^{\wr} \mu_{A}}}} 
 \circ 
 \pars*{\pars*{\pars*{\mu_{Y}^{-1} \oplus^{\wr} \mu_{X}^{-1}}
  \circ \mu_{Y \oplus^{\wr} X}} \oplus^{\wr} \id_{F\pars*{\ov{A}}}}
 \\
 &= \mu_{C}^{-1} \circ g \circ 
 \pars*{\id_{Y} \oplus^{\wr} f}
 \circ
 \pars*{\mu_{Y} \oplus^{\wr} \mu_{X} \oplus^{\wr} \mu_{A}} 
 \circ 
 \pars*{\pars*{\pars*{\mu_{Y}^{-1} \oplus^{\wr} \mu_{X}^{-1}}
  \circ \mu_{Y \oplus^{\wr} X}} \oplus^{\wr} \id_{F\pars*{\ov{A}}}}
 \\
 &= \mu_{C}^{-1} \circ g \circ 
 \pars*{\id_{Y} \oplus^{\wr} f}
 \circ
 \pars*{\mu_{Y \oplus^{\wr} X} \oplus^{\wr} \mu_{A}} 
  \, .   
\end{align*}
Taking the image of $v$ under $F$, we get
\begin{align*}
 F(v) &= F\pars*{\ov{g \circ (\id_{Y} \oplus^{\wr} f)}
 \circ \eta_{Y \oplus^{\wr} X, A}}
 \\
 &=
 F\pars*{\ov{g \circ (\id_{Y} \oplus^{\wr} f)}}
 \circ F\pars*{\eta_{Y \oplus^{\wr} X, A}}
 \\
 &=
 F\pars*{\ov{g \circ (\id_{Y} \oplus^{\wr} f)}}
 \circ 
 \mu_{Y \oplus^{\wr} X \oplus^{\wr} A}^{-1} \circ 
 \pars*{\mu_{Y \oplus^{\wr} X} \oplus^{\wr} \mu_{A}} \circ 
 J_{\ov{Y \oplus^{\wr} X},\ov{A}}^{-1}  
\end{align*}
Since
\[
 \xymatrixcolsep{3.1cm}
 \xymatrixrowsep{1.4cm} 
 \xymatrix{
 F\pars*{\ov{Y \oplus^{\wr} X \oplus^{\wr} A }} \ar[r]^-{F\pars*{\,\ov{g \circ (\id_{Y} \oplus^{\wr} f)}\,}} 
 \ar[d]_-{\mu_{Y \oplus^{\wr} X \oplus^{\wr} A}} &
 F(\ov{C}) \ar[d]^-{\mu_{C}}
 \\
 Y \oplus^{\wr} X \oplus^{\wr} A \ar[r]_-{g \circ (\id_{Y} \oplus^{\wr} f)} & C
 }
\]
commutes in $\GG^{\wr}$ by construction, we have
\begin{align*}
 F(v) \circ J_{\ov{Y \oplus^{\wr} X},\ov{A}} 
 &=
 \mu_{C}^{-1}
 \circ 
 \pars*{g \circ (\id_{Y} \oplus^{\wr} f)}\circ 
 \pars*{\mu_{Y \oplus^{\wr} X} \oplus^{\wr} \mu_{A}}
 \\
 &= F(u) \circ J_{\ov{Y \oplus^{\wr} X},\ov{A}} 
\end{align*}
and hence $F(u) = F(v)$, so $v \circ u^{-1} \colon \ov{C} \to \ov{C}$ actually lies in $\K$. Going back to the definitions of $u,v$, we have
\begin{align*}
 \mathbf{H}_{k}[Y,g] \circ \mathbf{H}_{k}[X,f] 
 &=
 \co_{k}\pars*{
 \Aut_{\K}\!
 \left[
 \ov{Y} \oplus \ov{X},\,
 \pars*{\rho^{-1} \oplus \id_{\ov{A}}} \circ u
 \right]
 }
 \\
 &=
 \co_{k}\pars*{
 \Aut_{\K}\!
 \left[
 \ov{Y \oplus^{\wr} X},u
 \right]
 }
\end{align*}
and hence by the functoriality of $\co_{k}$ and $\Aut_{\K}$ we have
\begin{align*}
 \co_{k}\pars*{\Aut_{\K}[0,v \circ u^{-1}]} \circ 
 \mathbf{H}_{k}[Y,g] \circ \mathbf{H}_{k}[X,f] 
 &=
 \co_{k}\pars*{
 \Aut_{\K}\!
 \left[
 \ov{Y \oplus^{\wr} X},v
 \right]
 }
 \\
 &= \mathbf{H}_{k}\pars*{
 [Y,g] \circ [X,f]} \, .
\end{align*}
As the group homomorphism
\[
 \Aut_{\K}[0,v \circ u^{-1}] \colon \Aut_{\K}\pars*{\ov{C}} \to \Aut_{\K}\pars*{\ov{C}}
\]
is by definition conjugation by $v \circ u^{-1} \in \Aut_{\K}\pars*{\ov{C}}$, we have
\[
 \co_{k}\pars*{\Aut_{\K}[0,v \circ u^{-1}]} = \id_{\co_{k}\pars*{\Aut_{\K}\pars*{\ov{C}}}}
\]
by \cite[Corollary II.6.2]{brown-coh-book}, hence
\[ 
 \mathbf{H}_{k}[Y,g] \circ \mathbf{H}_{k}[X,f] 
 = \mathbf{H}_{k}\pars*{
 [Y,g] \circ [X,f]} \, .
\]
as desired.

Moving on to the natural isomorphism claim, since here we need to deal with objects directly coming from $\GG$, we shall use uncapitalized letters for them. Given $a \in \Obj(U\GG) = \Obj(\GG)$, we need an isomorphism
\[ 
 \co_{k}\pars*{\Aut_{\K}\pars*{\ov{F(a)}}} = \mathbf{H}_{k}\pars*{F(a)}  
 \to
 \co_{k}\pars*{\Aut_{\K}(a)}
\] 
of $\zz$-modules. Note that we have already have an isomorphism
\[
 \mu_{F(a)} \colon F\pars*{\ov{F(a)}} \to F(a)
\]
in $\GG^{\wr}$. As $F$ is full, we may (and do) pick an isomorphism $\nu_{a} \colon \ov{F(a)} \to a$ in $\GG$ with $F(\nu_{a}) = \mu_{F(a)}$. From here we get an isomorphism
\begin{align*}
 \theta_{a} \colon \Aut_{\K}\pars*{\ov{F(a)}} &\to \Aut_{\K}\pars*{a}
 \\
 \kappa &\mapsto \nu_{a} \circ \kappa \circ \nu_{a}^{-1}
\end{align*}
of groups. We shall show that the isomorphisms
\[ 
 \co_{k}\pars*{\theta_{a}} \colon \co_{k}\pars*{\Aut_{\K}\pars*{\ov{F(a)}}} = \mathbf{H}_{k}\pars*{F(a)}  
 \to
 \co_{k}\pars*{\Aut_{\K}(a)}
\] 
of $\zz$-modules are natural in $a \in \Obj(U\GG)$. To that end, fix $[x,\vphi] \colon a \to b$ in $U\GG$ and noting that
\[
 U_{\mathbf{F}}[x,\vphi] = \left[F(x),\,F(\vphi) \circ J_{x,a}\right] \colon F(a) \to F(b)
\]
in $U\GG^{\wr}$, consider the square
\[
 \xymatrixcolsep{5.1cm}
 \xymatrixrowsep{1.4cm} 
 \xymatrix{
 \Aut_{\K}\pars*{\ov{F(a)}} \ar[r]^-{\Aut_{\K}\!\left[\,\ov{F(x)},\, 
 \ov{F(\vphi) \circ J_{x,a}}\,\circ \,\eta_{F(x),F(a)}\right]} 
 \ar[d]_-{\theta_{a}} &
 \Aut_{\K}\pars*{\ov{F(b)}} \ar[d]^-{\theta_{b}}
 \\
 \Aut_{\K}(a) \ar[r]_-{\Aut_{\K}[x,\vphi]} & \Aut_{\K}(b)
 }
\]
of groups. Let us write
\begin{align*}
 p &\coloneqq \Aut_{\K}[x,\vphi] \circ \theta_{a}\,,
 \\
 q &\coloneqq 
 \theta_{b} \circ \Aut_{\K}\!\left[\,\ov{F(x)},\, 
 \ov{F(\vphi) \circ J_{x,a}}\circ \eta_{F(x),F(a)}\right] \, .
\end{align*}
For every $\kappa \in \Aut_{\K}\pars*{\ov{F(a)}}$, on one hand we have
\begin{align*}
 p(\kappa)
 &=
 \Aut_{\K}[x,\vphi]\pars*{\nu_{a} \circ \kappa \circ \nu_{a}^{-1}}
 \\
 &=
 \vphi \circ \pars*{\id_{x} \oplus \pars*{\nu_{a} \circ \kappa \circ \nu_{a}^{-1}}} \circ \vphi^{-1} 
 \\
 &=
 \vphi \circ \pars*{
 \pars*{\nu_{x} \circ \id_{\ov{F(x)}} \circ \nu_{x}^{-1}} \oplus \pars*{\nu_{a} 
 \circ \kappa \circ \nu_{a}^{-1}}
 } 
 \circ \vphi^{-1}  
 \\
 &=
 \vphi \circ \pars*{\nu_{x} \oplus \nu_{a}} \circ 
 \pars*{\id_{\ov{F(x)}} \oplus \kappa} \circ 
 \pars*{\nu_{x} \oplus \nu_{a}}^{-1} \circ
 \vphi^{-1}  
\end{align*}
and on the other hand we have
\begin{align*}
 q (\kappa)
 &= \theta_{b}\pars*{\ov{F(\vphi) \circ J_{x,a}}\circ \eta_{F(x),F(a)}
 \circ \pars*{\id_{\ov{F(x)}} \oplus \kappa} \circ
 \eta_{F(x),F(a)}^{-1} \circ \ov{F(\vphi) \circ J_{x,a}}^{-1}
 }
 \\
 &= \nu_{b} \circ \ov{F(\vphi) \circ J_{x,a}}\circ \eta_{F(x),F(a)}
 \circ \pars*{\id_{\ov{F(x)}} \oplus \kappa} \circ
 \eta_{F(x),F(a)}^{-1} \circ \ov{F(\vphi) \circ J_{x,a}}^{-1}
 \circ \nu_{b}^{-1} \, .
\end{align*}
We shall show that the composite
\[
 \xymatrixcolsep{2.1cm}
 \xymatrixrowsep{1.4cm} 
 \xymatrix{
 \ov{F(x)} \oplus \ov{F(a)} \ar[r]^-{\eta_{F(x),F(a)}}  
 & \ov{F(x) \oplus^{\wr} F(a)} \ar[r]^-{\ov{F(\vphi) \circ J_{x,a}}} 
 & \ov{F(b)} \ar[d]^{\nu_{b}}
 \\ 
 \ov{F(x)} \oplus \ov{F(a)} 
 & x \oplus a \ar[l]_-{\pars*{\nu_{x} \oplus \nu_{a}}^{-1}} 
 & b \ar[l]_-{\vphi^{-1}} 
 }
\] 
 in $\GG$, that we shall denote by $\xi \in \Aut_{\GG}\pars*{\ov{F(x)} \oplus \ov{F(a)}}$, actually lies in $\K = \ker F$.
Here
\begin{align*}
 F\pars*{\ov{F(\vphi) \circ J_{x,a}}\circ \eta_{F(x),F(a)}}
 &=
 F\pars*{\ov{F(\vphi) \circ J_{x,a}}} \circ
 \mu_{F(x) \oplus^{\wr} F(a)}^{-1} \circ 
 \pars*{\mu_{F(x)} \oplus^{\wr} \mu_{F(a)}} \circ J_{\ov{F(x)},\ov{F(a)}}^{-1}
 \\
 &=
 F\pars*{\ov{F(\vphi) \circ J_{x,a}}} \circ
 \mu_{F(x) \oplus^{\wr} F(a)}^{-1} \circ 
 \pars*{F(\nu_{x}) \oplus^{\wr} F(\nu_{a})} \circ J_{\ov{F(x)},\ov{F(a)}}^{-1}
\end{align*}
Since
\[
 \xymatrixcolsep{3.1cm}
 \xymatrixrowsep{1.4cm} 
 \xymatrix{
 F\pars*{\ov{F(x) \oplus^{\wr} F(a)}} \ar[r]^-{F\pars*{\,\ov{F(\vphi) \circ J_{x,a}}\,}} 
 \ar[d]_-{\mu_{F(x) \oplus^{\wr} F(a)}} &
 F(\ov{F(b)}) \ar[d]^-{\mu_{F(b)} = F(\nu_{b})}
 \\
 F(x) \oplus^{\wr} F(a) \ar[r]_-{F(\vphi) \circ J_{x,a}} & F(b) 
 }
\]
commutes by construction, we have
\begin{align*}
 F\pars*{\ov{F(\vphi) \circ J_{x,a}}\circ \eta_{F(x),F(a)}}
 &=
 F(\nu_{b})^{-1} \circ
 F(\vphi) \circ J_{x,a} \circ 
 \pars*{F(\nu_{x}) \oplus^{\wr} F(\nu_{a})} \circ J_{\ov{F(x)},\ov{F(a)}}^{-1} \, .
\end{align*}
The naturality of $J$ yields a commutative diagram
\[
 \xymatrixcolsep{3.1cm}
 \xymatrixrowsep{1.4cm} 
 \xymatrix{
 F\pars*{\ov{F(x)}} \oplus^{\wr} F\pars*{\ov{F(a)}} 
 \ar[r]^-{F(\nu_{x}) \oplus^{\wr} F(\nu_{a})} 
 \ar[d]_-{J_{\ov{F(x)},\ov{F(a)}}} &
 F\pars*{x} \oplus^{\wr} F\pars*{a}
 \ar[d]^-{J_{x,a}} 
 \\
 F\pars*{\ov{F(x)} \oplus \ov{F(a)}} \ar[r]^-{F\pars*{\nu_{x} \oplus \nu_{a}}} &
 F\pars*{x \oplus a}\, , } 
\] 
hence
\begin{align*}
 F\pars*{\nu_{b} \circ \ov{F(\vphi) \circ J_{x,a}}\circ \eta_{F(x),F(a)}}
 &=
 F(\nu_{b}) \circ F\pars*{\ov{F(\vphi) \circ J_{x,a}}\circ \eta_{F(x),F(a)}}
 \\
 &=
 F(\vphi) \circ F\pars*{\nu_{x} \oplus \nu_{a}}
 \\
 &= F\pars*{\vphi \circ \pars*{\nu_{x} \oplus \nu_{a}}} \, ,
\end{align*}
meaning
\begin{align*}
 F(\xi) &= F\pars*{
 \pars*{\vphi \circ \pars*{\nu_{x} \oplus \nu_{a}}}^{-1}
 \circ
 \pars*{\nu_{b} \circ \ov{F(\vphi) \circ J_{x,a}}\circ \eta_{F(x),F(a)}}
 }
 \\
 &= \id_{F\pars*{\,\ov{F(x)} \oplus \ov{F(a)}\,}}
\end{align*}
as desired: $\xi \in \Aut_{\K}\pars*{\ov{F(x)} \oplus \ov{F(a)}}$. It follows that the composite
\[
 \xymatrixcolsep{1.8cm}
 \xymatrixrowsep{1.4cm} 
 \xymatrix{
 b \ar[r]^-{\vphi^{-1}}  
 & x \oplus a \ar[r]^-{\pars*{\nu_{x} \oplus \nu_{a}}^{-1}} 
 & \ov{F(x)} \oplus \ov{F(a)}  \ar[d]^{\xi}
 \\ 
 b 
 & x \oplus a \ar[l]_-{\vphi} 
 & \ov{F(x)} \oplus \ov{F(a)} \ar[l]_-{\nu_{x} \oplus \nu_{a}} 
 }
\] 
 in $\GG$, that we shall denote by $\eps$ similarly satisfies $\eps \in \Aut_{\K}(b)$. Writing
\begin{align*}
  c_{\eps} \colon \Aut_{\K}(b) &\to \Aut_{\K}(b)
  \\
  h &\mapsto \eps \circ h \circ \eps^{-1}
\end{align*}
for the conjugation map, we have for every $\kappa \in \Aut_{\K}\pars*{\ov{F(a)}}$ the equality
\begin{align*}
 \pars*{c_{\eps} \circ p}(\kappa) &= c_{\eps}\pars*{
 \vphi \circ \pars*{\nu_{x} \oplus \nu_{a}} \circ 
 \pars*{\id_{\ov{F(x)}} \oplus \kappa} \circ 
 \pars*{\nu_{x} \oplus \nu_{a}}^{-1} \circ
 \vphi^{-1}  
 }
 \\
 &= \vphi \circ \pars*{\nu_{x} \oplus \nu_{a}} \circ \xi \circ
 \pars*{\id_{\ov{F(x)}} \oplus \kappa} \circ \xi^{-1} \circ 
 \pars*{\nu_{x} \oplus \nu_{a}}^{-1} \circ \vphi^{-1}
 \\
 &= q(\kappa)
\end{align*}
and hence $p =c_{\eps}^{-1} \circ q$ with
\begin{align*}
 \pars*{\co_{k} \circ \Aut_{\K}}[x,\vphi] \circ \co_{k}(\theta_{a})
 &=
 \co_{k}(p)
 \\
 &= \co_{k}\pars*{c_{\eps}^{-1}} \circ \co_{k}(q)
 \\
 &= \co_{k}(q)
 \\
 &=
 \co_{k}\pars*{\theta_{b}} \circ \co_{k}\pars*{\Aut_{\K}\!\left[\,\ov{F(x)},\, 
 \ov{F(\vphi) \circ J_{x,a}}\circ \eta_{F(x),F(a)}\right]}
 \\
 &=
 \co_{k}\pars*{\theta_{b}} \circ \mathbf{H}_{k}\!\left[
 F(x),\, F(\vphi) \circ J_{x,a}
 \right]
 \\
 &=
 \co_{k}\pars*{\theta_{b}} \circ 
 \pars*{ \mathbf{H}_{k} \circ
 U_{\mathbf{F}}}[x,\vphi] \, ,
\end{align*}
where the third equality above is by \cite[Corollary II.6.2]{brown-coh-book}. In other words, the diagram
\[
 \xymatrixcolsep{4.1cm}
 \xymatrixrowsep{1.4cm} 
 \xymatrix{
 \pars*{ \mathbf{H}_{k} \circ U_{\mathbf{F}}}(a) 
 \ar[r]^-{\pars*{ \mathbf{H}_{k} \circ U_{\mathbf{F}}}[x,\vphi]} 
 \ar[d]_-{\co_{k}\pars*{\theta_{a}}} &
 \pars*{ \mathbf{H}_{k} \circ U_{\mathbf{F}}}(b) \ar[d]^-{\co_{k}\pars*{\theta_{b}}}
 \\
 \pars*{\co_{k} \circ \Aut_{\K}}(a) \ar[r]_-{\pars*{\co_{k} \circ \Aut_{\K}}[x,\vphi]} 
 & \pars*{\co_{k} \circ \Aut_{\K}}(b)
 }
\]
 commutes for every  $[x,\vphi] \colon a \to b$ in $U\GG$, exhibiting a natural isomorphism
\[
 \co_{k}(\theta) \colon \mathbf{H}_{k} \circ U_{\mathbf{F}} 
 \xrightarrow{\sim} \co_{k} \circ \Aut_{\K}
\]
of functors $U\GG \to \lMod{\zz}$.

Thanks to Lemma \ref{cat-epi}, the uniqueness claim about $\mathbf{H}_{k}$ will follow once we show that $U_{\mathbf{F}} \colon U\GG \to U\GG^{\wr}$ is essentially surjective and full. For the latter, suppose $X \in \Obj(U\GG^{\wr}) = \Obj(\GG^{\wr})$; we have already exhibited an object $\ov{A}\in \Obj(\GG)$ and an isomorphism
\[
 \mu_{A} \colon F(\ov{A}) \to A
\]
in $\GG$, therefore $[0,\mu_{A}] \colon F(\ov{A}) \to A$ is an isomorphism in $U\GG^{\wr}$. To see $U_{\mathbf{F}}$ is full, suppose $a,b \in \Obj(\GG)$ and
\[
 [X,f] \colon F(a) \to F(b)
\]
is a morphism in $U\GG^{\wr}$, that is $f \colon X \oplus^{\wr} F(a) \to F(b)$ is an (iso)morphism in $\GG^{\wr}$. Writing $x \coloneqq \ov{X} \in \Obj(\GG)$, we have a composite morphism
\[
 F(x \oplus a)
 \xrightarrow{J_{x,a}^{-1}}
 F(x) \oplus^{\wr} F(a) \xrightarrow{\mu_{X} \oplus^{\wr} \id_{F(a)}} X \oplus^{\wr} F(a)
 \xrightarrow{f} F(b)
\]
in $\GG^{\wr}$. As $F \colon \GG \to \GG^{\wr}$ is assumed full, there exists a morphism
\(
 \vphi \colon x \oplus a \to b
\)
such that
\[
 F(\vphi) = f \circ \pars*{\mu_{X} \oplus^{\wr} \id_{F(a)}} \circ J_{x,a}^{-1} \, .
\]
The morphism $U_{\mathbf{F}}[x,\vphi] \colon F(a) \to F(b)$ in $U\GG^{\wr}$ is given as
\begin{align*}
  U_{\mathbf{F}}[x,\vphi] &= \left[F(x),\,F(\vphi) \circ J_{x,a}\right] 
  \\
  &= \left[F(x),\,f \circ \pars*{\mu_{X} \oplus^{\wr} \id_{F(a)}}\right]= [X,f] \, ,
\end{align*}
showing $U_{\mathbf{F}}$ is full.
\end{proof}

\begin{proof}[Proof of \textbf{\emph{Theorem \ref{essek}}}]
 (1): Let us start with a morphism $(f,m,n) \colon m \to n$ in $\atsi_{\GG(c,x)}$. Recall that this means 
 \[
  f \in \GG(c,x)_{n} = \Aut_{\GG}\pars*{c \oplus x^{\oplus n}} \, ,
 \]
so that
\[ T_{\GG,c,x}(f,m,n) = \left[
 x^{\oplus n-m}, f
 \right] \colon c \oplus x^{\oplus m} \to c \oplus x^{\oplus n}
\]
in $U\GG^{\circ}$. Applying $U_{\pmb{\tau}}$ above, we get the morphism
\[
 U_{\pmb{\tau}}\!\left[
 x^{\oplus n-m}, f
 \right] = \left[
 x^{\oplus n-m},\, f \circ \tau_{x^{\oplus n-m}, \, c \oplus x^{\oplus m}}
 \right] \colon c \oplus x^{\oplus m} \to c \oplus x^{\oplus n}
\]
in $U\GG$. Finally, applying $\Aut_{\K}$ yields the group homomorphism
\begin{align*}
 \Aut_{\K}\pars*{c \oplus x^{\oplus m}} &\to \Aut_{\K}\pars*{c \oplus x^{\oplus n}}
 \\
 \kappa &\mapsto \pars*{f \circ \tau_{x^{\oplus n-m}, \, c \oplus x^{\oplus m}}} 
 \circ \pars*{\id_{x^{\oplus n-m}} \oplus \kappa} \circ 
 \pars*{f \circ \tau_{x^{\oplus n-m}, \, c \oplus x^{\oplus m}}}^{-1} \, .  
\end{align*}
The final long expression above can be rewritten as
\begin{align*}
 &\,\,\,\,\,
 \pars*{f \circ \tau_{x^{\oplus n-m}, \, c \oplus x^{\oplus m}}} 
 \circ \pars*{\id_{x^{\oplus n-m}} \oplus \kappa} \circ 
 \pars*{f \circ \tau_{x^{\oplus n-m}, \, c \oplus x^{\oplus m}}}^{-1}
 \\
 &=
 f \circ \tau_{x^{\oplus n-m}, \, c \oplus x^{\oplus m}} 
 \circ \pars*{\id_{x^{\oplus n-m}} \oplus \kappa} \circ 
 \tau^{-1}_{x^{\oplus n-m}, \, c \oplus x^{\oplus m}} \circ f^{-1}
 \\
 &=
 f \circ \pars*{\kappa \oplus \id_{x ^{\oplus n-m}}} \circ f^{-1}
 \\
 &= f \circ \GG(c,x)_{(m,n)}(\kappa) \circ f^{-1} \, .
\end{align*}
It remains to verify that the group homomorphism
\begin{align*}
 \Aut_{\K}\pars*{c \oplus x^{\oplus m}} &\to \Aut_{\K}\pars*{c \oplus x^{\oplus n}}
 \\
 \kappa &\mapsto f \circ \GG(c,x)_{(m,n)}(\kappa) \circ f^{-1}
\end{align*}
coincides with
\[
 \K(c,x)_{(f,m,n)} \colon \K(c,x)_{m} \to \K(c,x)_{n} \, .
\]
To that end, we compute 
\begin{align*}
 \K(c,x)_{n} &= \left\{u \in \GG(c,x)_{n} : \mathbf{F}[\gamma,\xi]_{n}(u) 
 = \id_{\GG^{\wr}(c^{\wr},x^{\wr})_{n}}
 \right\} 
 \\
 &= \left\{u \in \Aut_{\GG}\pars*{c \oplus x^{\oplus n}} : \mathbf{F}[\gamma,\xi]_{n}(u) 
 = \id_{c^{\wr} \oplus^{\wr} {x^{\wr}}^{\oplus^{\wr}n}}
 \right\} 
 \\
 &= 
\begin{cases}
 \left\{u \in \Aut_{\GG}(c) : \gamma \circ F(u) \circ \gamma^{-1} 
 = \id_{c^{\wr}}
 \right\} & \text{if $n=0$,}
 \vspace{0.1cm}
 \\
 \left\{
 u \in \Aut_{\GG}\pars*{c \oplus x^{\oplus n}} : \mathbf{F}\!\left[
 (\gamma \oplus^{\wr} \xi) \circ J_{c,x}^{-1},\,\xi
 \right]_{n-1}\!(u) = \id_{c^{\wr} \oplus^{\wr} {x^{\wr}}^{\oplus^{\wr}n}}
 \right\} & \text{if $n \geq 1$.}
\end{cases}
 \\
 &= 
\begin{cases}
 \left\{u \in \Aut_{\GG}(c) : F(u) = \id_{F(c)}
 \right\} & \text{if $n=0$,}
 \vspace{0.1cm}
 \\
 \left\{
 u \in \Aut_{\GG}\pars*{c \oplus x^{\oplus n}} : \mathbf{F}\!\left[
 (\gamma \oplus^{\wr} \xi) \circ J_{c,x}^{-1},\,\xi
 \right]_{n-1}\!(u) = \id_{c^{\wr} \oplus^{\wr} {x^{\wr}}^{\oplus^{\wr}n}}
 \right\} & \text{if $n \geq 1$.}
\end{cases}
 \\
 &= 
\begin{cases}
 \Aut_{\K}(c) & \text{if $n=0$,}
 \vspace{0.1cm}
 \\
 \left\{
 u \in \Aut_{\GG}\pars*{c \oplus x^{\oplus n}} : \mathbf{F}\!\left[
 (\gamma \oplus^{\wr} \xi) \circ J_{c,x}^{-1},\,\xi
 \right]_{n-1}\!(u) = \id_{c^{\wr} \oplus^{\wr} {x^{\wr}}^{\oplus^{\wr}n}}
 \right\} & \text{if $n \geq 1$.}
\end{cases}
\\
&= \Aut_{\K}\pars*{c \oplus x^{\oplus n}}
\end{align*}
on objects, and the description of the action on morphisms matches exactly that of Proposition \ref{conj-action}.

(2): We shall first show that 
\[
 \mathbf{F}[\gamma,\xi]_{n} \colon \GG(c,x)_{n} \to \GG^{\wr}\pars*{c^{\wr},x^{\wr}}_{n}
\]
(for every choice of $\gamma$, $\xi$) is surjective for every $n \in \zz_{\geq 0}$ by induction. If $n=0$, by definition (see Remark \ref{Gcx-maps}) we are looking at the map
\begin{align*}
 \mathbf{F}[\gamma,\xi]_{0} \colon \Aut_{\GG}(c) &\to \Aut_{\GG^{\wr}}(c^{\wr})
 \\
 f &\mapsto \gamma \circ F(f) \circ \gamma^{-1} \, ,
\end{align*}
which is surjective because $F$ is full and $\gamma$ is an isomorphism. For $n \geq 1$, the map
\[
 \mathbf{F}[\gamma,\xi]_{n} = \mathbf{F}\!\left[
 (\gamma \oplus^{\wr} \xi) \circ J_{c,x}^{-1},\,\xi
 \right]_{n-1}
\]
is surjective via the induction hypothesis. We can therefore apply Proposition \ref{factor-AQ} to the map $\mathbf{F}[\gamma,\xi] \colon \GG(c,x) \to \GG^{\wr}\pars*{c^{\wr},x^{\wr}}$ of $\atsi$-groups to conclude that there is a unique functor
\[
 \co_{k}\pars*{\K(c,x)} \colon \atsi_{\GG^{\wr}(c^{\wr},x^{\wr})} \to \lMod{\zz}
\]
that makes the diagram
\begin{align*}
 \xymatrixcolsep{1.5cm}
 \xymatrix{
 \atsi_{\GG(c,x)} \ar[r]^-{\K(c,x)} \ar[d]_-{\atsi_{\mathbf{F}[\gamma,\xi]}} & \Grp \ar[r]^-{\co_{k}} & \lMod{\zz}
 \\
 \atsi_{\GG^{\wr}(c^{\wr},x^{\wr})} \ar@{-->}_-{\,\,\,\,\co_{k}\pars*{\K(c,x)}}[urr]
 }
\end{align*}
commute. It follows that in the triangular prism (\ref{prism}), the large bottom (slanted) rectangle commutes (uniquely with $\co_{k}\pars*{\K(c,x)}$). 

The commutativity of the back triangle of (\ref{prism}) is (1). 

Next, we shall show that the
\[
 \xymatrixcolsep{2.1cm}
 \xymatrixrowsep{1.4cm} 
 \xymatrix{
 \atsi_{\GG(c,x)} \ar[r]^-{T_{\GG,c,x}} \ar[d]_-{\atsi_{\mathbf{F}[\gamma,\xi]}}
 & U\GG^{\circ} \ar[d]^-{U_{\mathbf{F}^{\circ}}}
 \\
 \atsi_{\GG^{\wr}(c^{\wr},x^{\wr})} \ar[r]^-{T_{\GG^{\wr},c^{\wr},x^{\wr}}}  & U{\GG^{\wr}}^{\circ}
 }
\] 
part of the top face of (\ref{prism}) commutes up to a natural isomorphism: this amounts to finding isomorphisms 
\[\lambda[\gamma,\xi]_{n} \colon F(c \oplus x^{\oplus n}) \to c^{\wr} \oplus^{\wr} {x^{\wr}}^{\oplus^{\wr} n}\] 
in $U{\GG^{\wr}}^{\circ}$ for every $n \in \zz_{\geq 0}$ such that
\[
 \xymatrixcolsep{2.5cm}
 \xymatrixrowsep{1.5cm} 
 \xymatrix{
 \Hom_{\atsi_{\GG(c,x)}}(m,n) \ar[r]^-{T_{\GG,c,x}} \ar[d]_-{\atsi_{\mathbf{F}[\gamma,\xi]}}
 & \Hom_{U\GG^{\circ}}\pars*{c \oplus x^{\oplus m},c \oplus x^{\oplus n}} 
 \ar[d]^-{U_{\mathbf{F}^{\circ}}}
 \\
 \Hom_{\atsi_{\GG^{\wr}(c^{\wr},x^{\wr})}}(m,n) \ar[dr]_-{T_{\GG^{\wr},c^{\wr},x^{\wr}}}  
 & \Hom_{U{\GG^{\wr}}^{\circ}}\pars*{F\pars*{c \oplus x^{\oplus m}},
 F\pars*{c \oplus x^{\oplus n}}} \ar[d]^{\lambda[\gamma,\xi]_{n} \circ \,-\, \circ \lambda[\gamma,\xi]_{m}^{-1}}
 \\
 & \Hom_{U{\GG^{\wr}}^{\circ}}\pars*{c^{\wr} \oplus^{\wr} 
 {x^{\wr}}^{\oplus^{\wr} m},c^{\wr} \oplus^{\wr} {x^{\wr}}^{\oplus^{\wr} n}}
 } \tag{$\dagger$} \label{tepe1}
\] 
commutes whenever $0\leq m \leq n$. We define $\lambda[\gamma,\xi]_{n}$ inductively as
\[
 \lambda[\gamma,\xi]_{n} \coloneqq 
\begin{cases}
 \left[0^{\wr},\gamma\right] & \text{if $n=0$,}
 \\
 \lambda\!\left[\pars*{\gamma \oplus^{\wr} \xi} \circ J_{c,x}^{-1},\,\xi\right]_{n-1} & 
 \text{if $n \geq 1$.}
\end{cases}
\]
To see (\ref{tepe1}) commutes, let us start with a morphism $(f,m,n) \colon m \to n$ in $\atsi_{\GG(c,x)}$. Recall that this means 
 \[
  f \in \GG(c,x)_{n} = \Aut_{\GG}\pars*{c \oplus x^{\oplus n}} \, ,
 \]
so that
\[ T_{\GG,c,x}(f,m,n) = \left[
 x^{\oplus n-m}, f
 \right] \colon c \oplus x^{\oplus m} \to c \oplus x^{\oplus n}
\]
in $U\GG^{\circ}$. According to Proposition \ref{UF-oldu}, applying $U_{\mathbf{F}^{\circ}}$ above yields
\begin{align*}
 U_{\mathbf{F}^{\circ}}\!\left[
 x^{\oplus n-m}, f
 \right] &= U_{F,J^{\circ},\psi}\!\left[
 x^{\oplus n-m}, f
 \right]
 \\
 &= \left[
 F\pars*{x^{\oplus n-m}},\,F(f) \circ J^{\circ}_{x^{\oplus n-m},c \oplus x^{\oplus m}}
 \right]
 \\
 &= \left[
 F\pars*{x^{\oplus n-m}},\,F(f) \circ J_{c \oplus x^{\oplus m}, x^{\oplus n-m}}
 \right] \, .
\end{align*}
Consequently carrying $(f,m,n)$ right-down-down in (\ref{tepe1}) results in
\[
 \lambda[\gamma,\xi]_{n} \circ
 \left[
F\pars*{x^{\oplus n-m}},\,F(f) \circ J_{c \oplus x^{\oplus m}, x^{\oplus n-m}}
 \right] \circ \lambda[\gamma,\xi]_{m}^{-1} \, .
\]
On the other hand, the morphism
\[
 \atsi_{\mathbf{F}[\gamma,\xi]}(f,m,n) = \pars*{\mathbf{F}[\gamma,\xi]_{n}(f),m,n} \colon
 m \to n
\]
in $\atsi_{\GG^{\wr}(c^{\wr},x^{\wr})}$ satisfies
\[
 T_{\GG^{\wr},c^{\wr},x^{\wr}}\pars*{\mathbf{F}[\gamma,\xi]_{n}(f),m,n} =
 \left[
 {x^{\wr}}^{\oplus^{\wr} n-m},\,\mathbf{F}[\gamma,\xi]_{n}(f)
 \right]
\]
which is the result of carrying $(f,m,n)$ down-downright in (\ref{tepe1}). We shall show equality of the two paths by induction on $m$. When $m=0$ the desired equality for every $f \in \Aut_{\GG}(c \oplus x^{\oplus n})$ is
\begin{align*}
 \lambda[\gamma,\xi]_{n} \circ
 \left[
F\pars*{x^{\oplus n}},\,F(f) \circ J_{c, x^{\oplus n}}
 \right] \circ \left[0^{\wr},\gamma\right]^{-1} 
 &=
 \left[
 {x^{\wr}}^{\oplus^{\wr} n},\,\mathbf{F}[\gamma,\xi]_{n}(f)
 \right] \, .
\end{align*}
This equality indeed holds because by Definition \ref{monoidal-F} the isomorphism $\psi \colon F(0) \to 0^{\wr}$ in $\GG^{\wr}$ satisfies
\begin{align*}
 \pars*{\gamma \circ F(f)} \circ \pars*{\psi {\oplus^{\wr}}^{\circ} \id_{F(c)}} 
 &=
 \gamma \circ F(f) \circ \pars*{\id_{F(c)} \oplus^{\wr}\, \psi}
 \\
 &=  
 \gamma \circ F(f) \circ J_{c,0} \, .
\end{align*}
When $n=m \geq 1$, the desired equality for every $f \in \Aut_{\GG}(c)$ is
\begin{align*}
 \lambda[\gamma,\xi]_{n} \circ
 \left[
F\pars*{0},\,F(f) \circ J_{c \oplus x^{\oplus n}, 0}
 \right] \circ \lambda[\gamma,\xi]_{n}^{-1}
 &= \left[
 0^{\wr},\,\mathbf{F}[\gamma,\xi]_{n}(f)
 \right]
 \\
 \lambda[\gamma,\xi]_{n} \circ
 \left[
F\pars*{0},\,F(f) \circ J_{c \oplus x^{\oplus n}, 0}
 \right] 
 &= \left[
 0^{\wr},\,\mathbf{F}[\gamma,\xi]_{n}(f)
 \right] \circ \lambda[\gamma,\xi]_{n} \, ,
\end{align*}
equivalently the equality of
\begin{align*}
 &\,\,\,\,\,
 \lambda\!\left[\pars*{\gamma \oplus^{\wr} \xi} \circ J_{c,x}^{-1},\,\xi\right]_{n-1} \circ
 \left[
F\pars*{0},\,F(f) \circ J_{c \oplus x^{\oplus n}, 0}
 \right] 
 \\
 &=
 \lambda\!\left[\pars*{\gamma \oplus^{\wr} \xi} \circ J_{c,x}^{-1},\,\xi\right]_{n-1} \circ
 \left[
F\pars*{0},\,F(f) \circ J_{(c \oplus x) \oplus x^{\oplus n-1}, 0}
 \right]  
\end{align*}
and
\begin{align*}
 \left[
 0^{\wr},\,\mathbf{F}\!\left[
 (\gamma \oplus^{\wr} \xi) \circ J_{c,x}^{-1},\,\xi
 \right]_{n-1}\!(f)
 \right] \circ \lambda\!\left[\pars*{\gamma \oplus^{\wr} \xi} \circ J_{c,x}^{-1},\,\xi\right]_{n-1} \, ,
\end{align*}
which follows from the inductive hypothesis. When $n > m > 0$, the desired equality for every $f \in \Aut_{\GG}(c \oplus x^{\oplus n})$ is
\[
 \lambda[\gamma,\xi]_{n} \circ
 \left[
F\pars*{x^{\oplus n-m}},\,F(f) \circ J_{c \oplus x^{\oplus m}, x^{\oplus n-m}}
 \right] 
=
 \left[
 {x^{\wr}}^{\oplus^{\wr} n-m},\,\mathbf{F}[\gamma,\xi]_{n}(f)
 \right] \circ \lambda[\gamma,\xi]_{m} \, ,
\]
equivalently the equality of
\begin{align*}
  &\,\,\,\,\,
 \lambda\!\left[\pars*{\gamma \oplus^{\wr} \xi} \circ J_{c,x}^{-1},\,\xi\right]_{n-1}
 \circ
 \left[
F\pars*{x^{\oplus n-m}},\,F(f) \circ J_{c \oplus x^{\oplus m}, x^{\oplus n-m}}
 \right] 
 \\
 &=
 \lambda\!\left[\pars*{\gamma \oplus^{\wr} \xi} \circ J_{c,x}^{-1},\,\xi\right]_{n-1}
 \circ
 \left[
F\pars*{x^{\oplus (n-1)-(m-1)}},\,F(f) \circ J_{(c \oplus x) \oplus x^{\oplus m-1}, x^{\oplus (n-1)-(m-1)}}
 \right] 
\end{align*}
and
\begin{align*}
 &\,\,\,\,\,\left[
 {x^{\wr}}^{\oplus^{\wr} n-m},\,\mathbf{F}\!\left[
 (\gamma \oplus^{\wr} \xi) \circ J_{c,x}^{-1},\,\xi
 \right]_{n-1}
 \right] \circ 
 \lambda\!\left[\pars*{\gamma \oplus^{\wr} \xi} \circ J_{c,x}^{-1},\,\xi\right]_{m-1}
 \\
 &=
 \left[
 {x^{\wr}}^{\oplus^{\wr} (n-1)-(m-1)},\,\mathbf{F}\!\left[
 (\gamma \oplus^{\wr} \xi) \circ J_{c,x}^{-1},\,\xi
 \right]_{n-1}
 \right] \circ 
 \lambda\!\left[\pars*{\gamma \oplus^{\wr} \xi} \circ J_{c,x}^{-1},\,\xi\right]_{m-1} \, ,
\end{align*}
which again follows from the inductive hypothesis.

Next, we shall show the remaining rectangle on the top face of the triangular prism (\ref{prism}), namely
\[
 \xymatrixcolsep{2.1cm}
 \xymatrixrowsep{1.4cm} 
 \xymatrix{
 U\GG^{\circ} \ar[r]^{U_{\pmb{\tau}}} \ar[d]_-{U_{\mathbf{F}^{\circ}}} &
 U\GG \ar[d]^-{U_{\mathbf{F}}}
 \\
 U{\GG^{\wr}}^{\circ} \ar[r]_{U_{\pmb{\tau}^{\wr}}}&
 U{\GG^{\wr}} 
 }
\] 
commutes. This follows because the above diagram is the image of
\[
 \xymatrix{
 \GG^{\circ} \ar[r]^{\pmb{\tau}} \ar[d]_-{\mathbf{F}^{\circ}} &
 \GG \ar[d]^-{\mathbf{F}}
 \\
 {\GG^{\wr}}^{\circ} \ar[r]_{\pmb{\tau}^{\wr}}&
 \GG^{\wr} 
 }
\] 
under $U \colon \bmgpd \to \moncat$, which is commutative by Proposition \ref{braid-op-iso}.

The existence and uniqueness of $\mathbf{H}_{k}$ that makes the side rectangular face of the triangular prism (\ref{prism})
\[
 \xymatrixcolsep{2.1cm}
 \xymatrixrowsep{1.4cm} 
 \xymatrix{
 U\GG \ar[r]^-{U_{\mathbf{F}}} \ar[d]_{\Aut_{\K}} &  U{\GG^{\wr}} \ar[d]^{\mathbf{H}_{k}}
 \\
 \Grp \ar[r]_{\co_{k}} & \lMod{\zz}
 }
\] 
 up to natural isomorphism, is Proposition \ref{conj-action-Hk}. 
 
It remains to show the commutativity of the front triangle of (\ref{prism}), namely
\[
 \xymatrixcolsep{2.1cm}
 \xymatrixrowsep{1.4cm} 
 \xymatrix{
  \atsi_{\GG^{\wr}(c^{\wr},x^{\wr})} \ar[r]^-{T_{\GG^{\wr},c^{\wr},x^{\wr}}} 
 \ar[drr]_-{\co_{k}\pars*{\K(c,x)}\,\,\,\,\,\,} & 
 U{\GG^{\wr}}^{\circ} 
 \ar[r]^-{U_{\pmb{\tau}^{\wr}}} & U\GG^{\wr} \ar[d]^-{\mathbf{H}_{k}}
 \\
  & & \lMod{\zz}  
 }
\] 
Fortunately we can see this without laborious computations. It follows from what we have already established that
\begin{align*}
 \co_{k}\pars*{\K(c,x)} \circ \atsi_{\mathbf{F}[\gamma,\xi]} 
 &\cong \co_{k} \circ \,\K(c,x) 
 \\
 &\cong \co_{k} \circ \Aut_{\K} \circ\, U_{\pmb{\tau}} \circ T_{\GG,c,x}
 \\
 &\cong \mathbf{H}_{k} \circ U_{\mathbf{F}} \circ U_{\pmb{\tau}} \circ T_{\GG,c,x}
 \\
 &\cong \mathbf{H}_{k} \circ U_{\pmb{\tau}^{\wr}} 
 \circ U_{\mathbf{F}^{\circ}} \circ T_{\GG,c,x} 
 \\
 &\cong \mathbf{H}_{k} \circ U_{\pmb{\tau}^{\wr}} 
 \circ T_{\GG^{\wr},c^{\wr},x^{\wr}} \circ \atsi_{\mathbf{F}[\gamma,\xi]} \, .
\end{align*}
Thanks to Lemma \ref{cat-epi}, we will be done once we establish that the functor 
\[
  \atsi_{\mathbf{F}[\gamma,\xi]} \colon \atsi_{\GG(c,x)} \to \atsi_{\GG^{\wr}(c^{\wr},x^{\wr})}
\] 
is essentially surjective and full. The map it induces on the objects is simply the identity
\[ 
 \id_{\zz_{\geq 0}} \colon \zz_{\geq 0} \to \zz_{\geq 0} \, ,
\]
so it remains to show fullness. We are to show that for every morphism $\pars*{g,m,n} \colon m \to n$ in $\atsi_{\GG^{\wr}(c^{\wr},x^{\wr})}$ there exists a morphism $\pars*{f,m,n} \colon m \to n$ in $\atsi_{\GG(c,x)}$ such that
\begin{align*}
 (g,m,n) &= \atsi_{\mathbf{F}[\gamma,\xi]}(f,m,n)
 \\
 &= \pars*{\mathbf{F}[\gamma,\xi]_{n}(f),m,n} \, ,
\end{align*}
so it suffices to see that the group homomorphism
\[
 \mathbf{F}[\gamma,\xi]_{n} \colon \GG(c,x)_{n} \to \GG^{\wr}\pars*{c^{\wr},x^{\wr}}_{n}
\]
is surjective for every $n \in \zz_{\geq 0}$, which we have already shown in the beginning of the proof of part (2) above.
\end{proof} 
\end{appendices}

\bibliographystyle{hamsalpha}
\bibliography{stable-boy}

\end{document}